\documentclass[12pt]{article}

\usepackage[letterpaper, margin=1in]{geometry}

\usepackage[format=plain,
            labelfont={bf},
            textfont=it]{caption}
\usepackage{wrapfig}
\usepackage{amsfonts} 
\usepackage{amsthm}
\usepackage[shortlabels]{enumitem}
\usepackage{framed}
\renewenvironment{proof}{{\bf \emph{Proof.} }}{\hfill $\Box$ \\} 
\usepackage{soul}
\usepackage{amsmath}
\usepackage{amssymb}
\usepackage{mathrsfs}
\usepackage{yhmath}
\usepackage{subcaption}
\usepackage{comment}
\usepackage{graphicx}

\usepackage{float}
\usepackage{hyperref}
\usepackage[utf8]{inputenc}
\usepackage[sorting=none]{biblatex}
\usepackage[dvipsnames]{xcolor}

\usepackage[many]{tcolorbox}
\newtcolorbox{myAwesomeBox}{title = \textcolor{black}, 
float,
floatplacement=htbp,
    colbacktitle=Cerulean!10,
    enhanced,
    code={\singlespacing},  
    colback=Cerulean!10,  
    colframe=Cerulean!10  
}

\newcommand{\RomanNumeralCaps}[1]
    {\MakeUppercase{\romannumeral #1}}

\DeclareMathOperator{\Tr}{Tr} 

\DeclareMathOperator{\dd}{\mathrm{d}} 

\DeclareMathOperator{\Law}{Law}
\DeclareMathOperator{\Vol}{Vol}
\DeclareMathOperator{\dist}{dist}
\DeclareMathOperator{\diam}{diam}

\newcommand{\bld}[1]{{\boldsymbol{#1}}}

\newcommand{\mcal}[1]{{\mathcal{#1}}}

\newcommand{\R}{\mathbb{R}}

\makeatletter
\renewcommand*\env@matrix[1][*\c@MaxMatrixCols c]{%
  \hskip -\arraycolsep
  \let\@ifnextchar\new@ifnextchar
  \array{#1}}
\makeatother

\usepackage{algorithm}
\usepackage[noend]{algpseudocode}
\usepackage{setspace}

\makeatletter
\def\BState{\State\hskip-\ALG@thistlm}
\makeatother

\usepackage{dsfont}
\usepackage[T1]{fontenc}
\usepackage{ebgaramond}

\usepackage{setspace}
\usepackage{etoc}

\newtheorem{proposition}{Proposition}
\newtheorem{assumption}{Assumption}
\newtheorem{theorem}{Theorem}
\newtheorem*{C1}{Corollary~\ref{corquerybudget}}

\newtheorem{lemma}{Lemma}
\newtheorem{problem}{Problem}
\newtheorem{corollary}{Corollary}
\newtheorem{definition}{Definition}
\newtheorem{remark}{Remark}

\usepackage{tikz}
\usetikzlibrary{arrows,positioning,shapes,calc}

\usepackage{authblk}

\title{Efficient Approximation of Structured Linear Operators using Randomized One-Sided Queries}
\author[1]{Esther Gallmeier}
\author[2]{Samuel E. Otto}
\affil[1]{Center for Applied Mathematics, Cornell University, Ithaca, NY 14853, USA}
\affil[2]{Sibley School of Mechanical and Aerospace Engineering, Cornell University, Ithaca, NY 14853, USA}
\date{}
\begin{document}
\onecolumn
\maketitle

\allowdisplaybreaks

\begin{abstract}
Approximating linear operators on high-dimensional spaces using only forward operator-vector product queries is a common but challenging task in scientific computing and machine learning arising when access to the adjoint action is experimentally unavailable or difficult to construct computationally.
Existing query-efficient methods with polylogarithmic query complexity in the error tolerance rely on self-adjointness or access to the adjoint operator, while existing adjoint-free methods based on global low-order regularity have polynomial query complexity.
We show that a priori local row-space information can be exploited for query-efficient, adjoint-free operator approximation, and that this information can be supplied by local high-order regularity estimates.
Our approach first extends the fixed-sparsity matrix approximation algorithm of (Amsel et al., SIMAX, 2026) to the task of approximating a matrix by one whose rows lie in prescribed subspaces, providing a general template for exploiting local row-space structure from forward queries alone.
We then apply this framework to Hilbert-Schmidt integral operators with asymptotically smooth kernels.
We establish asymptotic smoothness for Greens functions of canonical second-order elliptic PDEs in $d\leq 3$ dimensions including Poisson's equation in balls and constant coefficient equations on Tori.
In $d$ spatial dimensions, our algorithm uses $\mcal{O}\big(\log(\varepsilon^{-1})\log^{d}(\varepsilon^{-1}\delta^{-1/2})\big)$ non-adaptive randomized queries to produce an $\varepsilon$-accurate approximation in Hilbert-Schmidt norm with probability $1-\delta$.
This implies the error decreases as $\mcal{O}\big( \exp(- c q^{1/(d+1)}) \big)$ with respect to the number of queries $q$, which we observe in numerical experiments using tens to hundreds of queries to approximate solution operators of one-dimensional variable-coefficient second-order elliptic equations.
\end{abstract}

\tableofcontents

\newpage
\section{Introduction}

Many problems in scientific computing and machine learning seek to approximate operators from input-output observations between infinite-dimensional spaces.
These approximations serve as inexpensive surrogate models for large-scale inverse problems, design optimization, and solver acceleration \cite{azizzadenesheli2024neural}.
In this work, we consider a non-self-adjoint Hilbert-Schmidt integral operator $\mcal{A}: L^2(\Omega) \rightarrow L^2(\Omega)$ and assume access only to forward queries $f \mapsto \mcal{A}f$. 
This restriction arises naturally when observations come from experiments, passively observed systems, proprietary codes, complex or legacy solvers, where prescribing inputs and measuring outputs may be possible but constructing or querying an adjoint is not. 
The guiding question of this work is: what prior knowledge makes forward-only access sufficient to construct a provably accurate operator approximation using as few queries as possible. 

The finite-dimensional analogue of this problem for a matrix $\bld{A} \in \R^{m \times n}$ is of widespread interest in numerical linear algebra and reveals the fundamental limitations of forward-only access. 
Recent work on transpose-free linear algebra studies what can and cannot be inferred from  matrix-vector products $x \mapsto \bld{A}x$ when the transpose action $v \mapsto \bld{A}^T v$ is not available \cite{halikias2026transpose}.  
For several basic tasks, including least squares, maximum volume and aspects of spectral estimation, adversarial matrices can be constructed that agree on all observed forward queries but lead to substantially different solutions.
These findings complement recent $\Omega(n)$ lower bounds for computing approximate least squares solutions for a square linear system from forward queries only, even for perfectly conditioned matrices \cite{derezinski2026matrix}.
Similar difficulties arise for low-rank approximation.
It has been shown that exact low-rank matrix recovery is not uniquely determined from fewer than $n$ forward queries \cite{halikias2024structured}.
Related work by \cite{otto2023model} has characterized the set of possible row spaces of low-rank matrices that agree on the same set of forward sketches as submanifolds on the Grassmannian manifold, and another work by \cite{boulle2024operator} characterized the set of possible matrices satisfying identical sketching constraints that one could recover given partial prior knowledge of the symmetry of the desired unknown low-rank matrix. 
Moreover, \cite{halikias2026transpose} shows that $n$ forward queries are necessary even to obtain a low rank approximation within a constant factor of optimal. 
Together, these results demonstrate that many core linear algebra tasks are fundamentally limited when only given access to forward queries in the absence of additional prior information.
At the same time, these are worst-case analyses and do not rule out recovery or approximation using fewer queries when additional structure can be leveraged. 

A growing line of work studies matrix-vector query complexity of structured matrix approximation. 
Halikias and Townsend developed randomized algorithms for recovering matrices with known structures, including tridiagonal, Toeplitz-like and hierarchical low-rank matrices from general matrix-vector products\cite{halikias2024structured}. 
More recent work gives near-optimal guarantees for hierarchically off-diagonal low-rank approximation, although the algorithm requires products with both $\bld{A}$ and $\bld{A}^T$ \cite{chen2025near}.
Without attempting an exhaustive review here, related studies consider other structured matrix families, including low-rank, diagonal, sparse and butterfly approximations \cite{halko2011finding,rokhlin2010randomized,musco2015randomized,bakshi2023krylov, wimalajeewa2013recovery, dasarathy2015sketching, liu2021butterfly}. 
More broadly, recent results establish upper and lower bounds for approximating matrices from general structured matrix families using one- and two-sided queries \cite{amsel2025query}.
A particularly relevant forward-only result arises in fixed-sparsity approximation. If the desired approximation is known in advance to have a fixed sparsity pattern, then each row lies in a known coordinate subspace. The work in \cite{amsel2026fixed} shows that, when each row has at most $k$ nonzero entries at known locations, a randomized nonadaptive algorithm  returns a matrix with the prescribed sparsity pattern whose Frobenius error is within $(1 + \varepsilon)$ factor of the best such approximation using $\mcal{O}(k/\varepsilon)$ forward queries. 
Our first contribution generalizes this algorithm from a fixed-sparsity viewpoint to the setting where each row is approximated in a prescribed $k$-dimensional subspace.
We leverage that the essential ingredient for forward-only approximation is not coordinate sparsity itself, but the existence of low-dimensional row spaces known a priori.
Moreover, this finite-dimensional perspective provides the algorithmic template for constructing approximations of Hilbert-Schmidt integral operators. 
After discretizing the operator with locally supported basis functions, each row of the resulting matrix is determined by a localized portion of the integral kernel, and forward queries can approximate the discretized operator if suitable row spaces are supplied in advance by properties of the kernel.

Generally, neural operator architectures provide flexible frameworks for approximating maps between function spaces, including Fourier Neural Operators, Graph Neural Operators, PCA-Nets, DeepONets, and related models,  \cite{li2020fourier, li2020neural, li2020multipole, lanthaler2023operator, lu2021learning, kovachki2023neural, kovachki2024operator}. 
Their appeal lies in their broad applicability, as they require relatively little explicit information about the governing partial differential equation (PDE) or physical process generating the data. 
However, universal approximation theory is primarily qualitative. 
It shows that such architectures are expressive, yet it does not quantify the model size or number of training samples needed to achieve a prescribed tolerance. 
Recent quantitative results indicate that efficiency heavily depends on additional operator structure. 
For broad classes such as Lipschitz or Frechet-differentiable operators, standard neural operator architectures can suffer from a curse of parametric complexity, with worst-case model sizes growing exponentially in the inverse accuracy \cite{kovachki2024data,kovachki2021universal,lanthaler2021error}. 
More favorable results are known for structured classes, such as holomorphic operators or specific PDE solution maps, where algebraic rates in the inverse tolerance can be proven for parametric or sample complexity.
In particular, \cite{reinhardt2026statistical} establish statistical convergence guarantees for infinite-dimensional operator regression and algebraic sample-complexity rates for holomorphic operators represented by encoder-decoder neural operators.
These guarantees are often formulated either uniformly over compact input sets or in Bochner $L^2(\mu)$-type norms with respect to a prescribed input distribution. 

In contrast to neural operator architectures, structure-exploiting approaches have recently led to remarkable results for approximating the solution operators of linear elliptic PDEs. 
By taking advantage of the hierarchical structure of elliptic solution operators, the algorithms in \cite{boulle2023elliptic, schafer2024sparse} obtain exponential-type convergence rates with worst-case operator-norm  or Hilbert-Schmidt norm guarantees. 
These results demonstrate that, when the operator class is sufficiently well understood, one can go far beyond generic statistical learning rates.
In particular, we will describe an algorithm as \textit{query-efficient} if the error decays at a stretched exponential rate $\mcal{O}(e^{-c q^p})$ with respect to the number of queries $q$ for some $c > 0$ and $0 < p < 1$.  
However, these existing approaches face serious limitations in the non-self-adjoint setting. 
The method in \cite{schafer2024sparse} relies on sparse Cholesky factorization and is therefore naturally suited only to positive-definite self-adjoint operators. 
A proposed extension to non-self-adjoint operators by \cite{schafer2024sparse} seeks to recover the LU factorization of the operator from matrix-vector and matrix-transpose-vector products, requiring access to the adjoint operator. 
Similarly, the method of \cite{boulle2023elliptic} is based on an infinite-dimensional randomized SVD framework and can be extended formally to non-self-adjoint operators, but this also requires querying both the operator and its adjoint. 

Linear operator learning from one-sided input--output data has been studied under a variety of global statistical, spectral, and regularity assumptions.
Convergence rates have been established for self-adjoint operators with prescribed spectral structure \cite{dehoop2023convergence}, for Green's functions learned by functional regression in reproducing kernel Hilbert spaces \cite{stepaniants2023learning}, for elliptic pseudodifferential operators \cite{chen2026convergence}, and for general linear operators between Sobolev spaces \cite{chen2026optimal}.
Closest to the present setting, \cite{boulle2024operator} shows that global information about the range of $A^*$, supplied for elliptic solution operators by Sobolev regularity estimates, enables adjoint-free approximation from forward queries alone. 
These approaches yield algebraic convergence rates in the number of samples or queries under their respective assumptions.

Many operators of interest, however, possess local structure beyond what is captured by global regularity estimates, providing low-dimensional local row spaces enabling data-efficient one-sided approximations.
In this work, we demonstrate provably accurate, query-efficient approximation for a class of Hilbert-Schmidt integral operators whose kernels satisfy an asymptotic smoothness property in an $L^2$ sense with respect to the argument being integrated.
This is a high-order local regularity assumption on the kernel guaranteeing that polynomials exponentially approximate the ranges $\left(\mcal{A}|_{\mcal{X} \times \mcal{Y}}\right)^*$ restricted to \textit{admissible} subdomains $\mcal{X}, \mcal{Y} \subset \Omega$ whose separation is comparable to their diameters.
Related pointwise notions of asymptotic smoothness appear throughout the hierarchical-matrix and wavelet literature \cite{bebendorf2006accelerating,bebendorf2008hierarchical,hackbusch2015hierarchical}.
We introduce an $L^2$-based formulation suited to the Hilbert-Schmidt setting of our analysis and show that it is closely connected to standard interior elliptic regularity estimates. 
We then verify the property for the Green function of the Poisson equation on the unit ball in dimensions $d=1,2,3$, as well as for the Green function of a constant-coefficient linear elliptic operator on the flat torus under standard existence and off-diagonal regularity assumptions. 
Combined with the hierarchical decomposition induced by the admissibility condition, asymptotic smoothness yields low-dimensional subspaces that accurately approximate the rows of the operator after discretization in a locally supported finite element basis. 
We then apply our generalization of the matrix approximation algorithm from \cite{amsel2026fixed} to these local row spaces, obtaining an approximation of the underlying operator with Hilbert-Schmidt norm error $\mcal{O}\left( \exp\big(-c q^{1/(d+1)}\big) \right)$ using $q$ non-adaptive Gaussian random queries. 
The central message of our work is that a priori local row space information can enable query-efficient adjoint-free approximation for linear operators, and this information can be supplied by local high-order regularity. 

\subsection{Our Contributions and Roadmap}

This paper makes two main contributions: (1) a matrix approximation algorithm using non-adaptive randomized matrix-vector products to approximate a given matrix with a structured matrix whose rows live in subspaces known a priori and (2) an  operator approximation algorithm using only randomized forward queries to learn asymptotically smooth integral operators yielding Hilbert-Schmidt error guarantees with a stretched exponential convergence rate in high probability. 
The corresponding algorithms are explicitly given in Algorithms \ref{alg:amselgen} and \ref{adjointfreealg} with their contributions fully summarized in Theorems \ref{thm:amselgen} and \ref{bigtheorem}.

Algorithm~\ref{alg:amselgen} is based on a simple modification of an algorithm for approximating a matrix with a known sparsity pattern \cite{amsel2026fixed} to the case where each row of the constructed matrix lies in a pre-determined subspace of dimension at most $k$. It achieves a near-optimal, i.e. $(1 + \varepsilon)$, approximation to the best matrix approximation with those pre-determined subspaces with probability at least $1- \delta$ using $\mcal{O}(k/(\varepsilon\delta ))$ queries. 
The algorithm queries  the given matrix with standard Gaussian vectors and proceeds to solve row-wise least-squares problems to recover each row's coefficients in its corresponding subspace.

Algorithm~\ref{adjointfreealg} leverages Algorithm~\ref{alg:amselgen} to approximate asymptotically smooth integral operators by projecting the operator onto a finely resolved finite element space. 
It uses $\mcal{O}\left(\log(\varepsilon^{-1})
\log^d(\varepsilon^{-1}\delta^{-1/2})\right)$ queries to the underlying operator --- given by input functions with standard normal coefficients in the finite element basis --- to construct an $\varepsilon$-approximation in the Hilbert-Schmidt norm with probability at least $1 - \delta$. 
A schematic outline of the algorithm is provided in Figure~\ref{roadmap}. 
The algorithm discretizes the operator via an $L^2$-projection onto a finite element space consisting of piecewise constant polynomials, constructs low-dimensional subspaces that accurately approximate each locally supported row of the discretized operator, and then applies Algorithm~\ref{alg:amselgen} to approximate the operator. 
The crux of the argument is that asymptotically smooth kernels admit separable polynomial expansions with exponentially decaying remainders on admissible domains; combined with a hierarchical partition, this ensures that the size of each row subspace is $\mcal{O}(\log^{d+1}( \varepsilon^{-1}))$, where $d$ is the dimension of the spatial domain.

\section{Adjoint-Free Matrix Approximation in Fixed Row Subspaces}
\label{amselgensection}

In this section, we focus on the task of approximating a matrix $\bld{A} \in \R^{n_r \times n_c}$ with a structured matrix with pre-determined subspaces for each row, with error comparable to the best matrix approximation given by the subspaces with respect to the Frobenius norm. 
Let $\mcal{V}_i \subset \R^{n_c}$ denote the pre-determined subspace for each row $i$ and $\bld{P}_i = \bld{V}_i\bld{V}_i^T$ the orthogonal projector onto $\mcal{V}_i$ formed using an orthonormal basis $\bld{V}_i$ of $\mcal{V}_i$ for each $i = 1, \dots, n_r$. Define
\[
    \bld{P} [\bld{A}] :=
    \begin{bmatrix}
        - \ \ \bld{A}_{1,:} \bld{P}_1 \ \ - \\
        \vdots \\
        - \ \ \bld{A}_{n_r,:} \bld{P}_{n_r} \ \ -
    \end{bmatrix},
\]
as the matrix where each row is the orthogonal projection of the corresponding row of the given matrix $\bld{A}$ onto its respective pre-determined subspace, and observe that $\bld{P}[ \bld{A}]$ is the matrix nearest to the given matrix $\bld{A}$ in the Frobenius norm whose rows live in those fixed subspaces: 
\[
    \bld{P} [\bld{A}] 
        = \text{argmin}_{\bld{X} = \bld{P} [\bld{X}]} \| \bld{A} - \bld{X}\|_F. 
\]
Using this setup we generalize the fixed-sparsity approximation problem in \cite{amsel2026fixed} as follows:
\begin{problem}[fixed row-space approximation]
\label{prob:amselgen}
    Given a matrix $\bld{A} \in \R^{n_r\times n_c}$, pre-determined subspaces $\mcal{V}_i \subset \R^{n_c}$ for each row $i$, and a tolerance $\varepsilon > 0$, 
    find a matrix $\bld{\tilde{A}}$ so that $\bld{\tilde{A}} = \bld{P} [\bld{\tilde{A}}]$ and
    \[
        \| \bld{A} - \bld{\tilde{A}}\|_F 
        \leq (1 + \varepsilon) \| \bld{A} - \bld{P}[ \bld{A}]\|_F.
    \]
\end{problem}

Problem \ref{prob:amselgen} seeks a \textit{near-optimal} approximation to $\bld{A}$ where each row lives in the pre-determined subspaces. 
The problem considered by \cite{amsel2026fixed} corresponds to the case where each $\mcal{V}_i$ is spanned by a known set of canonical basis vectors.
To solve the generalized problem, we introduce Algorithm~\ref{alg:amselgen}, which generalizes the fixed-sparsity matrix approximation algorithm in \cite{amsel2026fixed}.
Both algorithms query the forward action of the matrix on $q$ Gaussian random vectors yielding $\bld{Z} = \bld{A} \bld{G}$, where $\bld{G}$ is an $n_c \times q$ matrix with independent standard normal entries.
These queries are subsequently used to solve $n_r$ independent row-wise least-squares problems.
Theorem~\ref{thm:amselgen} proves how well the resulting approximation  $\tilde{\bld{A}}$ recovers the row-wise projection $\bld{P}[\bld{A}]$.   
As in the fixed-sparsity case, when the dimensions of the pre-determined row subspaces are at most a constant $k$, Corollary~\ref{cor:amselsolvesproblem} shows that $\mcal{O}(k/\varepsilon)$ non-adaptive matrix-vector Gaussian queries suffice to solve Problem \ref{prob:amselgen} in expectation.

\begin{algorithm}
\caption{Structured Matrix Approximation with Fixed Row Subspaces} 
\label{alg:amselgen}

\begin{algorithmic}[1]
\Procedure{Rowwise-Matrix-Approximation}{$\bld{x} \mapsto \bld{A}\bld{x}, \bld{V}_1, ..., \bld{V}_{n_r}, q$}

\State $\text{Form} \ \bld{G} \sim \text{Gaussian}(n_c,q)$  \color{gray} \algorithmiccomment{$n_c \times q$ i.i.d. Gaussian matrix}
\color{black}

\State $\text{Compute} \ \bld{Z} = \bld{AG}$ \color{gray} \algorithmiccomment{$q$ non-adaptive matvec queries}

\color{black}
\BState \For {$i = 1,2, ..., n_r$} \do  \\

\State $\text{Project} \ \bld{\Gamma}_i = \bld{G}^T\bld{V}_i$ \color{gray} \algorithmiccomment{  $\bld{V}_i$ is $n_c \times k_i$ fixed row space}

\color{black}
\State $\text{Let} \ \bld{z}_i^T = \bld{Z}_{i,:} $ \color{gray} \algorithmiccomment{$i$-th row of $\bld{Z}$}

\color{black}
\State $\text{Compute} \ \bld{\Gamma}_i = \bld{QR}$ \color{gray} \algorithmiccomment{compute reduced QR}

\color{black}
\State $\text{Form} \ \bld{\tilde{x}}_i =  \bld{R}^{-1} \bld{Q}^T \bld{z}_i$ \color{gray} \algorithmiccomment{solve $q \times k_i$ least squares problem with $\bld{\Gamma}_i^\dagger$}

\color{black}
\State $\text{Set} \  \bld{\tilde{A}}_{i,:} = \bld{\tilde{x}}_i^T \bld{V}_i^T$ \color{gray} \algorithmiccomment{construct $i$-th row of $\bld{\tilde{A}}$}

\EndFor

\color{black}
\Return $\bld{\tilde{A}}$

\EndProcedure
\end{algorithmic}
\end{algorithm}

\begin{theorem}
\label{thm:amselgen}
    Given a matrix $\bld{A} \in \R^{n_r \times n_c}$, a natural number $k$, and subspaces $\mcal{V}_i \subset \R^{n_c}$, $i = 1,\ldots,n_r$, with dimensions $k_i:=\dim \mcal{V}_i \leq k$ and orthonormal bases for each subspace given by the columns of $\bld{V}_i$. 
    Then, using $q \geq k + 2$ randomized matrix-vector queries, Algorithm~\ref{alg:amselgen} returns a matrix 
    $\bld{\tilde{A}} = \text{Alg}(\bld{x} \mapsto \bld{A}\bld{x}, \bld{V}_1, ..., \bld{V}_{n_r}, q)$ 
    satisfying $\bld{\tilde{A}} = \bld{P}[\bld{\tilde{A}}]$ and
    \[
        \mathbb{E}\left[\| \bld{P} [ \bld{A}] - \bld{\tilde{A}}\|_F^2\right] 
        \leq \frac{k}{q-k-1}\| \bld{A} - \bld{P} [\bld{A}]\|_F^2.
    \]
\end{theorem}

\begin{proof}[Theorem~\ref{thm:amselgen}]
We closely follow the proof of Theorem~1.3 in \cite{amsel2026fixed}.
The algorithm processes the rows of $\bld{Z} = \bld{AG} \in \R^{n_r \times q}$ sequentially to construct approximations for the rows of $\bld{A}$. 
Let the columns of $\bld{V}_i$ be an orthonormal basis for the $k_i$-dimensional subspace $\mcal{V}_i \subset \R^{n_c}$, and complete it so that the columns of $\bld{Q}_i = \begin{bmatrix} \bld{V}_i & \bld{W}_i \end{bmatrix}$ are an orthonormal basis of $\R^{n_c}$. 
The $i$-th row of $\bld{A}$ can be expressed as
\[
    \bld{a}_i := \bld{A}_{i,:}^T = \bld{V}_i \bld{x}_i + \bld{W}_i \bld{y}_i
\]
with some coefficients $\bld{x}_i\in\R^{k_i}$ and $\bld{y}_i\in\R^{n_c-k_i}$. 
The $i$-th row of $\bld{AG}$ is given by
\[
    \bld{z}_i^T
    = \bld{a}_i^T \bld{G}
    = \bld{x}_i^T \bld{V}_i^T \bld{G} + \bld{y}_i^T \bld{W}_i^T \bld{G}.
\]
Letting $\bld{\Gamma}_i = \bld{G}^T \bld{V}_i$ and $\bld{\tilde{\Gamma}}_i = \bld{G}^T \bld{W}_i$, 
this becomes
\[
    \bld{z}_i = \bld{\Gamma}_i \bld{x}_i + \bld{\tilde{\Gamma}}_i \bld{y}_i. 
\]
Due to the orthogonal invariance of Gaussian random vectors, the random matrix 
$\begin{bmatrix}
    \bld{\Gamma}_i & \bld{\tilde{\Gamma}}_i
\end{bmatrix} = \bld{G}^T \bld{Q}$ has the same probability law $N(0,1)^{\otimes (q \times n_c)}$ as $\bld{G}^T$, that is, its entries are independent $N(0,1)$ random variables.
The algorithm solves a least squares problem using $\bld{z}_i \in \R^q$ and $\bld{\Gamma}_i\in\R^{q\times k_i}$ to produce an approximation of $\bld{x}_i$ given by
\[
    \bld{\tilde{x}}_i 
    := \bld{\Gamma}_i^\dagger \bld{z}_i 
    = \bld{\Gamma}_i^\dagger (\bld{\Gamma}_i \bld{x}_i + \bld{\tilde{\Gamma}}_i \bld{y}_i) 
    = \bld{x}_i + \bld{\Gamma}_i^\dagger \bld{\tilde{\Gamma}}_i \bld{y}_i,
\]
where $\bld{\Gamma}_i^\dagger$ denotes the Moore-Penrose pseudoinverse of $\bld{\Gamma}_i$.
Note that $\bld{\Gamma}_i \in \R^{q\times k_i}$ has full-column-rank with probability one because $q \geq k_i$.
Since $\bld{\Gamma}_i$ and $\bld{\tilde{\Gamma}}_i$ are independent and $\mathbb{E}[\bld{\tilde{\Gamma}}_i] = 0$, we have $\mathbb{E}[\bld{\tilde{x}}_i] = \bld{x}_i$.
Because $\Law\left(\begin{bmatrix}
    \bld{\Gamma}_i & \bld{\tilde{\Gamma}}_i
\end{bmatrix}\right) = N(0,1)^{\otimes (q \times n_c)}$ and $q \geq k_i + 2$, it follows from Proposition 2 in \cite{amsel2026fixed} that
\begin{align}
    \mathbb{E}\left[\left\| \bld{x}_i - \bld{\tilde{x}}_i\right\|_2^2\right] 
    &= \mathbb{E}\left[\big\| \bld{\Gamma}_i^\dagger \bld{\tilde{\Gamma}}_i \bld{y}_i\big\|_2^2\right]
    = \mathbb{E}\left[\mathbb{E} \left[ \big\| \bld{\Gamma}_i^\dagger \bld{\tilde{\Gamma}}_i \bld{y}_i\big\|_2^2 \ \big| \ \bld{\Gamma}_i\right]\right] \nonumber \\
    &= \mathbb{E}\left[\| \bld{y}_i\|_2^2 \  \Tr\left(\bld{\Gamma}_i^\dagger \bld{\Gamma}_i^{\dagger^T}\right) \right]
    = \mathbb{E} \left[ \| \bld{y}_i\|_2^2 \  \| \bld{\Gamma}_i^\dagger\|_F^2 \right] \nonumber \\
    &= \frac{k_i}{q-k_i-1} \| \bld{y}_i\|_2^2 \leq \frac{k}{q-k-1} \| \bld{y}_i\|_2^2, \nonumber 
\end{align}
where we have used that $k_i \leq k$. 
Therefore, the approximation $\bld{\tilde{a}}_i = \bld{V}_i \bld{\tilde{x}}_i$ of the $i$-th row of $\bld{A}$ satisfies
\[
    \mathbb{E}\big[ \|   \bld{P}_i \bld{a}_i - \bld{\tilde{a}}_i \|_2^2\big] 
    = \mathbb{E} \left[ \| \bld{x}_i - \bld{\tilde{x}}_i\|_2^2\right] 
    \leq \frac{k}{q-k-1} \| \bld{y}_i \|_2^2 
    = \frac{k}{q-k-1} \|  \bld{a}_i - \bld{P}_i \bld{a}_i \|_2^2.
\]
Summing over rows $i=1, \ldots, n_r$ completes the proof.
\end{proof}

Theorem~\ref{thm:amselgen} immediately yields the query complexity required to obtain a near-optimal approximation to $\bld{A}$ specified in Problem~\ref{prob:amselgen}.

\begin{corollary} 
\label{cor:amselsolvesproblem}
Assume the setting of Theorem~\ref{thm:amselgen}. Using $q \geq k + 2$ randomized matrix-vector queries, Algorithm~\ref{alg:amselgen} returns $\bld{\tilde{A}}$ such that
\[
    \mathbb{E}\left[\|  \bld{A} - \bld{\tilde{A}}\|_F^2\right] 
    \leq \left( 1 +  \frac{k}{q-k-1} \right)\| \bld{A} - \bld{P} [\bld{A}]\|_F^2.
\]
In particular using $q \geq \frac{k}{2\varepsilon} + k + 1$ queries, Algorithm~\ref{alg:amselgen} solves Problem \ref{prob:amselgen} in expectation. 
\end{corollary}

\begin{proof}
Because $\bld{P}_i$ is an orthogonal projection and $\bld{\tilde{A}}_{i,:} = \bld{\tilde{a}}_i = \bld{P}_i \bld{\tilde{a}}_i$, we have
\[
    \| \bld{a}_i - \bld{\tilde{a}}_i\|^2_2 = \| \bld{a}_i - \bld{P}_i \bld{a}_i\|_2^2 + \| \bld{P}_i \bld{a}_i - \bld{\tilde{a}}_i\|_2^2.
\] 
Summing over rows gives $\| \bld{A} - \bld{\tilde{A}}\|_F^2  = \| \bld{A} - \bld{P} [\bld{A}]\|_F^2 + \| \bld{P} [\bld{A}] - \bld{\tilde{A}}\|_F^2$.
Taking $q \geq k+2 $ in Theorem~\ref{thm:amselgen} yields
\[
    \mathbb{E}\left[\| \bld{A} - \bld{\tilde{A}}\|_F^2\right] 
    \leq \left( 1 + \frac{k}{q-k-1} \right) \| \bld{P}[\bld{A}] - \bld{A}\|_F^2.
\]
Closely following \cite{amsel2026fixed} and taking $q \geq \frac{k}{2\varepsilon} + k + 1$ queries yields
\[
    \left(\mathbb{E}[\| \bld{A} - \bld{\tilde{A}}\|_F]\right)^2 
    \leq \mathbb{E}[\| \bld{A} - \bld{\tilde{A}}\|_F^2] 
    \leq \left(1 + 2 \varepsilon\right)\| \bld{P}[\bld{A}] - \bld{A}\|_F^2 
    \leq \left( 1 + \varepsilon\right)^2\| \bld{P}[\bld{A}] - \bld{A}\|_F^2
\]
where the first inequality follows from Jensen's inequality and the last inequality follows from the fact that $\sqrt{1 + 2 \varepsilon} < 1 + \varepsilon$ for all $\varepsilon > 0$.
\end{proof}

\begin{remark}
A number of queries $q = \mcal{O}(k)$  can recover the matrix in expectation within a constant factor of the best possible error induced by row-wise projections of $\bld{A}$ onto the pre-determined subspaces. For example, use $q = 2k + 1$ queries in Theorem~\ref{thm:amselgen}, and then $\mathbb{E}[\|  \bld{A} - \bld{\tilde{A}}\|_F^2] \leq 2 \| \bld{A} - \bld{P} [\bld{A}]\|_F^2.$
\label{remarkafteramselgen}
\end{remark}

While enlarging the subspaces $\mcal{V}_i$ can reduce the row residuals $\| \bld{P}_i \bld{a}_i - \bld{a}_i\|_2$, it also increases the dimension bound $k$ and therefore the required number of queries. 
Thus, knowing accurate low-dimensional subspaces for each row leads to query-efficient approximation.
In the sequel we use this principle to obtain query-efficient approximations for a class of linear integral operators whose assumed regularity provides low-dimensional subspaces that are known a priori to well-approximate each horizontal slice, i.e., ``row'', of the kernel.
We note that because the subspaces $\mcal{V}_i$ vary from row-to-row, the approximation need not be low-rank.
Indeed, the class of operators we approximate can have algebraically decaying singular values, yet we leverage exponentially-convergent approximations of its horizontal slices in subspaces that vary slice-to-slice, as we explain below.

\section{Approximating Asymptotically Smooth Integral Operators}

\begin{figure}
    \begin{tcolorbox}[height from= 4cm to \maxdimen, title = (\RomanNumeralCaps{1}) FE Discretization (\textit{section} \ref{section:fediscretization})]%
    \begin{minipage}[c]{0.62\textwidth}
    Discretize the integral operator $\mcal{A}:L^2(\Omega) \rightarrow L^2(\Omega)$ with piecewise constant polynomials $\{\theta_i\}_{i=1}^{|\mcal{T}_h|}$ on a finite element mesh of mesh size $h$ with mesh cells $\{\mcal{X}_i\}_{i=1}^{|\mcal{T}_h|}$. 
    The problem reduces to approximating a matrix $\bld{A}_h \in \R^{|\mcal{T}_h| \times |\mcal{T}_h|}$ where the $i$-th row is determined by the kernel on $\mcal{X}_i \times \Omega$.
  \end{minipage}
  \hfill
   \begin{minipage}[c]{0.33\textwidth}
          \centering
          \includegraphics[width= 0.9\textwidth]{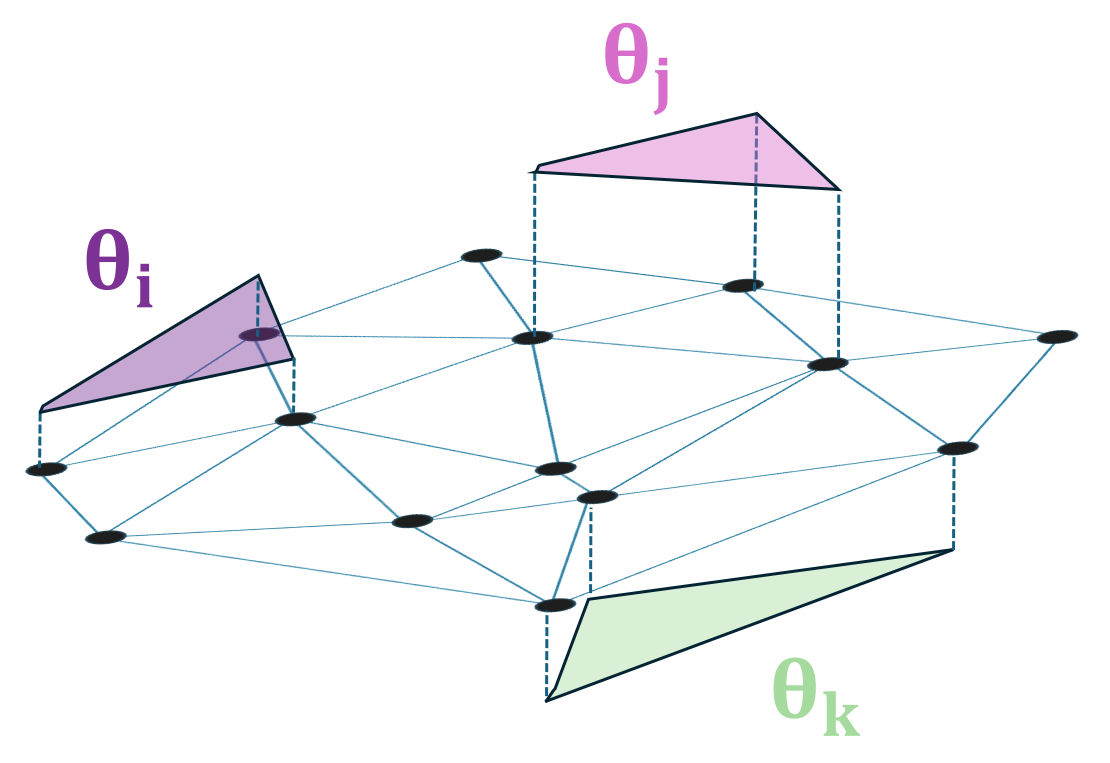}%
    \end{minipage}
\end{tcolorbox}
\vspace{-7mm}
\begin{tcolorbox}[title = (\RomanNumeralCaps{2}) Subspace Construction for each Row (\textit{section} \ref{rowspaceconstructionsection}), code={\singlespacing}]
\begin{minipage}[c]{0.62\textwidth}
 1. Partition $\mcal{X}_i \times \Omega$ into a hierarchical decomposition of $\mcal{O}(\log{(\varepsilon^{-1})})$ admissible domains and $\mcal{O}(1)$ small non-admissible domains. 
\end{minipage}
\hfill
\begin{minipage}[c]{0.33\textwidth}
    \centering
    \includegraphics[width =  \linewidth]{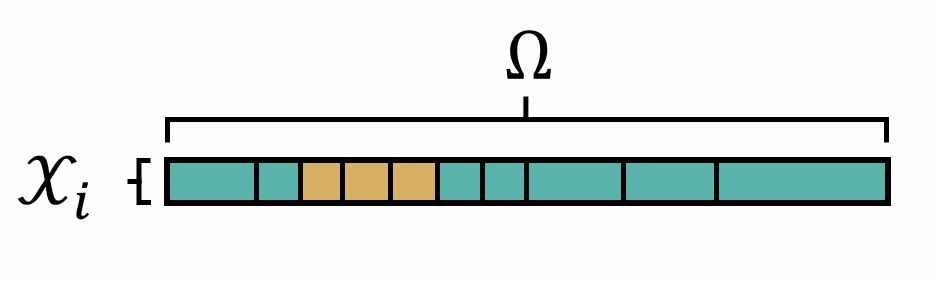}
\end{minipage}

\raggedright

\tcbline
2. Approximate the kernel by low-rank separable polynomial expansions on admissible sub-domains $\mcal{X}_i \times \mcal{Y}$. 

    \centering
    \includegraphics[width=0.5\linewidth]{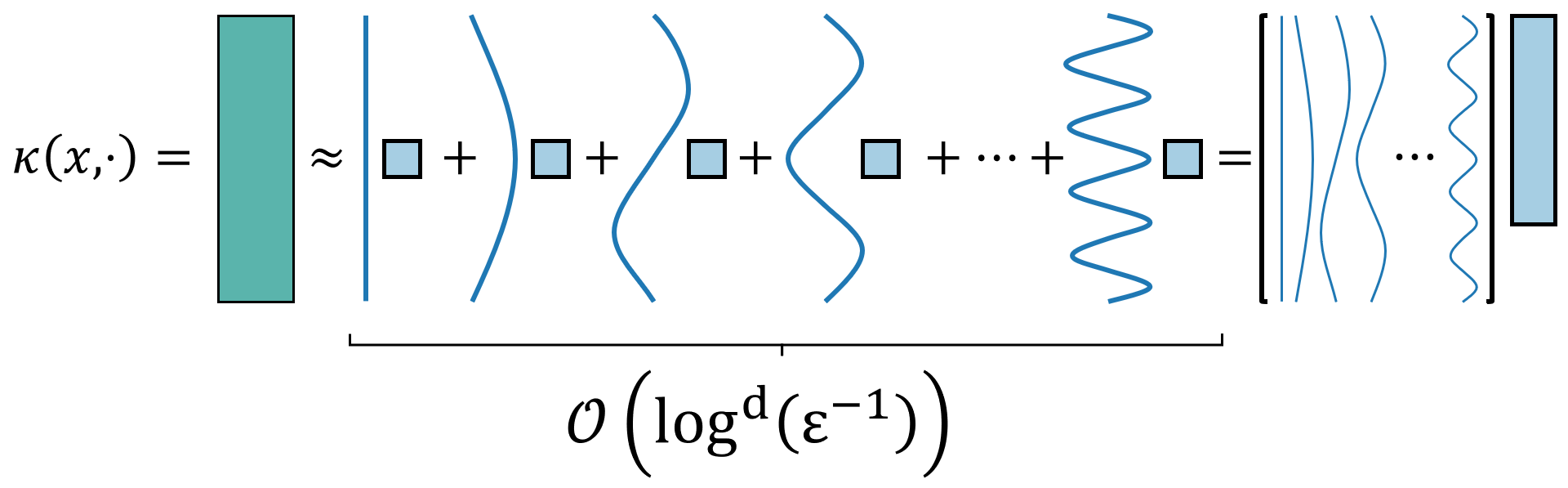}

\tcbline
\begin{minipage}[c]{0.62\textwidth}
3. Construct a $\mcal{O}(\text{log}^{d+1}(\varepsilon^{-1}))$-dimensional subspace $\mcal{V}_i$, spanned by  polynomial basis functions on the admissible domains and only a constant number of finite elements on the non-admissible domains, that approximates the $i$-th row of the discretized operator up to a desired error tolerance $\varepsilon$.
\end{minipage}
\hfill
\begin{minipage}[c]{0.33\textwidth}
    \centering
    \includegraphics[width =  0.95 \linewidth]{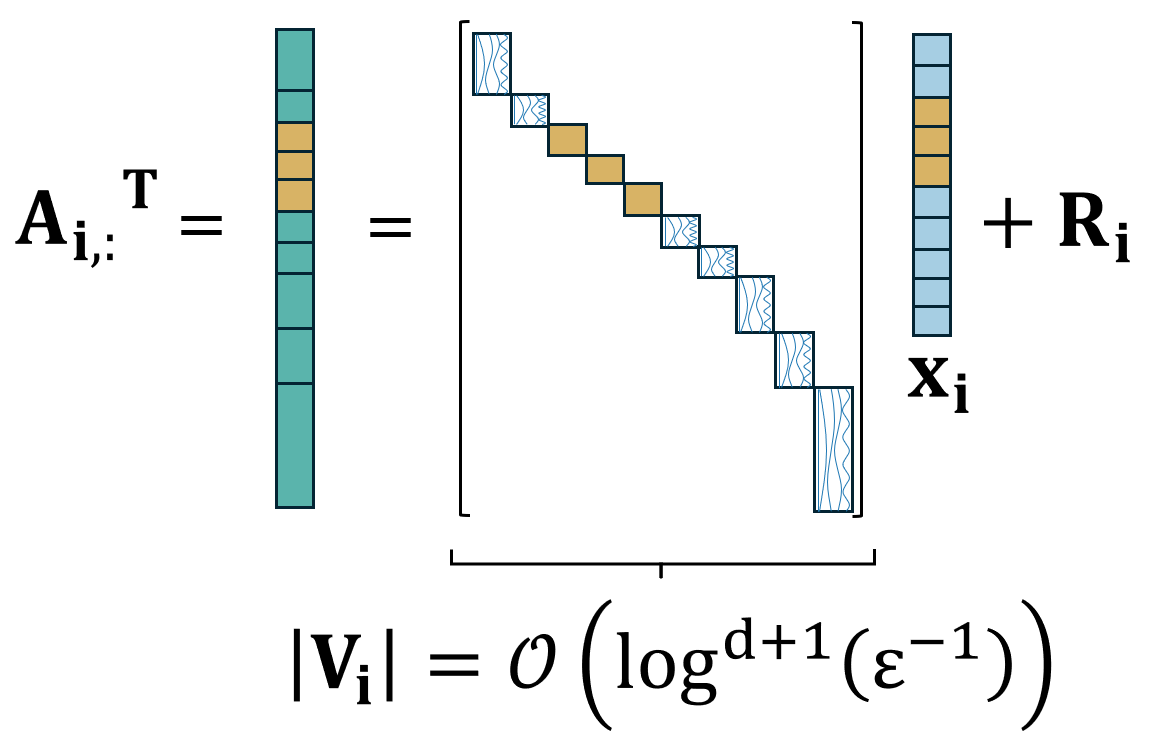}
\end{minipage}
\end{tcolorbox}
\vspace{-4.5mm}
\begin{tcolorbox}[title = (\RomanNumeralCaps{3}) Apply Matrix-Approximation Algorithm]
\begin{minipage}[c]{0.62\textwidth}
    Form subspaces $\mcal{V}_i$ for each row $i = 1, \dots, |\mcal{T}_h|$ of $\bld{A}_h$. 
    Query $\mcal{A}$ with $\mcal{O} \left( \rho^{-1}\log(\varepsilon^{-1})
\log^d(\varepsilon^{-1}\delta^{-1/2})\right)$  functions with standard Gaussian random coefficients in the finite element space to approximate $\bld{A}_h$ using Algorithm~\ref{alg:amselgen}. This gives a finite rank $\varepsilon$-approximation $\tilde{\mcal{A}}$ satisfying Eqn.(\ref{probboundintheorem}) with probability at least $1 - \delta$.
\end{minipage}
\hfill
\begin{minipage}[c]{0.33\textwidth}
    \centering
    \includegraphics[width = 0.95 \linewidth]{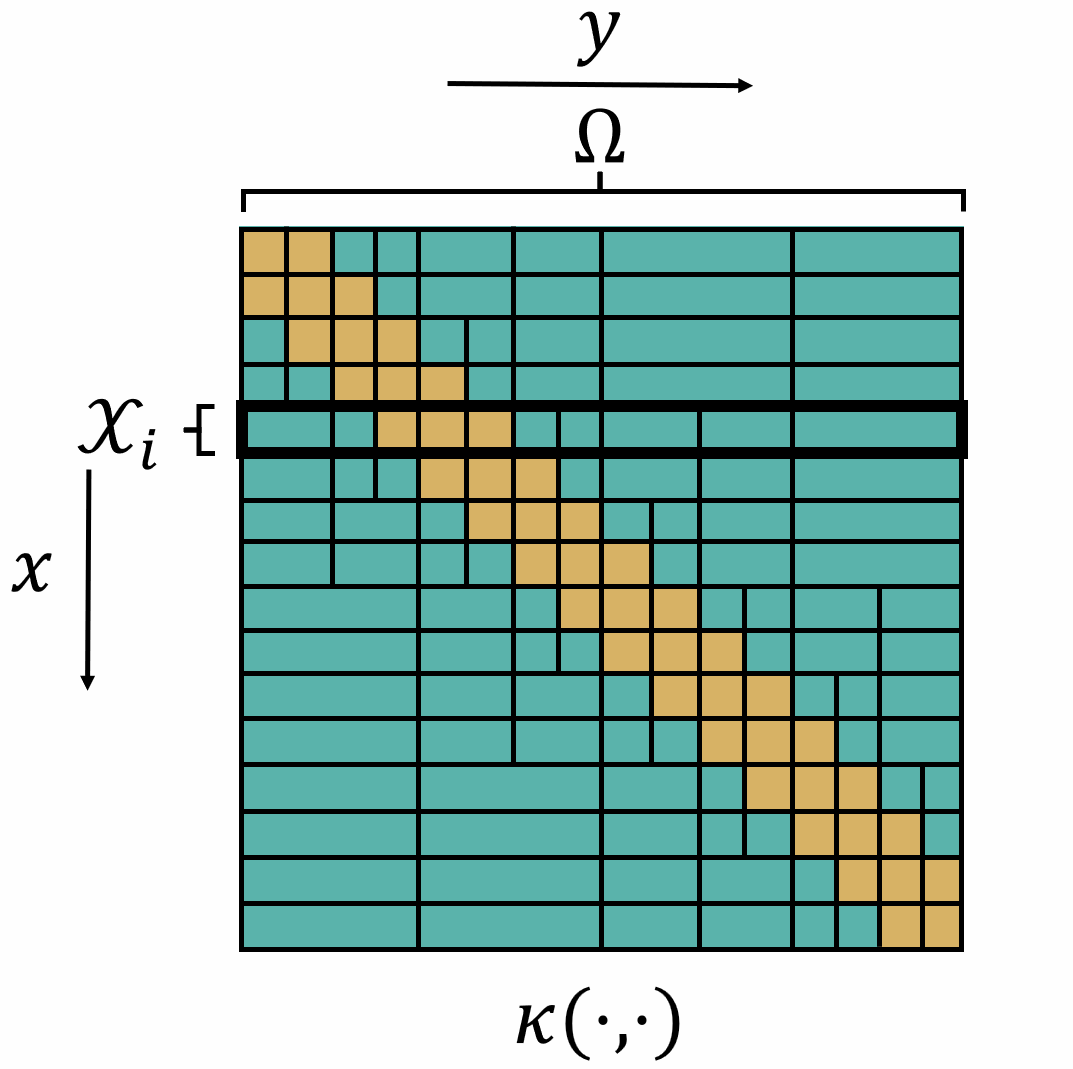}
\end{minipage}
\end{tcolorbox}
    \caption{Explanatory Sketch of Algorithm~\ref{adjointfreealg}}
    \label{roadmap}
\end{figure}

We introduce Algorithm~\ref{adjointfreealg} to approximate a class of Hilbert-Schmidt integral operators
\begin{equation}\label{eqn:integral_operator}
    \mcal{A}: L^2(\Omega) \to L^2(\Omega), \qquad
    \mcal{A}f(x) = \int_{\Omega} \kappa(x,y) f(y) \dd y
\end{equation}
with off-diagonal smoothness properties
using only forward queries $f \mapsto \mcal{A} f$ for a bounded domain $\Omega \subset \R^d$.
After giving the specific assumptions defining this class of possibly non-self-adjoint operators, we state Theorem~\ref{bigtheorem}.
This result shows that the algorithm produces finite dimensional approximations of these operators with $\varepsilon$-accuracy in Hilbert-Schmidt norm using a number of queries that grows only polylogarithmically with the inverse error tolerance $\varepsilon^{-1}$.
This implies a stretched exponential convergence rate with respect to the number of queries we are allowed to access, as we formalize in Corollary~\ref{corquerybudget}.
The result is significant because the operators being approximated can be non-self-adjoint and can have algebraic rates of singular value decay, yet Algorithm~\ref{adjointfreealg} converges at a superalgebraic rate without accessing adjoint queries $g \mapsto \mcal{A}^* g$.

A schematic overview of Algorithm~\ref{adjointfreealg} is shown in
Figure~\ref{roadmap}.
To apply the matrix approximation framework of
Section~\ref{amselgensection}, we represent the integral operator on a
piecewise-constant finite element space.
For $h\in\mcal{H}$, let
$\mcal{T}_h=\{\mcal{X}_i\}_{i=1}^{|\mcal{T}_h|}$
denote the quasi-uniform mesh from
Assumption~\ref{assumedomaingeometry}.
Quasi-uniformity implies
$|\mcal{T}_h|=\mcal{O}(h^{-d})$, as established in
Lemma~\ref{numberofmeshcells}.
For each mesh cell $\mcal{X}_i \in \mcal{T}_h$, define
\[
    \theta_i(x)
    := \frac{1}{\sqrt{\mathrm{vol}(\mcal{X}_i)}}
       \mathds{1}_{\mcal{X}_i}(x),
    \qquad
    \mcal{V}_h^0
    := \operatorname{span}
       \{\theta_i : i=1,\ldots,|\mcal{T}_h|\},
\]
where $\mathds{1}_{\mcal{X}_i}$ is the indicator function equal to $1$ on $\mcal{X}_i$ and $0$ everywhere else. 
Since the mesh cells have pairwise-disjoint interiors,
$\{\theta_i\}_{i=1}^{|\mcal{T}_h|}$ is an orthonormal basis for
$\mcal{V}_h^0\subset L^2(\Omega)$.
Let $J_h:\R^{|\mcal{T}_h|}\rightarrow L^2(\Omega)$ denote the
corresponding isometry $\bld{w} \mapsto \sum_{i=1}^{|\mcal{T}_h|}w_i\theta_i$ onto $\mcal{V}_h^0$.
Its adjoint maps a function to its coefficients in this basis, with
$(J_h^*v)_i = \langle v,\theta_i \rangle_{L^2(\Omega)}$.
Hence the $L^2(\Omega)$-orthogonal projection onto $\mcal{V}_h^0$ is
$P_h=J_hJ_h^*$, and the projected operator $P_h\mcal{A}P_h$ is
represented by the matrix
\[
    \bld{A}_h
    := J_h^*\mcal{A}J_h
    \in \R^{|\mcal{T}_h|\times|\mcal{T}_h|}.
\]
Thus, for a fixed mesh resolution $h$, approximating
$P_h\mcal{A}P_h$ reduces to approximating the matrix $\bld{A}_h$ from
forward matrix-vector products
$\bld{x}\mapsto J_h^*\mcal{A}(J_h\bld{x})$,
placing the problem in the setting of
Section~\ref{amselgensection}.

Algorithm~\ref{adjointfreealg} combines the finite-element representation above with structured row-space approximation to construct an adjoint-free approximation of $\mcal{A}$.
Theorem~\ref{bigtheorem} shows that, under our asymptotic smoothness and regularity assumptions, its parameters can be chosen to guarantee a prescribed Hilbert-Schmidt accuracy with high probability.
For a given mesh, the algorithm constructs low-dimensional subspaces $\bld{V}_i$ that approximate the rows of $\bld{A}_h$ using local polynomial approximations on a hierarchical partition of $\Omega$, as illustrated in Figure~\ref{roadmap}.
It then applies randomized forward queries of $\mcal{A}$, using functions with standard Gaussian coefficients in the piecewise-constant finite element basis, and invokes Algorithm~\ref{alg:amselgen} to recover an approximation $\tilde{\bld{A}}_h$ whose rows lie in the prescribed subspaces.
The corresponding operator approximation is
\[
    \tilde{\mcal{A}} = J_h \tilde{\bld{A}}_h J_h^\star.
\]
The construction depends on four parameters $(h,\eta,m,q)$, where $h$ is the mesh size, $\eta$ is the admissibility parameter governing the hierarchical partition, $m$ is the polynomial degree used on admissible subdomains, and $q$ is the number of non-adaptive Gaussian queries.
The hierarchical partition and row-space construction are implemented by subroutines developed in later sections.
Under the assumptions of Theorem~\ref{bigtheorem}, the theorem gives sufficient conditions on $(h,\eta,m,q)$ under which $\tilde{\mcal{A}}$ achieves an $\varepsilon$-accurate Hilbert-Schmidt approximation with probability at least $1-\delta$.

\begin{algorithm}
\caption{Adjoint-free approximation of asymptotically smooth integral operators} \label{adjointfreealg}
 \hspace*{\algorithmicindent}    
\begin{algorithmic}[1]\onehalfspacing

\Procedure{Adjoint-free Approximation}{ $f \mapsto \mcal{A}f, h, \eta, m,q$}

\State $\text{Form operators} \ J_h,J_h^*$  \color{gray} \algorithmiccomment{FE mapping operators}
\color{black}

\For {$i = 1,2, ..., |\mcal{T}_h|$} \do   

\State $\text{Select $i$-th mesh cell} \ \mcal{X}_i$  

\State Construct an $\eta$-hierarchical decomposition of $[0,1]^d$ w.r.t $\mcal{X}_i$ \color{gray} 
\algorithmiccomment{Lemmas~\ref{lemmaNAdomains},\ref{lemmaAdomains}}
\color{black} 

 $\left(\{\mcal{Y}_{i,j}\}_{j=1}^{N_\text{adm}(i)}, \{\mcal{N}_{i,p}\}_{p=1}^{N_\text{non}(i)} \right) = \textsc{Hierarchical Partition}(\mcal{X}_i,[0,1]^d, \eta, h)$ 
 \color{gray}
\algorithmiccomment{Alg. \ref{alghierarchicaldecomp}}
\color{black}

\State Form an orthonormal basis to approximate $i$-th row of $\bld{A}_h:=J_h^* \mcal{A} J_h$  \color{gray} 
\algorithmiccomment{Lemma~\ref{lemmarowsubspaceconstruction}}
\color{black} 

 $\bld{V}_i$ = \textsc{Row-Subspace Construction}$\left(\{\mcal{Y}_{i,j}\}_{j=1}^{N_\text{adm}(i)}, \{\mcal{N}_{i,p}\}_{p=1}^{N_\text{non}(i)}, m \right)$ \color{gray} 
\algorithmiccomment{Alg.  \ref{algconstructrowsubspace}}
\color{black}

\EndFor

\State Apply $q$ Gaussian queries to  $\bld{A}_h$ to recover it via row-wise least squares  \color{gray} 
\algorithmiccomment{Theorem~\ref{thm:amselgen}}
\color{black}

 $\bld{\tilde{A}}_h$ = \textsc{Rowwise-Matrix-Approximation} $\left( \bld{x} \mapsto J_h^* \mcal{A} (J_h \bld{x}), \bld{V}_1, ...,\bld{V}_{|\mcal{T}_h|}, q \right)$ \color{gray} 
\algorithmiccomment{Alg.  \ref{alg:amselgen}}
\color{black}

\Return $\tilde{\mcal{A}} =  J_h \bld{\tilde{A}}_h J_h^*$
\EndProcedure
\end{algorithmic}
 \hspace*{\algorithmicindent} \textbf{Output:} Adjoint-free reconstruction $\tilde{\mcal{A}}$ satisfying \[\mathbb{P}[\| \mcal{A} - \tilde{\mcal{A}}\|_{HS:L^2 \rightarrow L^2} \leq  \varepsilon \left( \| \mcal{A}\|_{HS:L^2 \rightarrow H^\rho} + \| \mcal{A}^*\|_{HS:L^2 \rightarrow H^\rho} \right)] \geq 1 - \delta\]
 for $q,m,h,\eta$ specified in Theorem~\ref{bigtheorem}
\end{algorithm}

The following assumptions define the class of integral operators where we guarantee the performance of Algorithm~\ref{adjointfreealg}.
 Assumption~\ref{assumedomaingeometry} enforces suitable conditions on the physical domain geometry and finite element mesh.
 We do not prescribe a particular construction of the finite element mesh.
 Instead, Assumption~1\ref{assumefemmesh} assumes that $\Omega$ admits a family of exact, quasi-uniform meshes ensuring geometric control of the mesh cells. 
 Exact curved mesh constructions for sufficiently regular domains are available in the finite element literature; see, for example, \cite{bernardi1989optimal,dubois1990discrete,kawecki2020finite}. 
 This formulation allows for different mesh constructions to be used provided they satisfy the stated geometric bounds. 
 The rescaling in 1\ref{assumerescaling} is a matter of convenience, as it is simpler to partition the cube $[0,1]^d$ and then restrict the partition to $\Omega$ than to partition $\Omega$ directly. 
 Furthermore, Assumption~1\ref{assumedomaincondition} ensures that restricting a cube in $[0,1]^d$ to $\Omega$ is a finite union of convex, hence locally Lipschitz components; see Remark~\ref{rmkconvexlocallipschitz} below.

 
\begin{assumption}[Domain \& discretization geometry]
\label{assumedomaingeometry}
~
\begin{enumerate}[(a)]
    \item \label{assumedomaincondition} 
  Let $\Omega \subset \R^d$ be a bounded Lipschitz domain. 
  Assume that $\overline{\Omega}$ is the union of finitely many compact convex sets $\mcal{C}_1, \ldots, \mcal{C}_{N_{cc}}$ with pairwise-disjoint nonempty interiors.
  \item \label{assumefemmesh} Assume $\Omega$ admits a sequence of quasi-uniform meshes $(\mcal{T}_h)_{h \in \mcal{H}}$ where each mesh $\mcal{T}_h$ exactly covers $\overline{\Omega}$ by cells with pairwise-disjoint interiors, and the sequence of mesh sizes $\mcal{H} = \{ h_n \}_{n=1}^{\infty}$ satisfies $h_n \to 0$ and the refinement condition $c_{\mcal{H}} h_n \leq h_{n+1} < h_n$ with fixed $c_{\mcal{H}} \in (0,1)$.
  With $h_{\mcal{X}} = \diam(\mcal{X})$ denoting the diameter of a mesh cell $\mcal{X}$, the diameter of the largest ball it contains being $\rho_{\mcal{X}}$, and $h = \max_{\mcal{X}\in \mcal{T}_h} h_{\mcal{X}}$ being the mesh size, we assume there is a shape-regularity constant $\sigma_\#$ (Definition 11.2 in \cite{ern2021finite}) and a quasi-uniformity constant $c_\text{quasi}$ (Definition 22.20 in \cite{ern2021finite}) such that $\sigma_{\mcal{X}} := \frac{h_{\mcal{X}}}{\rho_{\mcal{X}}} \leq \sigma_{\#}$ and $\frac{h_{\mcal{X}}}{h} \geq c_\text{quasi}$ for all $\mcal{X}\in\mcal{T}_h$ and all $h\in\mcal{H}$.
  \item \label{assumerescaling}Without loss of generality, we assume $\Omega \subset [0,1]^d$ and $\diam(\Omega) = 1$.
\end{enumerate}
\end{assumption}

\begin{remark}
    Open bounded convex domains of $\R^d$ have local Lipschitz boundary via Corollary 1.2.2.3 in \cite{grisvard2011elliptic}. Compactness of the boundary ensures the domain satisfies the strong local Lipschitz condition defined in Section 4.9 of \cite{adams2003sobolev}. 
    \label{rmkconvexlocallipschitz}
\end{remark}

\begin{remark}[Approximate meshes]
Assumption~1\ref{assumefemmesh} requires the meshes to cover $\Omega$ exactly. 
The analysis could also be extended to meshes that approximate $\Omega$, i.e. 
\[
    \Omega_h := \bigcup_{i=1}^{|\mcal{T}_h|} \mcal{X}_i \subset \Omega.
\]
In this case, the piecewise-constant finite element space may be viewed as a subspace of $L^2(\Omega)$ by zero extension on $\Omega \setminus \Omega_h$, and $P_h$ denotes the corresponding $L^2(\Omega)$-orthogonal projection. 
It suffices to establish the approximation estimate
\[
    \| (I-P_h)v\|_{L^2(\Omega)} \leq c h^\rho \|v \|_{H^\rho(\Omega)}
\]
for all $\rho \in (0,1), v \in H^\rho(\Omega)$ and $h \in \mcal{H}$, with $c$ independent of $\rho$ and $h$.  
For approximating meshes, this estimate contains an additional term $\|v\|_{L^2(\Omega \setminus \Omega_h)}$ measuring the approximation of the domain $\Omega$ by $\Omega_h$. 
Since controlling this term is independent of the structured matrix-approximation arguments developed here, we restrict the presentation to exact meshes.   
\end{remark}

The key regularity properties defining the relevant class of integral operators are stated below.
\begin{assumption}[Regularity]
    \label{assumeoperatorregularity}
    We assume the integral operator in Equation~\ref{eqn:integral_operator} satisfies:
    \begin{enumerate}[(a)]
        \item  $\mcal{A}$ and $\mcal{A}^*$ are bounded as operators from $L^2(\Omega)$ into a Sobolev space $H^s(\Omega)$ for some $s > \frac{d}{2}$.\label{assumeglobalregularity}
        \item The kernel function $\kappa$ is $L^2$- asymptotically smooth in the $y$-variable, as given by Definition~\ref{def:l2asymptoticsmoothness}. \label{assumeasymptotic smoothness}
    \end{enumerate}
\end{assumption}

\begin{definition}[$L^2$ Asymptotic Smoothness]
\label{def:l2asymptoticsmoothness}
A kernel function $\kappa(\cdot, \cdot): \Omega \times \Omega \rightarrow \R$ is $L^2$- asymptotically smooth with respect to the $y$-variable if there exist constants $C,\gamma > 0$ such that the following holds. 
Let $\mcal{X} \subset \Omega$ be measurable and let $\mcal{Y} \subset \Omega$ be a bounded, connected Lipschitz domain satisfying $\dist(\mcal{X}, \mcal{Y}) > 0$. 
Define an enlargement $\widehat{\mcal{Y}} := \{z \in \Omega: \dist(z,\mcal{Y}) \leq \frac{1}{2}\dist(\mcal{X},\mcal{Y})\}$. 
Then, for every multi-index $\alpha \in \mathbb{N}_0^d$ and almost every $x \in \mcal{X}$, the function $\kappa(x,\cdot) \in L^2(\widehat{\mcal{Y}})$,
the weak derivative $\partial_y^\alpha \kappa(x,\cdot) \in L^2(\mcal{Y})$ and the following inequality is satisfied
\[\| \partial_y^\alpha \kappa(x,\cdot)\|_{L^2(\mcal{Y})} \leq \frac{C |\alpha|! \gamma^{|\alpha|}}{\text{dist}(\mcal{X},\mcal{Y})^{|\alpha|}} \| \kappa(x,\cdot)\|_{L^2(\widehat{\mcal{Y}})}.\]
\end{definition}

Assumptions 1\ref{assumedomaincondition} and 2\ref{assumeglobalregularity} automatically ensure the integral operator is Hilbert-Schmidt via a fractional-order extension of Maurin's Theorem (see Lemma 2.3 in \cite{kovacs2023hilbert}). 
Furthermore, they ensure $\mcal{A}$ and $\mcal{A}^*$ are Hilbert-Schmidt as operators from $L^2(\Omega)$ into any Sobolev space $H^{\rho}(\Omega)$ with $0 \leq \rho < s - \frac{d}{2}$.
We rely on this to deduce algebraic rates of convergence for finite element approximations in the $L^2(\Omega) \to L^2(\Omega)$ Hilbert-Schmidt norm.
The global regularity in Assumption~2\ref{assumeglobalregularity} holds for solution operators of second-order uniformly elliptic PDEs in dimensions $d = 1$, $2$, and $3$ as a consequence of standard $L^2(\Omega) \to H^2(\Omega)$ regularity results \cite{renardy2004introduction}.
$L^2$ asymptotic smoothness (Definition~\ref{def:l2asymptoticsmoothness}) captures the ``smoothing at a distance'' properties expected from elliptic solution operators with sufficiently smooth coefficients.
Roughly speaking, it ensures that the kernel becomes progressively smoother in its second entry on domain pairs that are further apart.
Nevertheless, the definition permits kernels with low-regularity or even singular behavior along the diagonal $x=y$, as we discuss further in Section~\ref{sec:asymptoticsmoothness}.

\begin{tcolorbox}[code = {\singlespacing}]
    \begin{theorem}
Let Assumptions 1\ref{assumedomaincondition},\ref{assumefemmesh},\ref{assumerescaling} and 2\ref{assumeglobalregularity},\ref{assumeasymptotic smoothness} hold and let the admissibility constant $0 < \eta < \text{min}\{\frac{1}{\gamma d},1\}$.
There exist $\varepsilon_0, K_0, K^\prime_0$ with $\varepsilon_0 > 0$ such that the following holds for all $\varepsilon \in (0, \varepsilon_0), K \geq K_0, K^\prime \geq K_0^\prime, \delta \in (0,1)$, and $\rho \in (0,1)$ with $\rho < s - \frac{d}{2}$.
Let the finite element mesh size parameter be 
\[
    h:=\max\{h'\in\mcal H: h'\le \varepsilon^{2/\rho}\},\]
and let the polynomial degree on convex components of admissible subdomains satisfy
\[m  \geq  K \log \left( \varepsilon^{-1}\delta^{-1/2} \right). \]
If the number of queries is at least
\[q \geq \frac{K^{\prime}}{\rho}\log(\varepsilon^{-1}) m^d,\]
then Algorithm~\ref{adjointfreealg} produces an approximation $\tilde{\mcal{A}} = Alg(f \mapsto \mcal{A} f, q, m, h, \eta)$ satisfying
\begin{equation}
\label{probboundintheorem}
\mathbb{P} \left( \| \mcal{A} - \tilde{\mcal{A}}\|_{HS:L^2 \rightarrow L^2} 
\leq \varepsilon \left( \| \mcal{A}\|_{HS:L^2 \rightarrow H^\rho} + \| \mcal{A}^*\|_{HS:L^2 \rightarrow H^\rho} \right) \right)
\geq 1 - \delta.\end{equation}
 \label{bigtheorem}
\end{theorem}
\end{tcolorbox}
The constants $\varepsilon_0, K_0,$ and $K_0'$ appearing in the theorem depend only on the fixed constants in the domain, mesh, regularity, and asymptotic-smoothness assumptions. In particular, they are independent of $\varepsilon, \delta,$ and $\rho$, and of the operator except through the constants $C$ and $\gamma$. 
Their precise dependencies are given later in the proof in Section~\ref{sectheoremproof}.
Theorem~\ref{bigtheorem} therefore uses only$\mcal{O}\left(\rho^{-1}\log(\varepsilon^{-1})
\log^d(\varepsilon^{-1}\delta^{-1/2}) \right)$ queries to the operator to achieve an $\varepsilon$-accurate approximation, demonstrating the query-efficiency of Algorithm~\ref{adjointfreealg}. 
When the discretization and the dimensions of the local approximation spaces are chosen in accordance with an available query budget, the error bound can be inverted to yield a stretched exponential rate in the number of Gaussian queries.
This consequence is formalized in Corollary~\ref{corquerybudget}. 

\begin{corollary} [Query Budget]
\label{corquerybudget}
    Assume the setting of Theorem~\ref{bigtheorem} with $\eta < \min\big\{\frac{1}{\gamma d}, 1\big\}$. There exist constants $\tilde{\varepsilon}_0, \tilde{K}_0,\tilde{K}^\prime_0$ such  that for all $\tilde{K} \geq \tilde{K}_0, \tilde{K}^\prime \geq \tilde{K}_0^\prime$, $\delta \in (0,1)$ and $\rho \in (0,1)$ with $\rho < s - \frac{d}{2}$, using a number of Gaussian queries 
\[
    q > q_0
        := \frac{\tilde{K}^\prime \tilde{K}^d}{\rho} \left[ \max \Big\{\log{(\tilde{\varepsilon}_0^{-1})}, \log{(\delta^{-1/2})}\Big\} \right]^{d+1}
\] 
with a finite element mesh size
\[
    h = \text{max}\Bigg\{h' \in \mcal{H}: h' \leq e^{-\frac{2}{\rho}\left(\frac{ q \rho}{\tilde{K}^\prime\tilde{K}^d} \right)^\frac{1}{d+1}}\Bigg\}
\]
and a polynomial degree of $m_0$ for each admissible cube restricted to a convex component of $\Omega$ 
\[
    m_0 = \min\Bigg\{m \in \mathbb{N}: m \geq (\tilde{K}-1)\left( \frac{q \rho}{\tilde{K}^\prime\tilde{K}^d} \right)^\frac{1}{d+1} \Bigg\}, 
\]
Algorithm~\ref{adjointfreealg} constructs a finite rank approximation $\tilde{\mcal{A}} = Alg(f \mapsto \mcal{A}f, q, m_0, h, \eta)$ such that 
\[
    \mathbb{P}\left( \| \mcal{A} - \tilde{\mcal{A}}\|_{HS:L^2 \rightarrow L^2}
        \leq e^{- \left( \frac{q \rho}{\tilde{K}^\prime\tilde{K}^d} \right)^\frac{1}{d+1}} \left(\| \mcal{A}\|_{HS:L^2 \rightarrow H^\rho} + \| \mcal{A}^*\|_{HS:L^2 \rightarrow H^\rho} \right) \right) 
        \geq 1 - \delta. 
\]
\end{corollary}

\begin{proof}
    Proof in \ref{adxcorquerybudget}.
\end{proof}

The corollary identifies an asymptotic regime in which, after fixing the problem-dependent constants, increasing the query budget while refining the discretization and increasing the dimensions of the local approximation spaces yields stretched-exponential convergence.
A key structural ingredient enabling this query-efficiency is the asymptotic smoothness assumption of the kernel. 
The next section~\ref{sec:asymptoticsmoothness} therefore briefly places our formulation of asymptotic smoothness in the context of previous work and discusses its connection to hierarchical approximation and elliptic Green's functions. 
We then turn to numerical experiments in Section~\ref{sec:numericalexamples} for one-dimensional non-self-adjoint elliptic solution operators to investigate whether the predicted stretched exponential decay rate can be observed for moderate numbers of forward queries and how the observed convergence changes as the underlying PDE becomes more challenging. 

\subsection{Discussion on Asymptotic Smoothness}
\label{sec:asymptoticsmoothness}

Asymptotic smoothness is a regularity condition used to characterize the off-diagonal compressibility of kernels.
Several closely related formulations appear in the hierarchical-matrix and wavelet literature. 
The definition used by Bebendorf and Grzhibovskis in \cite{bebendorf2006accelerating} is of particular relevance. For fixed $x$, they call a kernel $\kappa(x,\cdot)$ asymptotically smooth with respect to the $y$-variable if, for $p = |\alpha|$ and $R = \|x-y\| / 2$,  
\[
    |\partial_y^\alpha \kappa(x,y)| \leq c p! \gamma^{p} R^{-p}\sup_{z \in B_R(y)} |\kappa(x,z)|.
\]
Thus, away from the singularity, higher derivatives are controlled by the magnitude of the kernel on a neighborhood whose size is proportional to the distance from $x$.
They prove asymptotic smoothness of the entries of the singularity matrix of constant coefficient elliptic systems on $\R^d$ based on an argument on a nested sequence of subdomains. 
Related notions of asymptotic smoothness arise elsewhere in the hierarchical matrix literature, including a formulation that holds for fundamental solutions of elliptic operators in \cite{hackbusch2015hierarchical} and a pointwise variant in \cite{bebendorf2008hierarchical}. 
Although the precise formulations differ among these definitions, their common feature is factorial derivative growth with an inverse power of the separation distance, which yields exponentially accurate polynomial approximations on well-separated domains. 

Definition~\ref{def:l2asymptoticsmoothness} is based on the same principle, but is formulated in an $L^2$-sense. 
Rather than bounding a derivative at a point by a local pointwise measure of the kernel, we compare the $L^2$-norm of $\partial^\alpha_y\kappa(x,\cdot)$ on a separated domain with the $L^2$ norm of $\kappa(x,\cdot)$ on a slightly enlarged neighborhood. 
This modification is natural for the present setting of this work for two reasons. 
First, the resulting approximation estimates integrate directly into the Hilbert-Schmidt error analysis used throughout this work. 
Second, $L^2$ derivative estimates of this form are closely aligned with standard interior elliptic regularity Caccioppoli estimates. 
To illustrate that the assumption is not merely technical, we verify the property for the Green function of the Poisson equation prescribed with homogeneous Dirichlet boundary conditions on the unit ball and present an argument for the Green function of a constant-coefficient elliptic operator on the flat torus. 

\subsubsection{Green's Functions for Poisson Equation}

Consider the Dirichlet problem on the unit ball
\[
    \begin{cases}
        -\Delta u = f \quad \text{in} \ B_1(0), \\
        u = 0 \quad \text{on} \ \partial B_1(0).
    \end{cases}
\]
Its solution admits the Green representation
\[
    u(x) = \int_{B_1(0)} G(x,y) f(y) dy
\]
for $x \in B_1(0)$. 
Let $\omega_d$ denote the volume of the $d$-dimensional unit ball. 
The fundamental solution of Laplace's equation is 
\[
    \Phi(x):= \begin{cases}
        \frac{-1}{2 \pi} \log{|x|}, \quad d = 2 \\
        \frac{1}{d(d-2) \omega_d} \frac{1}{|x|^{d-2}}, \quad d \geq 3
    \end{cases}
\]
defined for $x \in \mathbb{R}^n$ and $x \neq 0$.
The Green's function is defined as follows
\[
    G(x,y):= \Phi(y-x) - \Phi(|x|(y-x^\star)), \quad x,y \in B_1(0), \quad x \neq y
\]
where $x^\star = \frac{x}{|x|^2}$ denotes the Kelvin inversion point (see Section 2.2 Equation 41 in \cite{evans2022partial}).

\begin{proposition}
The Green's functions on the unit ball for dimensions $d = 2$ and 3 are $L^2$-asymptotically smooth.     
\end{proposition}

\begin{proof}
    Let $\mcal{X}, \mcal{Y} \subset B_1(0)$ such that $\text{dist}(\mcal{X}, \mcal{Y}) > 0$. 
    Fix $x \in \mcal{X}$. 
   
    For $x \neq 0$, extend $G(x,\cdot)$ to the function defined as 
    \[
        \tilde{G}_x(y) := \begin{cases}
        \frac{-1}{2\pi} \log{(|y-x|)} + \frac{1}{2\pi} \log{(|x||y - x^\star|)}, \quad d = 2 \\

        \frac{1}{3\omega_3} \left[ \frac{1}{|y-x|} - \frac{1}{|x| |y - x^\star|}  \right], \quad d = 3
        \end{cases} 
    \]
    on $\R^d \backslash \{x, x^\star\}$. 
    If $x=0$, set 
    \[
    \tilde{G}_0(y) = \begin{cases}
        \frac{-1}{2\pi} \log{|y|}, \quad d = 2 \\
        \frac{1}{3 \omega_3} \left( \frac{1}{|y|} - 1\right), \quad d = 3
    \end{cases} 
    \]
    by continuity.
    Importantly, properties of the fundamental solution yield that
    \[
        \Delta \tilde{G}_x = 0 \quad \text{on} \quad \R^d \backslash \{x,x^\star\},
    \]
    and $\tilde{G}_x \in C^\infty(\R^d \backslash \{x,x^\star\})$ for $x \neq 0$. Similarily for $x = 0$, $\tilde{G}_0 \in C^\infty(\R^d \backslash \{0\})$ and is harmonic on $\R^d \backslash \{0\}$.
    
    Let $r:= \text{dist}(\mcal{X}, \mcal{Y})>0$, and define
    $\widehat{\mcal{Y}}:= \{z \in B_1(0): \text{dist}(z,\mcal{Y}) \leq \frac{r}{2}\}$ together with  an open set $D_{\mcal{Y}} := \{z \in \R^d: \text{dist}(z,\mcal{Y}) < \frac{r}{2}\}$.
    In particular, $\widehat{\mcal{Y}} \subset \overline{D}_{\mcal{Y}}$ and $D_\mcal{Y} \cap B_1(0) \subset \widehat{\mcal{Y}}$.

    We first verify that $\overline{D_\mcal{Y}}$ contains neither of the singularities of $\tilde{G}_x$.
    Since $x \in \mcal{X}$, $\text{dist}(x, \mcal{Y}) \geq r$.
    Moreover, for $x \neq 0$ and $y \in B_1(0)$,
    \begin{equation}
        |x|^2|x^\star - y|^2 = |x-y|^2 + (1-|x|^2)(1-|y|^2) \geq |x-y|^2. 
        \label{eqn:inversiondistance}
    \end{equation}
    Thus for all $y \in \mcal{Y}$, $|x^\star - y| \geq |x-y| \geq r$ since $|x| < 1$. 
    Hence, $\text{dist}(x^\star, \mcal{Y}) \geq r$.
    It follows that
    $\overline{D}_{\mcal{Y}} \subset \R^d \backslash \{x, x^\star\}$ for $x \neq 0$. 
    When $x = 0$, the extension $\tilde{G}_0$ has only one singularity at the origin, and $\text{dist}(0,\mcal{Y}) \geq r$ gives $\overline{D_\mcal{Y}} \subset \R^d \backslash \{0\}$.
    In both cases, $\tilde{G}_x$ is smooth and harmonic on an open neighborhood of $\overline{D_\mcal{Y}}$.

    Let $p \geq 1$ and let $\alpha \in \mathbb{N}_0^d$ satisfy $|\alpha| = p$. 
    Write $\alpha = e_{i_1} + \ldots + e_{i_p}$ for some $i_1, \ldots, i_p \in \{1, \ldots, d\}$. 
    Define nested domains 
    \[
        D_j = \Big\{z \in \R^d: \text{dist}(z,\mcal{Y}) < (1 - \frac{j}{p})\frac{r}{2} \Big\}
    \]
    for $j = 0, \ldots, p-1$, and set $D_p := \mcal{Y}$. 
    Then, 
    \[
        \mcal{Y} = D_p \subset D_{p-1} \subset \ldots \subset D_1 \subset D_0 =D_\mcal{Y},
    \]  
    and by construction, $\text{dist}(\overline{D_j}, \partial D_{j-1}) \geq \frac{r}{2p}$ for $j = 1, \ldots, p$.
    We now iterate a Caccioppoli estimate through this sequence of domains. 
    Define $v_0 := \tilde{G}_x$ and $v_j:= \partial_{i_j} v_{j-1}$ for $j = 1, \ldots, p$. 
    Then, $v_p = \partial^\alpha \tilde{G}_x$.
    Since $\tilde{G}_x$ is smooth and harmonic on a neighborhood of $\overline{D_0}$, each
    $v_{j-1} \in H^1(D_{j-1})$ is harmonic on $D_{j-1}$ for $j \in \{1, \ldots, p\}$. Applying
    Lemma 3.2 of \cite{bebendorf2006accelerating} with the compact set $\overline{D}_j \subset D_{j-1}$ gives
    \begin{align}
        \|v_j\|_{L^2(D_j)} & = \| \partial_{i_j} v_{j-1} \|_{L^2(D_j)} \leq \|\nabla v_{j-1}\|_{L^2(\overline{D}_j)}  \leq  \frac{2p c_\text{cac}}{r} \|v_{j-1}\|_{L^2(D_{j-1})},
    \end{align}
where $c_\text{cac} = \sqrt{8d^2(1 + 8d)}$. 
Iterating this estimate yields 
\begin{align}
     \| \partial^\alpha G(x, \cdot) \|_{L^2(\mcal{Y})} & = \| \partial^\alpha \tilde{G}_x \|_{L^2(\mcal{Y})} = \| v_p \|_{L^2(D_p)} \nonumber \\
    & \leq \left( \frac{2pc_\text{cac}}{r}\right)^p \|v_0\|_{L^2(D_0)} \nonumber \\
    & \leq \frac{(2ec_\text{cac})^p p!}{r^p} \|\tilde{G}_x\|_{L^2(D_\mcal{Y})} \nonumber
\end{align}
using that $p^p \leq e^p p!$ for integer $p \geq 1$ in the last inequality.

It remains to bound $\|\tilde{G}_x\|_{L^2(D_\mcal{Y})}$ in terms of $\|G(x,\cdot)\|_{L^2(\widehat{\mcal{Y}})}$. 
Decomposing $D_\mcal{Y}$ into disjoint subdomains inside and outside the unit ball gives 
\[             
    \|\tilde{G}_x\|^2_{L^2(D_\mcal{Y})} = \|\tilde{G}_x\|^2_{L^2(D_\mcal{Y} \backslash B_1(0))} + \|\tilde{G}_x\|^2_{L^2(D_\mcal{Y} \cap B_1(0))}.
\]
Since $\tilde{G}_x(z) = G(x,z)$ for $z \in B_1(0)$ and $D_\mcal{Y} \cap B_1(0) \subset \widehat{\mcal{Y}}$, 
we obtain 
\[
    \|\tilde{G}_x\|_{L^2(D_\mcal{Y} \cap B_1(0))} \leq \| G(x,\cdot)\|_{L^2(\widehat{\mcal{Y}})}.
\]
We next consider the portion outside of the unit ball. 
For $|z| > 1$, $|z^\star| < 1$ and rearranging the explicit formula for $\tilde{G}_x(z)$ gives
\[
    \tilde{G}_x(z) = - |z|^{2-d} G(x,z^\star).
\]
The inversion map $z \rightarrow z^\star$ is injective on $\R^d \backslash \{0\}$ and maps $D_\mcal{Y} \backslash B_1(0)$ into $\widehat{\mcal{Y}}$.
For $|z| \geq 1$ and any $y \in B_1(0)$,  Equation~\ref{eqn:inversiondistance} applied to $z^\star$ implies $|z - y| \geq |z^\star -y|$.
Thus, if $z \in D_\mcal{Y} \backslash B_1(0)$, 
\[
    \text{dist}(z^\star, \mcal{Y}) \leq \text{dist}(z,\mcal{Y}) < \frac{r}{2}.
\]
Thus, $z^\star \in \widehat{\mcal{Y}}$. 
Moreover, the determinant of the Jacobian for changing variables from $z \in \R^d$ to $z^\star = \frac{z}{|z|^2}$ satisfies $|\text{det} D z^\star (z)| = |z|^{-2d} = |z^\star|^{2d}$, so
\[
    dz = |z|^{2d} dz^\star.
\]
Since $\mcal{X}, \mcal{Y} \subset B_1(0)$, $r \leq 2$. 
Then, $z \in D_\mcal{Y} \backslash B_1(0)$  implies that $|z| < 2$, and, in particular, $|z|^{4} < 16$. 
The change of variables to $z^\star$ gives
\begin{align}
    \int_{D_\mcal{Y} \backslash B_1(0)} |\tilde{G}_x(z)|^2 dz 
    & = \int_{D_\mcal{Y} \backslash B_1(0)} |z|^{4-2d} |G(x,z^\star)|^2 dz \nonumber \\
    & \leq 16 \int_{D_\mcal{Y} \backslash B_1(0)} |z|^{-2d} |G(x,z^\star)|^2 dz \nonumber \\
    & \leq \int_{\widehat{\mcal{Y}}} 16 |G(x, z^\star)|^2 dz^\star = 16 \|G(x,\cdot)\|_{L^2(\widehat{\mcal{Y}})}^2. \nonumber
\end{align}
Combining the interior and exterior estimates yields
\[
\|\tilde{G}_x\|_{L^2(D_\mcal{Y})}^2 \leq 17 \| G(x,\cdot) \|_{L^2(\widehat{\mcal{Y}})}^2.
\]
Consequently, for every $\alpha \in \mathbb{N}_0^d$ with $p := |\alpha| \geq 1$, 
\[
    \|\partial^\alpha G(x,\cdot)\|_{L^2(\mcal{Y})}  \leq \sqrt{17} \frac{(2ec_\text{cac})^p p!}{\text{dist}(\mcal{X},\mcal{Y})^p} \|G(x,\cdot)\|_{L^2(\widehat{\mcal{Y}})}.
\]
The case $p=0$ follows trivially from $\mcal{Y} \subset \widehat{\mcal{Y}}$.
Hence, $G$ is $L^2$-asymptotically smooth with $C=\sqrt{17}$ and $\gamma = 2ec_\text{cac}$.

\end{proof}
\begin{remark}
    For $d=1$, the Green's function is piecewise affine in the $y$-variable away from the diagonal $x=y$. Its $L^2$-asymptotic smoothness follows directly. 
    The Green's functions for $d \geq 4$ remain $L^2$-asymptotically smooth off the diagonal, but are not globally square-integrable on $B_1(0) \times B_1(0)$. 
\end{remark}
\subsubsection{Green's Functions for Constant-Coefficient Elliptic Operators on a Flat Torus}

We next give a concrete class of solution operators whose Green functions satisfy the $L^2$-asymptotic smoothness assumption in Definition~\ref{def:l2asymptoticsmoothness}. 
Define a second-order linear elliptic PDE with constant coefficients defined on the flat torus $\mathbb{T}^d$
\[
    \mcal{L} u := -\nabla \cdot (\bld{A} \nabla u) + \bld{b}\cdot \nabla u + c u = f
\]
with forcing $f \in L^2(\mathbb{T}^d)$.
In particular, the coefficient matrix $\bld{A} \in \R^{d \times d}$ is symmetric with all eigenvalues $\lambda$
\[
    0 < \lambda_\text{min} \leq \lambda \leq \lambda_\text{max}
\]
and $\bld{b} \in \R^d$ and $c > 0$.
Since $c > 0$, the associated bilinear form is coercive and $\mcal{L}$ is invertible. 
The adjoint operator is $\mcal{L}^\star = -\nabla \cdot (\bld{A} \nabla) - \bld{b}\cdot \nabla + c$. 
For an open subdomain $D \subset \mathbb{T}^d$, a function $u \in H^1(D)$ is a weak solution of $\mcal{L}^\star u = f$ on $D$ if 
\[
    a^\star(u,v) := \int_{D} \left(\bld{A} \nabla u \cdot \nabla v - \bld{b} \cdot \nabla u v + c u v \right)dx = \int_{D} f v dx
\]
for all $v \in H^1_0(D)$.

This setting isolates the interior regularity mechanism underlying asymptotic smoothness while avoiding any complications with boundary regularity.
The argument follows a similar strategy as \cite{bebendorf2006accelerating} by deriving a Caccioppoli inequality for $\mcal{L}^\star$-harmonic functions and iterating it on nested subdomains to obtain analytic growth bounds for arbitrary derivatives. 
These estimates are then applied to the Green function away from its singularity.
For $d \leq 3$, standard elliptic regularity together with the Hilbert-Schmidt embedding $H^2(\mathbb{T}^d) \hookrightarrow L^2(\mathbb{T}^d)$ imply the solution operator $\mcal{L}^{-1}: L^2(\mathbb{T}^d) \to L^2(\mathbb{T}^d)$ is Hilbert-Schmidt and admits a kernel in $L^2(\mathbb{T}^d \times \mathbb{T}^d)$. 
Rather than developing Green-function existence and regularity theory, we assume the standard off-diagonal properties stated explicitly in Proposition~\ref{prop: torusgreenfcn}. 
First, we establish the first-order Caccioppoli inequality. 

\begin{lemma}[Caccioppoli]
\label{lemma_caccio_zeroth}
    Let $D \subsetneq \mathbb{T}^d$ be open. 
    Let $K \subset D$ be a nonempty measurable set satisfying $\text{dist}(K, \partial D) > 0$. 
    Then for all $u \in H^1(D)$ such that $u$ is a weak solution of $\mcal{L}^\star u=0$, 
    \[\|\nabla u\|_{L^2(K)} \leq \left( \frac{4}{\text{dist}(K, \partial D)} \sqrt{\frac{\lambda_\text{max}}{\lambda_\text{min}}}  + \frac{\|\bld{b}\|}{\lambda_\text{min}} \right) \|u\|_{L^2(D)}.\]
    
\end{lemma}
\begin{proof}
Set $r := \text{dist}(K, \partial D)$.
Since $\overline{K}$ is a compact subset of $D$, 
let $\chi \in C_c^1(D)$ satisfy $0 \leq \chi \leq 1$, $\chi = 1$ in $K$ and $|\nabla \chi| \leq \frac{2}{r}$ in $D$. 
Since $\text{supp}(\chi) \subset D$, $\phi := \chi^2 u \in H_0^1(D)$ may be used as a test function in the weak formulation $a^\star(u,\phi) = 0$ as follows
\begin{equation*}
    \int_{D} \left(\bld{A}\nabla u \cdot \nabla(\chi^2u) - \bld{b}\cdot \nabla(u)\chi^2u + c u \chi^2 u \right) dx = 0.
\end{equation*}
Since $c \geq 0$ and $u^2 \chi^2 \geq 0$, we obtain that $\int_{D} \left(\bld{A}\nabla u \cdot \nabla(\chi^2u) - \bld{b}\cdot \nabla u\chi^2u \right) dx \leq 0.$
By property of $\bld{A}$ ($0 < \lambda_\text{min} \leq \lambda \leq \lambda_\text{max}$), 
\begin{align*}
    \int_{D} \chi^2 |\bld{A}^{\frac{1}{2}} \nabla u|^2 dx & = \biggl | \int_{D} \chi^2 (\nabla u)^T \bld{A} (\nabla u) dx \biggr | \nonumber \\
    & \leq \biggl | \int_{D} 2 \bld{A} \nabla u \cdot \nabla \chi u \chi dx \biggr | + \biggl | \int_{D} \bld{b} \cdot \nabla u u \chi^2 dx \biggr | \nonumber \\
    & := I_1 + I_2
\end{align*}
For $I_1$, 
\begin{align*}
    I_1 & \leq 2 \int_{D} \chi |u| \biggl | \bld{A}^\frac{1}{2} \nabla \chi \biggr | \biggl | \bld{A}^\frac{1}{2} \nabla u \biggr | dx 
     \leq 4 \frac{\sqrt{\lambda_\text{max}}}{r} \int_{D} |u| \left( \chi \biggl | \bld{A}^\frac{1}{2} \nabla u \biggr | \right) dx \nonumber \\
    & \leq 4 \frac{\sqrt{\lambda_\text{max}}}{r} \left( \int_{D} |\chi \bld{A}^\frac{1}{2} \nabla u|^2 dx \right)^\frac{1}{2} \|u\|_{L^2(D)}. \nonumber \\
\end{align*}
Also, for $I_2$, 
\begin{align*}
    I_2 & \leq \int_{D} \chi |u| \bigl | \langle \chi \nabla u , \bld{b} \rangle \bigr | dx  \leq \int_{D} \chi |u| \ \|\chi \nabla u\| \ \|\bld{b}\| dx  \nonumber \\
    & \leq  \frac{1}{\sqrt{\lambda_\text{min}}} \int_{D} |\chi u| \ \|\chi \bld{A}^\frac{1}{2} \nabla u \| \ \|\bld{b}\| dx  \leq \frac{\|\bld{b} \|}{\sqrt{\lambda_\text{min}}} \int_{D} |\chi u| \ || \chi \bld{A}^\frac{1}{2} \nabla u || dx \nonumber \\
    & \leq \frac{\|\bld{b}\|}{\sqrt{\lambda_\text{min}}}  ||\chi u||_{L^2(D)} \left( \int_D \|\chi \bld{A}^\frac{1}{2} \nabla u \|^2 dx \right)^\frac{1}{2}  \leq \frac{\|\bld{b}\|}{\sqrt{\lambda_\text{min}}}  ||u||_{L^2(D)} \left( \int_D \| \chi \bld{A}^\frac{1}{2} \nabla u \|^2 dx \right)^\frac{1}{2} \nonumber.
\end{align*}
Therefore, 
\begin{equation*}
    ||\chi \bld{A}^\frac{1}{2} \nabla u ||_{L^2(D)} \leq \left( 4 \frac{\sqrt{\lambda_\text{max}}}{r} + \frac{\|\bld{b}\|}{\sqrt{\lambda_\text{min}}} \right) \ ||u||_{L^2(D)},
\end{equation*} 
 and we obtain the estimate
\begin{equation*}
    ||\nabla u||_{L^2(K)} \leq ||\chi \nabla u ||_{L^2(D)} \leq \frac{1}{\sqrt{\lambda_\text{min}}} ||\chi \bld{A}^\frac{1}{2} \nabla u ||_{L^2(D)}.
\end{equation*}
 
\end{proof}

The following lemma proves via a nesting argument that arbitrary higher order derivatives of smooth $\mcal{L}^\star$-harmonic functions can be bounded with analytic growth bounds.

\begin{lemma}[Higher-order Caccioppoli]
    Let $D \subsetneq \mathbb{T}^d$ be open. 
    Let $K \subset D$ be a nonempty measurable set satisfying $\text{dist}(K, \partial D) > 0$. 
    Define $C_\mcal{L} = 4e \sqrt{\frac{\lambda_\text{max}}{\lambda_\text{min}}}  + \frac{e\|\bld{b}\| \text{diam}(\mathbb{T}^d)}{\lambda_\text{min}}$.
    Let $p \in \mathbb{N}$ such that $p\geq 1$, and suppose that $u \in C^\infty(D) \cap H^1(D)$ is a weak solution of $\mcal{L}^\star u=0$ in $D$.
    Then, for every multi-index $\alpha \in \mathbb{N}_0^d$ such that $|\alpha| = p$, 
\[
    \|\partial^\alpha u\|_{L^2(K)} 
    \leq \frac{\left(C_\mcal{L} \right)^p p!}{\text{dist}(K, \partial D)^p} \|u\|_{L^2(D)}.
\]  
\label{lemma:higherordercaccioppoli}
\end{lemma}
\begin{proof}
    Let $u \in C^\infty(D)$ be a weak solution of $\mcal{L}^\star u =0$.  
    Let $r:=\text{dist}(K, \partial D)$.
    Define $D_j = \{z \in \mathbb{T}^d: \text{dist}(z,K) < (1 - \frac{j}{p})r\}$ for $j = 0, \ldots, p-1$, and $D_p = K$. 
    Thus, 
    \[
        K = D_p \subset D_{p-1} \subset \ldots \subset D_1 \subset D_0 \subset D.
    \]  
By construction, $\text{dist}(D_j, \partial D_{j-1}) \geq \frac{r}{p}$ for $j = 1, \ldots, p$.
Since $u \in C^\infty(D) \subset H^1_\text{loc}(D)$, $\overline{D_j} \Subset D$ for $j = 1, \ldots, p-1$ implies that $\partial^\beta u \in H^1(D_j)$. 
Moreover, since $\mcal{L}^\star$ has constant coefficients, 
\[
    \mcal{L}^\star(\partial^\beta u) = \partial^\beta(\mcal{L}^\star u) = 0
\]
in the distributional sense for any $\beta \in \mathbb{N}_0^d$. 
Hence $\partial^\beta u \in H^1(D_j)$ is a weak solution of $\mcal{L}^\star \left( \partial^\beta u \right) = 0$ for any $\beta \in \mathbb{N}_0^d$ and for $j = 1, \ldots,p-1$.
    
Let $\alpha \in \mathbb{N}_0^d$ such that $|\alpha| = p$. 
Then, $\alpha = e_{i_1} + \ldots + e_{i_p}$ for some $i_1, \ldots, i_p \in \{1, \ldots, d\}$. 
We  iterate through the derivatives in a nested fashion in the following manner. 
    Define $v_0 := u$ and $v_j:= \partial_{i_j} v_{j-1}$ for $j = 1, \ldots, p$. 
    Then, $v_p = \partial^\alpha u$. 
    For $j \in \{1, \ldots, p\}$, $v_{j-1}$ is $\mcal{L^\star}$-harmonic, and Lemma~\ref{lemma_caccio_zeroth} on the domain pair $D_j \subset D_{j-1}$ gives
    \begin{align}
        \|v_j\|_{L^2(D_j)} & = \| \partial_{i_j} v_{j-1} \|_{L^2(D_j)} \leq \|\nabla v_{j-1}\|_{L^2(D_j)} \nonumber \\
        & \leq \left( \frac{4}{\text{dist}(D_j, \partial D_{j-1})} \sqrt{\frac{\lambda_\text{max}}{\lambda_\text{min}}}  + \frac{\|\bld{b}\|}{\lambda_\text{min}}\right) \|v_{j-1}\|_{L^2(D_{j-1})} \nonumber \\
        & \leq \left( \frac{4p}{r} \sqrt{\frac{\lambda_\text{max}}{\lambda_\text{min}}}  + \frac{\|\bld{b}\|}{\lambda_\text{min}}\right) \|v_{j-1}\|_{L^2(D_{j-1})}. \nonumber
    \end{align}
    By iteration, 
    \begin{align}
        \|\partial^\alpha  u \|_{L^2(K)} & = \| v_p\|_{L^2(D_p)} \nonumber \\
        & \leq 
         \left( \frac{4p}{r}\sqrt{\frac{\lambda_\text{max}}{\lambda_\text{min}}} + \frac{\|\bld{b}\|}{\lambda_\text{min}} \right)^p \|v_0 \|_{L^2(D_0)}  \nonumber \\
         & \leq \left( \frac{4p}{r}\sqrt{\frac{\lambda_\text{max}}{\lambda_\text{min}}} + \frac{p\text{diam}(\mathbb{T}^d) \|\bld{b}\|}{\lambda_\text{min}r} \right)^p \|u \|_{L^2(D)} \nonumber
    \end{align}
    where in the last inequality we used that $p \geq 1$ and $r \leq \text{diam}(\mathbb{T}^d)$.
    Thus, 
    \[
        \|\partial^\alpha  u \|_{L^2(K)} 
        \leq \frac{\left(4 \sqrt{\frac{\lambda_\text{max}}{\lambda_\text{min}}} + \frac{\text{diam}(\mathbb{T}^d)\|\bld{b}\|}{\lambda_\text{min}} \right)^p p^p}{r^p } \|u \|_{L^2(D_0)}.
    \]
    In particular, using that $p^p \leq e^p p!$ for integer $p \geq 1$,
    \begin{align}
        \|\partial^\alpha  u \|_{L^2(K)} 
        & \leq \frac{\left(4e \sqrt{\frac{\lambda_\text{max}}{\lambda_\text{min}}} + \frac{e\text{diam}(\mathbb{T}^d)\|\bld{b}\|}{\lambda_\text{min}} \right)^p p!}{r^p } \|u \|_{L^2(D_0)} \nonumber \\
        & \leq \frac{\left(4e \sqrt{\frac{\lambda_\text{max}}{\lambda_\text{min}}} + \frac{e\text{diam}(\mathbb{T}^d)\|\bld{b}\|}{\lambda_\text{min}} \right)^p p!}{r^p } \|u \|_{L^2(D)} \nonumber 
    \end{align}
    where in the final inequality we used that $D_0 \subset D$. 

\end{proof}

We now apply Lemma~\ref{lemma:higherordercaccioppoli} to the Green's function of $\mcal{L}$. 
We do not address the existence of the Green's function or its off-diagonal regularity here.
Instead, we assume the standard properties needed for the argument. 

\begin{proposition}
    Let $d \leq 3$. 
    Assume that the solution operator $\mcal{L}^{-1}$ admits a Green function $G \in L^2(\mathbb{T}^d \times \mathbb{T}^d)$ such that, for almost every $x \in \mathbb{T}^d$,
    \[
        G(x, \cdot) \in C^\infty(\mathbb{T}^d \backslash \{x\}),  
    \]
    and
    \[
        \mcal{L}^\star_y G(x,\cdot) = \delta_x \quad \text{in} \ \mcal{D}^\prime(\mathbb{T}^d).
    \]
    Under the above assumptions, the Green function $G$ is $L^2$-asymptotically smooth with respect to the $y$-variable.
    \label{prop: torusgreenfcn}
\end{proposition}

\begin{proof}
    Let $\mcal{X} \subset \mathbb{T}^d$ be measurable and $\mcal{Y} \subset \mathbb{T}^d$ a bounded connected Lipschitz domain such that $\text{dist}(\mcal{X}, \mcal{Y}) > 0$. 
    Define $\widehat{\mcal{Y}} := \{z \in \mathbb{T}^d: \text{dist}(z,\mcal{Y}) \leq \frac{1}{2}\text{dist}(\mcal{X},\mcal{Y})\}$, and a subdomain $D_\mcal{Y} := \{z \in \mathbb{T}^d: \text{dist}(z, \mcal{Y}) < \frac{\text{dist}(\mcal{X},\mcal{Y})}{2}\}$. 

    Fix $x \in \mcal{X}$ for which the assumptions on $G(x,\cdot)$ hold. 
    Then, since  $\overline{D_\mcal{Y}} \subset \mathbb{T}^d \backslash \{x\}$, $G(x,\cdot) \in C^\infty(\overline{D_\mcal{Y}}) \cap H^1(D_\mcal{Y})$ and satisfies $\mcal{L}^\star G(x,\cdot) = 0$ on $D_\mcal{Y}$. 
    By construction, $\text{dist}(\mcal{Y}, \partial D_\mcal{Y}) \geq \frac{\text{dist}(\mcal{X},\mcal{Y})}{2}$. 
    Hence for $|\alpha| \geq 1$, Lemma~\ref{lemma:higherordercaccioppoli} gives  
    \[
        \|\partial^\alpha G(x,\cdot)\|_{L^2(\mcal{Y})} \leq \frac{\left( 2C_\mcal{L} \right)^{|\alpha|} |\alpha|!}{\text{dist}(\mcal{X}, \mcal{Y})^{|\alpha|}} \|G(x,\cdot)\|_{L^2(\widehat{\mcal{Y}})}.
    \]  
    For $|\alpha| = 0$,  $\|G(x,\cdot)\|_{L^2(\mcal{Y})} \leq \|G(x, \cdot)\|_{L^2(\widehat{\mcal{Y}})}$. 
    Thus, take $C = 1$ and $\gamma = 2C_\mcal{L}$.
    
\end{proof}

\subsection{Numerical Examples}
\label{sec:numericalexamples}

We illustrate the approximation of solution operators of a one-dimensional linear elliptic PDE problem using Algorithm~\ref{adjointfreealg}.
In particular, consider
\[
    \begin{cases}
        -(a(x) u'(x))' + b(x) u'(x) + c(x)u(x) = f(x), \quad x \in (0,1), \\
        u(0) = u(1) = 0
    \end{cases}
\]
on the unit interval $\Omega = [0,1]$. 
We first consider the constant-coefficient problem $-u'' + bu' = f$. 
For $b \neq 0$, an analytical expression for the exact Green's function is given by
\[
    G(x,y) = \frac{(1-e^{b(x-1)})(1-e^{-by})}{b(1-e^{-b})}H(x-y) + \frac{(-1 + e^{bx})(e^{-by}-e^{-b})}{b(1-e^{-b})}H(y-x)
\]
where $H$ denotes the Heaviside step function.  
This example provides a useful benchmark because the continuous Green's function is known explicitly, allowing the reference operator to be constructed without a finite-element PDE solve.   
We subsequently consider four families of variable-coefficient problems, summarized in 
Figure~\ref{fig:numerical_experiments} below. 

\begin{figure}[!t]
    \centering
    \begin{minipage}[c]{0.475\linewidth}
        \centering
        \includegraphics[width=\linewidth]{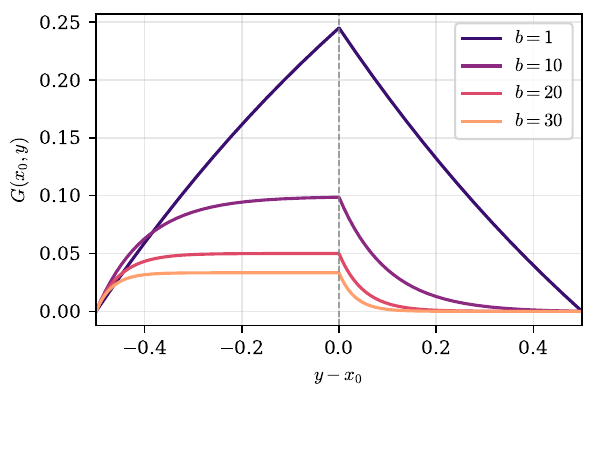}
    \end{minipage}
    \hfill
    \begin{minipage}[c]{0.475\linewidth}
        \centering
        \includegraphics[width=\linewidth]{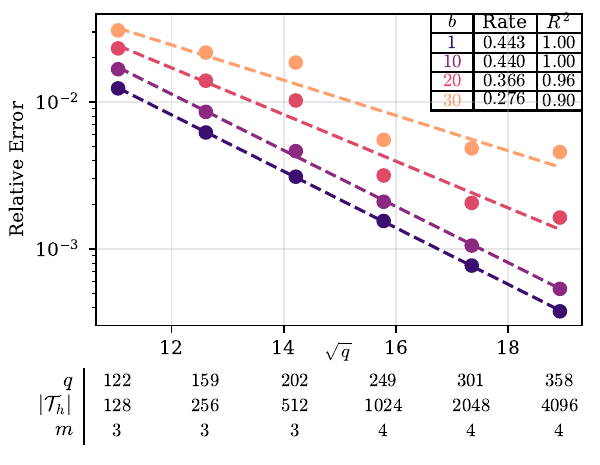}
    \end{minipage}

    \caption{
        Varying the advection strength $b$ in the one-dimensional
        advection-diffusion equation $-u''+bu'=f$.
        Left: Slices of the Green's function $G(x_0,y)$ for various
        advection strengths, centered at $x_0=\frac12$.
        Right: Approximation errors (dots) and root-exponential fits
        (dashed lines) as functions of the number of queries.
        Each data point corresponds to a single Gaussian realization.
    }
    \label{fig:constant-numerics}
\end{figure}
\begin{figure}[!t]
    \centering
    \includegraphics[
        width=\linewidth,
        height=0.62\textheight,
        keepaspectratio
    ]{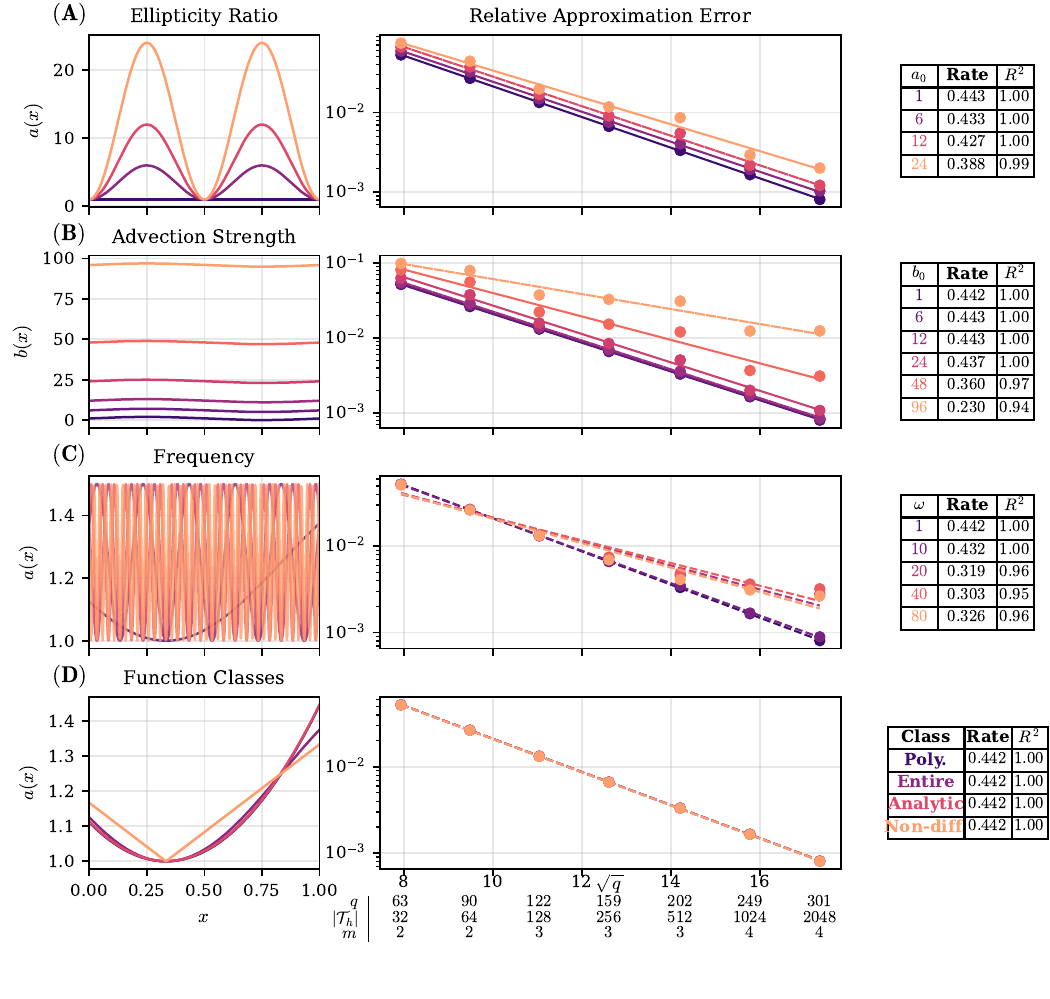}

    \caption{
        Variable coefficients experiments.
        Left column: coefficients. Right column: total approximation error
        (dots) and root-exponential fits (dashed lines) as functions of the
        number of queries. 
        Each data point corresponds to a single Gaussian realization.
        (\textbf{A}) Varying the ellipticity ratio with $a(x)=1+(a_0-1)\sin^2(2\pi x)$, $b(x) = 2 + 5x$ and $c(x) = 1 + x^2$. 
        (\textbf{B}) Varying the advection strength with $a(x) = 2 + 5x$, $b(x)=b_0+\sin(2\pi x)$ and $c(x) = 1 + x^2$.
        (\textbf{C}) Varying the oscillation frequency with $a(x)=\frac54-\frac14\cos\!\left(\omega\pi(x-\frac13)\right)$, $b(x) = 2 + 5x$ and $c(x) = 1 + x^2$.
        (\textbf{D}) Comparing polynomial, entire, analytic, and nondifferentiable diffusion coefficients $a(x)$, with $b(x) = 2 + 5x$ and $c(x) = 1 + x^2$.  
    }
    \label{fig:numerical_experiments}
\end{figure}
Since the infinite-dimensional operator $\mcal{A}$ cannot be represented exactly, we construct a reference matrix $\bld{A}_\text{ref}$ on a fine reference mesh $\mcal{T}_{h_\text{ref}}$. 
Both its input and output are represented in the $L^2$-normalized piecewise constant basis on this mesh. 
For the constant-coefficient problem, its entries are obtained by integrating the Green's function over pairs of reference cells. 
For the variable-coefficient problems, we solve the PDE using continuous piecewise linear finite elements on the reference mesh and project the solutions onto the piecewise constant reference space. 
We treat this reference discretization as ground truth. 
The precise matrix constructions are given in Appendix~\ref{appendix:numerics}. 
Algorithm~\ref{adjointfreealg} operates on coarser meshes $\mcal{T}_h$, which are chosen so that their piecewise constant spaces are nested in the reference space. 
Each query uses a piecewise constant forcing whose coefficients in the normalized coarse basis are independent standard normal random variables. 
The response is computed using the reference operator and projected onto the coarse piecewise constant space.

Let $\tilde{\bld{A}}_h$ be the matrix returned by the algorithm, and let $   \tilde{\bld{A}}_\text{ref}$ denote its representation in the normalized reference basis. 
We report the relative error
\[
    \frac{\| \bld{A}_\text{ref} - \tilde{\bld{A}}_\text{ref}\|_F}{\|\bld{A}_\text{ref}\|_F}.
\]
Since the reference basis is $L^2$-orthonormal, this is the relative Hilbert-Schmidt error for the corresponding operators on the reference space. 
The resolution of the reference mesh is chosen sufficiently fine to ensure further refinement do not considerably change the reported errors. 
Appendix~\ref{appendix:numerics} describes the relation between this reported error and the continuous Hilbert-Schmidt error. 

To motivate the choice of mesh size $h$, polynomial degree $m$ and number of queries $q$, we use the scalings predicted by Corollary~\ref{corquerybudget}.
In one dimension for fixed $\delta$ and $\rho$, the corollary gives a root-exponential error bound $\| \mcal{A} - \tilde{\mcal{A}} \|_{HS} = \mcal{O}(e^{-\kappa \sqrt{q}})$, together with parameter scalings 
\[
    h \approx e^{-c_h \sqrt{q}}, \quad m \approx c_m \sqrt{q},
\]
for positive constants $c_h$ and $c_m$. 
We prescribe a nested sequence of coarse meshes and choose $q$ and $m$ according to these scalings, up to integer rounding.
We additionally require $q \geq \text{max}_i k_i + 2$, where $k_i$ is the dimension of the $i$-th row subspace, as required by Algorithm~\ref{alg:amselgen}.
We use $c_h = 0.44, c_m = 0.3$ and $\eta = 0.3$ throughout the experiments. 

Figure~\ref{fig:constant-numerics} first examines the constant-coefficient problem $-u'' + bu' = f$ as the advection strength $b$ increases. As $b$ grows, the Green's function becomes increasingly asymmetric about the source point and shifts its weight in the direction of transport, as illustrated in the left panel.
The right panel shows the measured approximation errors with fits of the form $C_0e^{-\kappa \sqrt{q}}$. 
Over the range of queries tested, the observed decay is well-described by this root-exponential scaling, in agreement with the upper bounds in Theorem~\ref{bigtheorem}. 
Interestingly, at larger advection strengths, the errors develop visible plateaus.
These plateaus align with ranges of $q$ for which the rounded polynomial degree $m$ remains unchanged.
This suggests that local polynomial approximation and matrix approximation errors become more pronounced for stronger advection, thus exposing the discrete nature of $m$. 

Figure~\ref{fig:numerical_experiments} tests four families of variable-coefficient problems.
Problem \textbf{A} investigates how the performance of the algorithm depends on the ellipticity ratio $\frac{\lambda_\text{max}}{\lambda_\text{min}}$ by tuning a parameter $a_0$ in the diffusion coefficient. 
Increasing the ratio leads to slower observed convergence and makes the plateau structure increasingly visible.   
Problem \textbf{B} varies a parameter in the first order coefficient to increase the advection strength, which makes the problem increasingly non-self-adjoint. 
For weaker advection, the error decreases smoothly with mesh refinement, while for stronger advection the plateau structure becomes more apparent. 
Problem \textbf{C} defines the diffusion coefficient as a trigonometric function to test the effect of increasingly rapid coefficient oscillations through the frequency $\omega$ while preserving the same ellipticity ratio and qualitative regularity class. 
Lastly, problem $\textbf{D}$ compares diffusion coefficients from different regularity classes, chosen to have similar shapes and magnitudes. 
We consider the polynomial $a(x) = 1+\left(x-\frac13\right)^2$
, the entire function $a(x) = \frac54-\frac14 \cos\left(\pi\left(x-\frac13\right)\right)$
, the analytic function 
\[
    a(x)=1+\left(x-\frac13\right)^2+\frac{0.2r^2\left(x-\frac13\right)^2}{\left(x-\frac13\right)^2+r^2}, 
    \qquad r = 0.1
\]
with radius of convergence $r$ about $x = \frac{1}{3}$ due to complex poles $\frac{1}{3} \pm ir$, and the non-differentiable function 
$a(x)=1+\frac12\left|x-\frac13\right|$, respectively.
Over the range of queries tested, the approximation error shows surprisingly little sensitivity to the regularity class.
In particular, the non-differentiable coefficient exhibits convergence comparable to the polynomial, entire and analytic examples.

\section{Proof of Main Theorem \ref{bigtheorem}}
Theorem \ref{bigtheorem} reduces the linear operator approximation problem into a matrix approximation problem by discretizing the input and output function spaces with a piecewise constant finite element basis, making use of its local nature.
Let $J_h$ be the prolongation operator on this basis defined above.
Algorithm~\ref{alg:amselgen} approximates the projected operator
\[P_h \mcal{A} P_h  = J_h \bld{A}_h J_h^* \]
by constructing a structured approximation $\bld{\tilde{A}}_h$ of the dense matrix $\bld{A}_h := J_h^* \mcal{A} J_h \in \R^{|\mcal{T}_h| \times |\mcal{T}_h|}$.
Note that $\bld{A}_h$ is never constructed explicitly, but can be queried via the action of the original operator,
\[
    \bld{v} \mapsto \bld{A}_h \bld{v} = J_h^* \mcal{A} (J_h \bld{v}).
\]
The structured approximation then yields the finite-rank approximation $\tilde{\mcal{A}} = J_h \bld{\tilde{A}}_h  J_h^*$ to $\mcal{A}$ whose error we seek to bound in the Hilbert-Schmidt norm. 
By triangle inequality, this error decomposes into the discretization error and the algorithmic error
\begin{equation}
    \| \mcal{A} - \tilde{\mcal{A}}\|_{HS:L^2 \rightarrow L^2} 
        \leq \| \mcal{A} - P_h \mcal{A} P_h\|_\mathrm{HS:L^2 \rightarrow L^2} + \| P_h \mcal{A} P_h - \tilde{\mcal{A}}\|_\mathrm{HS:L^2 \rightarrow L^2}.
        \label{eqn:discerrdecomp}
\end{equation}
Both terms will be bounded by $\frac{\varepsilon}{2} \left(\| \mcal{A}\|_{HS:L^2 \rightarrow H^\rho} + \| \mcal{A}^*\|_{HS:L^2 \rightarrow H^\rho} \right)$. 
We bound the discretization error in Lemma~\ref{FEdiscretizationlemma} in Section~\ref{section:fediscretization} by leveraging standard finite element approximation estimates. 
Proving that we can bound the algorithmic error with a query-efficient approximation is more involved and requires several lemmas and propositions established in the subsequent sections, collectively leading to the proof of Theorem~\ref{bigtheorem} in Section \ref{sectheoremproof}.

The algorithmic error is directly controlled by the matrix approximation error, since
\begin{align}
    \| P_h \mcal{A} P_h - \tilde{\mcal{A}}\|_{HS:L^2 \rightarrow L^2}
        & = \|  J_h  (\bld{A}_h - \bld{\tilde{A}}_h) J_h^*\|_\mathrm{HS} \nonumber \\
        & \leq  \| J_h\|_\text{op} \| \bld{A}_h - \bld{\tilde{A}}_h\|_F \| J_h^*\|_\text{op} = \| \bld{A}_h - \bld{\tilde{A}}_h\|_F \nonumber
\end{align}
where the final equality follows from $\| J_h\|_\text{op}^2 = \| J_h^*\|_\text{op}^2 = \| J_h^* J_h\|_\text{op} = 1$. 
To obtain a query-efficient approximation using Algorithm~\ref{alg:amselgen}, we construct subspaces in which each row of the discretized operator is exponentially well-approximated. The key argument is that asympotically smooth kernels have separable polynomial approximations with an exponentially decaying remainder on $\eta-$admissible domains, defined below, when $\eta < \frac{1}{\gamma d}$.  
\begin{definition}[$\eta$-admissible domains]\label{def:admissible}
    Given $\eta > 0$,
    a pair $(\mcal{X}, \mcal{Y})$ of disjoint subsets $\mcal{X}, \mcal{Y} \subset \R^d$ is called $\eta$-admissible if they satisfy
    \[
        \eta \dist(\mcal{X}, \mcal{Y}) \geq \diam(\mcal{Y}).
    \]
    We will sometimes simply call a pair admissible, dropping the ``$\eta$-'', when it is obvious from context.
\end{definition}
\noindent
This separable polynomial approximation on $\eta$-admissible domains established in Section~\ref{expconsepexpsection} allows us to form low-dimensional row-wise bases based on a Whitney-like decomposition of the domain into $\eta$-admissible and non-admissible domains. 
We describe this domain decomposition, which we refer to as a hierarchical partition, in Section~\ref{hierSAdecompsection}. 
Although the number of degrees of freedom in the finite element discretization grows polynomially with the inverse error tolerance, 
these two results are combined in Section \ref{rowspaceconstructionsection} to construct subspaces $\mcal{V}_i \subset \R^{|\mcal{T}_h|}$ whose projection errors can be made arbitrarily small while their dimensions grow only polylogarithmically with the inverse error tolerance.  
Finally, Section \ref{sectheoremproof} applies Algorithm~\ref{alg:amselgen}, together with the discretization and matrix approximation estimates, to prove Theorem~\ref{bigtheorem}.
Table~\ref{tab:notation} provides a summary of notation used in the proof.

\begin{table}
    \centering
    \caption{A summary of our notation}
    \begin{tabular}{c c}
    \hline
    \hline
       Notation  & Description  \\
       \hline
       \hline
       $\bld{A}_{i,:}$ & $i$-th row of matrix $\bld{A}$\\
        $\rho$ & Sobolev regularity exponent \\
        $|\mcal{T}_h|$ & total number of mesh cells in the finite element mesh \\
        $\sigma_\#$ & shape-regularity constant of finite element mesh \\
        $c_\text{quasi}$ & quasi-uniformity constant of finite element mesh\\
        $N_{cc}$ & uniform number of convex components of spatial domain $\Omega$ \\
        $\mcal{V}_i$ & subspace for row $i$\\
        $\bld{V}_i$ & orthonormal basis of $\mcal{V}_i$ \\
        $\bld{P}[\bld{A}]$ & structured matrix where each row $i$ is the orthogonal projection of $\bld{A}_{i,:}$ onto $\mcal{V}_i$ \\
        $\bld{P}_i$ & the orthogonal projector onto $\mcal{V}_i$ \\
        $k$ & uniform bound on subspace dimension  \\
        $d$ & spatial dimension \\
        $\Omega$ & bounded spatial domain in $\R^d$ \\
        $\mcal{C}_n$ & $n$-th convex component of spatial domain $\Omega$ \\
        $\eta$ & admissibility constant\\
        $\delta$ & failure probability\\
        $h$ & finite element mesh size\\

        $\theta_i$ & $i$-th normalized piecewise constant basis function\\
        $\mcal{X}_i$ & mesh cell support of $i$-th finite element\\
        $\mcal{N}_{i,p}$ & arbitrary non-admissible cube from hierarchical partition of $\mcal{X}_i \times \Omega$ \\
        $\mcal{Y}_{i,j}$ & arbitrary admissible cube from hierarchical partition of $\mcal{X}_i \times \Omega$ \\
        $q$ & number of queries \\
        $n_h$ & hierarchical level chosen w.r.t $h$ \\
        $m$ & polynomial degree \\
        $C, \gamma$ & asymptotic smoothness constants \\
        $c_\mcal{H}$ & mesh refinement constant (see Assumption~1\ref{assumefemmesh}) \\
        $K_{d,\eta}$ & bound on non-admissible domains at each hierarchical level (see Equation~\ref{numberofnonadmissible})\\
        $K_{W,\text{adm}}$ & overlap bound for enlarged admissible cubes with an admissible leaf (see Lemma~\ref{CZadm}) \\
        $K_{W,\text{non}}$ & overlap bound for enlarged admissible cubes with a non-admissible leaf (see Lemma~\ref{CZnon})  \\
        $K_\text{non, FE}$ & bound on finite elements intersecting the non-admissible region (see Lemma~\ref{lemmanumofFEtocoverNA}) \\
        $K_\text{mesh}$ & mesh cell-volume lower-bound constant (see proof of Lemma~\ref{numberofmeshcells}) \\
    \hline
    \end{tabular}
    \label{tab:notation}
\end{table}

\subsection{Finite Element Approximation}
\label{section:fediscretization}

In this section, we bound the finite element discretization error term appearing in Equation~\ref{eqn:discerrdecomp} in Lemma~\ref{FEdiscretizationlemma}. 
To do this, we derive a local approximation property to establish a global $L^2$-approximation estimate in Lemma~\ref{lemmaglobalapproxfe} for the orthogonal projection $P_h$ onto the piecewise-constant finite element space $\mcal{V}_h^0$, and then lift this estimate to the Hilbert-Schmidt norm. 
\begin{lemma}
    There is $c$ such that 
    \[
        \| v - P_hv \|_{L^2(\Omega)} \leq c h^\rho \|v\|_{H^\rho(\Omega)}
    \]
    for all $\rho \in (0,1)$, all $v \in H^\rho(\Omega)$ and all $h \in \mcal{H}$. 
    \label{lemmaglobalapproxfe}
\end{lemma} 
\begin{proof}
    For each $\mcal{X}_i \in \mcal{T}_h$, the restriction of $P_h v$ to $\mcal{X}_i$ is the $L^2(\mcal{X}_i)$-orthogonal projection of $v|_{\mcal{X}_i}$ onto the space of degree-zero polynomials. Hence, $P_hv|_{\mcal{X}_i} = \frac{1}{\mathrm{vol}(\mcal{X}_i)} \int_{\mcal{X}_i} v(y) dy.$
    Expanding the square gives
    \[
        \|v - P_h v\|_{L^2(\mcal{X}_i)}^2 
            = \int_{\mcal{X}_i} v(x)^2 dx - \frac{1}{|\mcal{X}_i|} \left( \int_{\mcal{X}_i} v(y) dy \right)^2 
            = \frac{1}{2|\mcal{X}_i|} \int_{\mcal{X}_i} \int_{\mcal{X}_i} |v(x) - v(y)|^2 dx dy.
    \]
    For $\rho \in (0,1)$, the following holds
    \[
        \|v - P_h v\|_{L^2(\mcal{X}_i)}^2  \leq \frac{h_{\mcal{X}_i}^{d + 2 \rho}}{2 |\mcal{X}_i|} \int_{\mcal{X}_i} \int_{\mcal{X}_i} \frac{|v(x) - v(y)|^2}{|x-y|^{d + 2 \rho}} dx dy
    \]
    since $h_{\mcal{X}_i}$ is the diameter of the mesh cell.
    The proof of Lemma~\ref{numberofmeshcells} establishes that the volume of an arbitrary mesh cell of $\mathcal{T}_h$ is lower bounded by the quantity $|\mcal{X}_i| \geq K_\text{mesh} h^d$. Recalling that $h_{\mcal{X}_i} \leq h$, this lower bound together with the definition of the fractional Gagliardo semi-norm (see Section 2 of \cite{di2012hitchhikers}) yields
    \[
         \|v - P_h v\|_{L^2(\mcal{X}_i)}^2 \leq \frac{1}{2 K_\text{mesh}}h^{2\rho} |v|^2_{H^\rho(\mcal{X}_i)}.
    \]
    Set $c := \frac{1}{\sqrt{2 K_\text{mesh}}}$, and thus
    \[
        \| v - P_h v\|_{L^2(\mcal{X}_i)} \leq c h^\rho \|v\|_{H^\rho(\mcal{X}_i)}
    \]
    for all $\rho \in (0,1)$, all $v \in H^\rho(\mcal{X}_i)$, each $i = 1, \dots, |\mcal{T}_h|$ and all $h \in \mcal{H}$. 
    Since the mesh exactly partitions $\Omega$, squaring and summing over $i = 1, \dots, |\mcal{T}_h|$ gives 
    \[
        \| v - P_hv \|^2_{L^2(\Omega)} 
            = \sum_{i=1}^{|\mcal{T}_h|}  \| v - P_h v\|_{L^2(\mcal{X}_i)}^2 
            \leq c^2 h^{2\rho} \sum_{i=1}^{|\mcal{T}_h|} \|v\|_{H^\rho(\mcal{X}_i)}^2
            \leq c^2 h^{2\rho} \|v\|_{H^\rho(\Omega)}^2.
    \]
    
\end{proof}

We now apply this projection estimate in Lemma~\ref{lemmaglobalapproxfe} to the range of $\mcal{A}$ and $\mcal{A}^\star$ to obtain the required Hilbert-Schmidt bound. 

\begin{lemma}[Discretization Error]
    \label{FEdiscretizationlemma}
    Let Assumptions~1\ref{assumedomaincondition}, 1 \ref{assumefemmesh} and 2\ref{assumeglobalregularity} hold. 
    Let $\rho \in (0,1)$ with $\rho < s - \frac{d}{2}$. 
    Then, there exists a constant $c$ such that 
    \[
        \| \mcal{A} - P_h \mcal{A} P_h\|_{HS:L^2 \rightarrow L^2} 
            \leq c h^\rho \left( \| \mcal{A}\|_{HS:L^2 \rightarrow H^\rho} + \| \mcal{A}^*\|_{HS:L^2 \rightarrow H^\rho} \right)
    \]
     for all $h \in \mcal{H}$.
\end{lemma}

\begin{proof}
Let $h \in \mcal{H}$. 
Using that $\rho \in (0,1)$, there is a constant $c$ via Lemma~\ref{lemmaglobalapproxfe} such that the $L^2(\Omega)$-orthogonal projection satisfies 
    \begin{equation*}
        \| (I - P_h)w\|_{L^2(\Omega)} 
           \leq ch^{\rho} \| w\|_{H^\rho(\Omega)} \nonumber
    \end{equation*}
    for all $w \in H^\rho(\Omega)$. 
    Letting $\{ e_j \}_{j=1}^{\infty}$ be an orthonormal basis for $L^2(\Omega)$, we then have
    \begin{align}
        \| (I-P_h) \mcal{A}\|_{HS: L^2 \rightarrow L^2}^2 
            & = \sum_{j=1}^\infty \| (I - P_h) \mcal{A} e_j\|^2_{L^2(\Omega)}
            \nonumber \\
            & \leq \left( c h^\rho \right)^2 \sum_{j=1}^\infty \| i_{H^s \rightarrow H^\rho} \circ \mcal{A} e_j\|^2_{H^\rho(\Omega)} 
            = \left( c h^\rho\right)^2  \| \mcal{A}\|^2_{HS:L^2 \rightarrow H^\rho}  \nonumber
    \end{align}
    using the fact that the Sobolev embedding $i:H^s(\Omega) \hookrightarrow H^\rho(\Omega)$ is Hilbert-Schmidt (see Lemma 2.3 in \cite{kovacs2023hilbert}) and $\rho < s - \frac{d}{2}$.
    The adjoint operator also has a similar bound, 
    \[
        \| (I-P_h) \mcal{A}^*\|^2_{HS:L^2 \rightarrow L^2}  
        \leq  (ch^\rho)^2 \| \mcal{A}^*\|^2_{HS:L^2 \rightarrow H^\rho}.
    \]
    By the triangle inequality and the fact that $\| P_h\|_\text{op} = 1$, we obtain
    \begin{align} 
        \| \mcal{A} - P_h \mcal{A} P_h\|_{HS:L^2 \rightarrow L^2} 
            & \leq \| \mcal{A} - P_h \mcal{A}\|_{HS:L^2 \rightarrow L^2} + \| P_h \mcal{A} - P_h \mcal{A} P_h\|_{HS:L^2 \rightarrow L^2} \nonumber \\
            & \leq \| (I - P_h) \mcal{A}\|_{HS:L^2 \rightarrow L^2} + \| (I-P_h) \mcal{A}^*\|_{HS:L^2 \rightarrow L^2}
            \nonumber \\
            & \leq c h^\rho \left( \| \mcal{A}\|_{HS:L^2 \rightarrow H^\rho} + \| \mcal{A}^*\|_{HS:L^2 \rightarrow H^\rho} \right).\nonumber
    \end{align} 
\end{proof}
\subsection{Locally Exponentially Convergent Separable Polynomial Approximations}
\label{expconsepexpsection}
We first consider an $\eta$-admissible domain $\mcal{X} \times \mcal{Y}$ and show that asymptotically smooth kernels admit exponentially accurate polynomial approximations in the $y$-variable on such domains.  
Proposition~\ref{expconvsepexpansion} establishes this local approximation property by constructing an explicit polynomial of total degree at most $m$. 
Corollary~\ref{sepexpansionwithL2orthogonal} then uses the best-approximation-property of the $L^2(\mcal{Y})$-orthogonal projection to show that projection onto the space of all  polynomials with total degree $\leq m$ defined on $\mcal{Y}$, denoted by 
\[
    \mcal{P}^d_{m}(\mcal{Y}) 
        := \{p:\mcal{Y} \subset \R^d \rightarrow \R: p = \sum_{|\alpha| \leq m} a_\alpha y^\alpha, a_\alpha \in \R\},
\] achieves the same accuracy. 
This latter formulation provides the finite-dimensional local approximation spaces used in Section~\ref{rowspaceconstructionsection} to construct the row subspaces. 
We state Proposition~\ref{expconvsepexpansion} first. 

\begin{proposition}
\label{expconvsepexpansion}
    Let $\Omega \subset \R^d$ satisfy Assumption~1\ref{assumedomaincondition} and a kernel function $\kappa(\cdot,\cdot): \Omega \times \Omega \rightarrow \R$ be $L^2$-asymptotically smooth in the $y$-variable with constants $C,\gamma$, as defined in Definition~\ref{def:l2asymptoticsmoothness}. 
    Let $\eta < \frac{1}{\gamma d}$ and $\mcal{X} \subset \Omega$.
    Let $\mcal{Y} \subset \Omega$ be a bounded, connected domain with local Lipschitz boundary such that $\mcal{X} \times \mcal{Y}$ is $\eta$-admissible (see Definition~\ref{def:admissible}).
    Define the enlargement of $\mcal{Y}$ with respect to $\mcal{X}$ by $\widehat{\mcal{Y}} = \{z \in \Omega: \dist(z,\mcal{Y}) \leq \frac{1}{2}\dist(\mcal{X},\mcal{Y})\}.$
    For any $m \in \mathbb{N}$, the following approximation $\tilde{T}_m[\kappa](x,\cdot)$ is well-defined
    \[
        \tilde{T}_m[\kappa](x,y)  
            = \frac{1}{|\mcal{Y}|} \int_\mcal{Y} \sum_{|\alpha| \leq m} \frac{(y-y_0)^\alpha}{\alpha !} \partial_y^\alpha \kappa(x,y_0) d y_0
    \] 
    as a polynomial of total degree $\leq m$ on $\mcal{Y}$ and satisfies
    \[
    \| \kappa(x,\cdot) - \tilde{T}_m[\kappa](x,\cdot)\|_{L^2(\mcal{Y})} 
        \leq C \frac{(\gamma \eta d)^m}{ 1 - \gamma \eta d} \| \kappa(x,\cdot)\|_{L^2(\widehat{\mcal{Y}})} \quad \text{for almost every} \ x \in \mcal{X}.
    \]  \end{proposition} 
The proof of Proposition~\ref{expconvsepexpansion} relies on two lemmas. The approximation constructed in Proposition~\ref{expconvsepexpansion} relies on averaging Taylor expansions of the kernel over $\mcal{Y}$. 
Lemma~\ref{analyticitylemma} establishes that asymptotic smoothness guarantees both the required analyticity and the absolute integrability needed to interchange the Taylor series with this averaging integral over its expansion center. 
\begin{lemma}
\label{analyticitylemma}
Let $\Omega \subset \R^d$ satisfy Assumption~1\ref{assumedomaincondition} and let the kernel function $\kappa(\cdot,\cdot): \Omega \times \Omega \rightarrow \R$ be $L^2$-asymptotically smooth in the $y$-variable with constants $C,\gamma$, as defined in Definition~\ref{def:l2asymptoticsmoothness}. 
Let $\eta < \frac{1}{\gamma d}$ and $\mcal{X} \subset \Omega$. 
Let $\mcal{Y} \subset \Omega$ be a bounded, connected domain with local Lipschitz boundary such that $\mcal{X} \times \mcal{Y}$ is $\eta$-admissible. 
Then, for almost every $x \in \mcal{X}$, the kernel $\kappa(x,\cdot) \in C^\omega(\mcal{Y}^\circ)$ is analytic in the interior $\mcal{Y}^\circ$.  Furthermore, the function $\kappa(x, \cdot)$ agrees to its Taylor series about any $y_0 \in \mcal{Y}^\circ$ on $\mcal{Y}^\circ$, i.e.
\[
    \kappa(x,y) 
    = \sum_{\alpha} \frac{(y-y_0)^\alpha}{\alpha!} \partial_y^\alpha \kappa(x, y_0) 
    \quad \forall  \ y, y_0 \in \mcal{Y}^\circ \quad \text{for almost every } x \in \mcal{X}.
\]

Moreover, for almost every $x \in \mcal{X}$ and every $y \in \mcal{Y}$, the following holds 
\[
\sum_{\alpha \in \mathbb{N}_0^d} \int_\mcal{Y} \bigg| \frac{(y - y_0)^\alpha}{\alpha!} \partial_y^\alpha \kappa(x,y_0) \bigg| dy_0 < \infty.
\]
 
\end{lemma}

\begin{proof}[Lemma~\ref{analyticitylemma}]

First, we prove that there exist constants $M,s > 0$ such that 
$|\partial^\beta_y \kappa(x,y)| \leq M |\beta|! s^{-|\beta|} \ \ \ \forall \beta \in \mathbb{N}_0^d$ for all $y \in \mcal{Y}$ 
for almost every $x \in \mcal{X}$. 
This derivative bound implies that for every compact set $S \subset \mcal{Y}^\circ$,  $|\partial^\beta_y \kappa(x,y)| \leq M |\beta|! s^{-|\beta|} \ \ \ \forall \beta$ 
for every $y \in S$. 
Therefore, we can conclude that $\kappa(x, \cdot) \in C^\omega(\mcal{Y}^\circ)$ by Definition 2.16 and Theorem 2.17 in \cite{renardy2004introduction}.

Choose $l := \lfloor \frac{d}{2} \rfloor + 1 > \frac{d}{2}$. Let $\beta \in \mathbb{N}_0^d$. 
Let $y \in \mcal{Y}$. 
Since $\mcal{Y}$ has local Lipschitz boundary and its boundary is compact, it satisfies the strong local Lipschitz condition in Section 4.9 of \cite{adams2003sobolev}, as noted in Remark \ref{rmkconvexlocallipschitz}.
The Sobolev Imbedding theorem (in Section 4.12 Part II of \cite{adams2003sobolev}) asserts that $H^l(\mcal{Y}) \subset C(\overline{\mcal{Y}})$. Thus, for almost every $x \in \mcal{X}$, the following pointwise estimate with a Sobolev embedding constant $C_\text{SE}$ depending on $l$ holds:
\begin{align}
    |\partial_y^\beta \kappa(x,y)| & \leq \| \partial_y^\beta \kappa(x,\cdot)\|_{C(\overline{\mcal{Y}})}
        \leq C_\text{SE} \| \partial^\beta_y \kappa(x,\cdot)\|_{H^l(\mcal{Y})} 
        \leq C_\text{SE} \sum_{|\alpha| \leq l} \| \partial^{\beta + \alpha} \kappa(x,\cdot)\|_{L^2(\mcal{Y})} \nonumber \\
        & \leq C_\text{SE} \sum_{|\alpha| \leq l} C \left( \frac{\gamma}{\dist(\mcal{X},\mcal{Y})}\right)^{|\alpha + \beta|} |\alpha + \beta|! \| \kappa(x,\cdot)\|_{L^2(\widehat{\mcal{Y}})} \nonumber
\end{align}
using the $L^2$-asymptotic smoothness assumption (Definition~\ref{def:l2asymptoticsmoothness}).

Define a constant $\lambda := \frac{1}{2} + \frac{1}{2\gamma \eta d}$ so that $1 < \lambda < \frac{1}{\gamma \eta d}$.
Since $\lambda > 1$, exponential growth $\lambda^{|\beta|}$ dominates polynomial growth with respect to $|\beta|$, and there exists a constant $C_{l,\lambda} > 0$ such that 
\[
    |\alpha + \beta|! 
    \leq (l + |\beta|)!
    \leq (l + |\beta|)^l |\beta|!
    \leq C_{l,\lambda} \lambda^{|\beta|} |\beta|!
    \qquad
    \forall |\alpha| \leq l, \quad \forall \beta \in \mathbb{N}_0^d.
\]
Using this, we obtain
\begin{align}
    |\partial_y^\beta \kappa(x,y)|
        \leq  \underbrace{C_\text{SE} C C_{l,\lambda} \| \kappa(x,\cdot)\|_{L^2(\widehat{\mcal{Y}})}    \sum_{|\alpha| \leq l} \bigg( \frac{\gamma}{\dist(\mcal{X},\mcal{Y})}\bigg)^{|\alpha|} }_{=: M} |\beta|!  \bigg( \underbrace{\frac{ \lambda \gamma}{\dist(\mcal{X},\mcal{Y})}}_{=:s^{-1}}\bigg)^{|\beta|}. \nonumber 
\end{align}
Thus, $\kappa(x,\cdot) \in C^\omega(\mcal{Y}^\circ)$ for almost every $x \in \mcal{X}$.

Let $y_0 \in \mcal{Y}^\circ$. We now prove that the function $\kappa(x,\cdot)$ agrees to its Taylor series about $y_0 \in \mcal{Y}^\circ$ on $\mcal{Y}^\circ$ for almost every $x \in \mcal{X}$. Let $g_x(y) =   \sum_{\alpha} \frac{(y-y_0)^\alpha}{\alpha!} \partial^\alpha_y\kappa(x,y_0)$ denote the Taylor series of $\kappa(x,\cdot)$ about $y_0$ for $ y \in \mcal{Y}^\circ$.
Using the observation that $\sum_{|\alpha| = l} \frac{l!}{\alpha!} = d^l$, we find that
for all $y \in \mcal{Y}^\circ$, the absolute sum
\begin{align}
    \sum_{\alpha} \frac{|(y - y_0)^\alpha|}{\alpha!} |\partial^\alpha_y \kappa(x,y_0)|
    & \leq \sum_{\alpha} \frac{\diam(\mcal{Y})^{|\alpha|}}{\alpha !} M s^{-|\alpha|} |\alpha| ! 
    = M \sum_{l=0}^\infty \left(\frac{\diam(\mcal{Y})}{s}\right)^{l} \sum_{|\alpha| = l} \frac{l !}{\alpha !} \nonumber \\
    & \leq M \sum_{l=0}^\infty  \left( \frac{\lambda d \diam(\mcal{Y}) \gamma}{\dist(\mcal{X},\mcal{Y})}\right)^{l} 
    < M \sum_{l=0}^\infty  (\lambda d \eta \gamma)^{l}< \infty \nonumber 
\end{align}
converges as a geometric series
by the choice of $\eta$ and $\lambda$.
Real-analyticity of both $g_x(\cdot)$ and $\kappa(x,\cdot)$ on $\mcal{Y}^\circ$ and the fact that $\kappa(x,y) = g_x(y)$ on some neighborhood of $y_0$ in $\mcal{Y}^\circ$ allows us to conclude via the identity theorem of real-analytic functions (see 9.4.2 in \cite{dieudonne1969foundations}) that these real-analytic functions must agree on the entire open connected domain $\mcal{Y}^\circ$, i.e. $\kappa(x,y) = g_x(y)$ for all $y \in \mcal{Y}^\circ$. 

Following a similar argument with the pointwise analytic growth bound on the derivatives, observe that for almost every $x \in \mcal{X}$ and $y \in \mcal{Y}$
\begin{align}
    \sum_{\alpha \in \mathbb{N}_0^d} \int_\mcal{Y} \bigg| \frac{(y-y_0)^\alpha}{\alpha !} \partial_y^\alpha \kappa(x, y_0) \bigg| dy_0 \leq M |\mcal{Y}| \sum_{l = 0}^\infty (\lambda d\eta \gamma)^l < \infty. \nonumber
\end{align}

\end{proof}

The next lemma is a trivial estimate used in the proof of Proposition~\ref{expconvsepexpansion}. 
\begin{lemma}
    Let $\mcal{Y} \subset \R^d$ be a bounded, measurable domain. 
    Then, for every $y \in \mcal{Y}$
    \[
        \sum_{|\alpha| = l} \binom{l}{\alpha} \left(\int_\mcal{Y} |y-y_0|^{2 \alpha} dy_0 \right)^\frac{1}{2} 
            \leq |\mcal{Y}|^\frac{1}{2} \diam(\mcal{Y})^l d^l.
    \]
    \label{lemmadiameter}
\end{lemma}
\begin{proof}
    Let $y \in \mcal{Y}$. 
    Let $\alpha \in \mathbb{N}_0^d$ such that $|\alpha| = l$.
    For any $y_0 \in \mcal{Y}$, observe that
    \[
        |y - y_0|^\alpha 
            = \prod_{i=1}^d |y_i - y_{0_i}|^{\alpha_i} \leq \prod_{i=1}^d \diam(\mcal{Y})^{\alpha_i} 
            = \diam(\mcal{Y})^l. 
    \]

    Therefore, the summation
    \begin{align}
        \sum_{|\alpha| = l} \binom{l}{\alpha} \left(\int_\mcal{Y} |y-y_0|^{2 \alpha} dy_0 \right)^\frac{1}{2} 
            & \leq \diam(\mcal{Y})^l \sum_{|\alpha| = l} \binom{l}{\alpha} \left( \int_\mcal{Y} 1 dy_0 \right)^\frac{1}{2} \nonumber \\
            &  = |\mcal{Y}|^\frac{1}{2} \diam(\mcal{Y})^l \sum_{|\alpha| = l} \binom{l}{\alpha} 
            = |\mcal{Y}|^\frac{1}{2} \diam(\mcal{Y})^l d^l. \nonumber
    \end{align} \end{proof}
The proof of Proposition~\ref{expconvsepexpansion} now follows.

\begin{proof}[Proposition~\ref{expconvsepexpansion}]
We first note that the domain $\mcal{Y}$ with local Lipschitz boundary is Lebesgue-measurable.
Let $m \in \mathbb{N}$.
For almost every $x \in \mcal{X}$, the approximation error is expressed as 
\begin{align}
    \| \kappa(x,\cdot) - \tilde{T}_m[\kappa](x,\cdot)\|_{L^2(\mcal{Y})}^2 
        & = \int_\mcal{Y} \bigg | \kappa(x,y) - \frac{1}{|\mcal{Y}|} \int_\mcal{Y} \sum_{|\alpha| \leq m } \frac{(y-y_0)^\alpha}{\alpha !} \partial_y^\alpha \kappa(x,y_0) dy_0 \bigg |^2 dy \nonumber \\
        & = \int_\mcal{Y} \bigg | \frac{1}{|\mcal{Y}|} \int_\mcal{Y} \left[ \kappa(x,y) - \sum_{|\alpha| \leq m} \frac{(y-y_0)^\alpha}{\alpha !} \partial_y^\alpha \kappa(x,y_0) \right] dy_0 \bigg|^2 dy. \nonumber \\
\intertext{We note that the boundary $\partial \mcal{Y}$ is a set of measure zero and Lemma~\ref{analyticitylemma} can be applied to conclude  }
        \| \kappa(x,\cdot) - \tilde{T}_m[\kappa](x,\cdot)\|_{L^2(\mcal{Y})}^2 
            & = \int_\mcal{Y} \bigg | \frac{1}{|\mcal{Y}|} \int_\mcal{Y} \sum_{|\alpha| > m} \frac{(y-y_0)^\alpha}{\alpha !} \partial_y^\alpha \kappa(x,y_0) d y_0 \bigg|^2 dy. \nonumber \\
\intertext{Lemma~\ref{analyticitylemma} allows us to leverage Fubini's Theorem over $\mathbb{N}_0^d \times \mcal{Y}$ to yield}
        \| \kappa(x,\cdot) - \tilde{T}_m[\kappa](x,\cdot)\|_{L^2(\mcal{Y})}^2 
            & = \int_\mcal{Y} \bigg | \frac{1}{|\mcal{Y}|} \sum_{|\alpha| > m} \frac{1}{\alpha !} \int_\mcal{Y} (y-y_0)^\alpha \partial_y^\alpha \kappa(x,y_0) dy_0 \bigg|^2 dy \nonumber \\
            & \leq \int_\mcal{Y} \left( \frac{1}{|\mcal{Y}|} \sum_{|\alpha| > m} \frac{1}{\alpha !} \int_\mcal{Y} \big|(y-y_0)^\alpha \partial_y^\alpha \kappa(x,y_0)\big| dy_0 \right)^2 dy. \nonumber \\
\intertext{Applying Cauchy-Schwarz's inequality, we obtain}
        \| \kappa(x,\cdot) - \tilde{T}_m[\kappa](x,\cdot)\|_{L^2(\mcal{Y})}^2 
            & \leq \int_\mcal{Y} \left( \frac{1}{|\mcal{Y}|} \sum_{|\alpha| > m} \frac{1}{\alpha!} \left( \int_\mcal{Y} |y-y_0|^{2\alpha} dy_0\right)^\frac{1}{2} \| \partial_y^\alpha \kappa(x, \cdot)\|_{L^2(\mcal{Y})} \right)^2 dy \nonumber 
\end{align}
Now, we exploit the $L^2$-asymptotic smoothness assumption in the $y$-variable and make use of the fact that the summation of non-negative terms over $|\alpha| > m$ is bounded above by the summation over $|\alpha| \geq m$ to further bound $\| \kappa(x,\cdot) - \tilde{T}_m[\kappa](x,\cdot)\|_{L^2(\mcal{Y})}^2$ 

\begin{equation*}
    \leq \int_\mcal{Y}\left(\frac{1}{|\mcal{Y}|} \sum_{|\alpha| \geq m} \frac{1}{\alpha !} \left( \int_\mcal{Y} |y-y_0|^{2 \alpha} dy_0 \right)^\frac{1}{2} \frac{C \gamma^{|\alpha|} |\alpha|!}{\dist(\mcal{X},\mcal{Y})^{|\alpha|}} \| \kappa(x,\cdot )\|_{L^2(\widehat{\mcal{Y}})}\right)^2  dy
\end{equation*}
\begin{equation*}
    = \int_\mcal{Y} \left( \frac{\| \kappa(x,\cdot)\|_{L^2(\widehat{\mcal{Y}})}}{|\mcal{Y}|} \sum_{l=m}^\infty \frac{C \gamma^l}{\dist(\mcal{X},\mcal{Y})^l} \sum_{|\alpha| = l} \binom{l}{\alpha} \left( \int_\mcal{Y} |y-y_0|^{2 \alpha} dy_0 \right)^\frac{1}{2} \right)^2 dy. 
\end{equation*} 
Leveraging Lemma~\ref{lemmadiameter} here to bound the summation over $|\alpha| = l$ gives
\begin{align}
    \| \kappa(x,\cdot) - \tilde{T}_m[\kappa](x,\cdot)\|_{L^2(\mcal{Y})}^2 
        & \leq \int_\mcal{Y} \left( \frac{\| \kappa(x,\cdot)\|_{L^2(\widehat{\mcal{Y}})}}{|\mcal{Y}|} \sum_{l=m}^\infty \frac{C \gamma^l}{\dist(\mcal{X},\mcal{Y})^l} |\mcal{Y}|^\frac{1}{2} (\diam(\mcal{Y}) d)^l \right)^2 dy \nonumber \\
        & = \int_\mcal{Y} \frac{C^2 \| \kappa(x,\cdot)\|_{L^2(\widehat{\mcal{Y}})}^2}{|\mcal{Y}|} \left( \sum_{l=m}^\infty \frac{\gamma^l d^l \diam(\mcal{Y})^l}{\dist(\mcal{X},\mcal{Y})^l}\right)^2 dy. \nonumber 
\end{align}
Here we use the admissibility condition $\eta \geq \frac{\diam(\mcal{Y})}{\dist(\mcal{X},\mcal{Y})}$ to render
\begin{equation*}
   \| \kappa(x,\cdot) - \tilde{T}_m[\kappa](x,\cdot)\|_{L^2(\mcal{Y})}^2 
   \leq \frac{C^2 \| \kappa(x,\cdot)\|^2_{L^2(\widehat{\mcal{Y}})}}{|\mcal{Y}|} \int_\mcal{Y} \left( \sum_{l=m}^\infty (\gamma \eta d)^l \right)^2 dy
\end{equation*}
For ease of notation, denote  $\theta:= \gamma \eta d$. The choice of $\eta$ enforces the condition $\theta < 1$ and allows us to bound the approximation error by a convergent geometric series as follows
\begin{equation*}
    \| \kappa(x,\cdot) - \tilde{T}_m[\kappa](x,\cdot)\|_{L^2(\mcal{Y})}^2 
        \leq \frac{C^2}{|\mcal{Y}|} \| \kappa(x,\cdot)\|_{L^2(\widehat{\mcal{Y}})}^2 \int_\mcal{Y} \left( \frac{\theta^m}{1-\theta}\right)^2 dy 
        = C^2 \| \kappa(x,\cdot)\|^2_{L^2(\widehat{\mcal{Y}})} \frac{\theta^{2m}}{(1-\theta)^2}. \nonumber
\end{equation*} 
Taking square roots yields
\[ 
\| \kappa(x,\cdot) - \tilde{T}_m[\kappa](x,\cdot)\|_{L^2(\mcal{Y})} 
    \leq  C \frac{\theta^{m}}{1-\theta}\| \kappa(x,\cdot)\|_{L^2(\widehat{\mcal{Y}})} .
\]
\end{proof}

Corollary~\ref{sepexpansionwithL2orthogonal} uses the best-approximation-property of the $L^2(\mcal{Y})$-orthogonal projection to show that the kernel can be approximated to arbitrary accuracy by projection onto $\mcal{P}_m^d(\mcal{Y})$, where the required polynomial degree $m$ grows only logarithmically with the inverse error tolerance. 
\begin{corollary}
\label{sepexpansionwithL2orthogonal}
Assume the setting of Proposition~\ref{expconvsepexpansion}. 
For all $\varepsilon^\prime > 0$ and $m \in \mathbb{N}$ such that $m \geq m_{\varepsilon^\prime} := \max \Bigg\{ 1, \Bigg\lceil \frac{\log{\left( \frac{C}{\varepsilon^\prime (1-\gamma \eta d)}\right)}}{-\log{(\gamma \eta d)}} \Bigg\rceil \Bigg\}$, the $L^2(\mcal{Y})$-orthogonal projection of $\kappa(x,\cdot)$ onto $\mcal{P}^d_{m}(\mcal{Y})$ expressed with an $L^2(\mcal{Y})$-orthonormal basis $\{\phi_\alpha \}_{|\alpha| \leq m}$ of $\mcal{P}_{m}^d(\mcal{Y})$  as    
\[
    p^\star_m(x, \cdot) 
        = \sum_{|\alpha| \leq m} \langle \kappa(x, \cdot), \phi_\alpha \rangle_{L^2(\mcal{Y})} \phi_\alpha
\] 
satisfies
\[
    \| \kappa(x,\cdot) - p^\star_m(x,\cdot)\|_{L^2(\mcal{Y })} 
        \leq \varepsilon^\prime \| \kappa(x,\cdot)\|_{L^2(\widehat{\mcal{Y}})}
\]
for almost every $x \in \mcal{X}$.   
\end{corollary}

\begin{proof}
Let $\varepsilon^\prime > 0$. 
Let $m$ be a nonnegative integer such that $m \geq m_{\varepsilon^\prime}$. 
Proposition~\ref{expconvsepexpansion} proves there exists a separable polynomial approximation $\tilde{T}_{m}[\kappa](x,\cdot)$ such that 
\[
    \| \kappa(x,\cdot) - \tilde{T}_{m}[\kappa](x,\cdot)\|_{L^2(\mcal{Y})} 
        \leq C \frac{(\gamma \eta d)^m}{1 - \gamma \eta d}\| \kappa(x,\cdot)\|_{L^2(\widehat{\mcal{Y}})} \leq \varepsilon^\prime \| \kappa(x,\cdot)\|_{L^2(\widehat{\mcal{Y}})}
\]
for almost every $x \in \mcal{X}$. 
The space $\mcal{P}_{m}^d(\mcal{Y})$ is finite-dimensional and thus a closed linear subspace of $L^2(\mcal{Y})$. 
Clearly, the polynomial approximation $\tilde{T}_{m}[\kappa](x,\cdot) \in \mcal{P}_{m}^d(\mcal{Y})$ for almost every $x \in \mcal{X}$. 
By the orthogonal projection theorem (via Corollary 5.4 in \cite{brezis2011functional}), the orthogonal projection of $\kappa(x,\cdot)$ onto $\mcal{P}_{m}^d(\mcal{Y})$ achieves the best polynomial approximation:
\begin{align}
    \| \kappa(x,\cdot) - p^\star_m(x,\cdot)\|_{L^2(\mcal{Y})} 
    & \leq \| \kappa(x,\cdot) - \tilde{T}_{m}[\kappa](x,\cdot)\|_{L^2(\mcal{Y})} 
    \leq \varepsilon^\prime \| \kappa(x,\cdot)\|_{L^2(\widehat{\mcal{Y}})} \nonumber
\end{align}
for almost every $x \in \mcal{X}$.
\end{proof}
\begin{remark}
    Observe that $\dim \mcal{P}_{m}^d(\mcal{Y}) = \binom{m+d}{d} = \mcal{O}(m^d)$. 
    Therefore, choosing degree $m = \mcal{O}(\log{(\varepsilon^{{\prime}^{-1}})})$ yields an approximation space of dimension $\mcal{O}(\log^d{(\varepsilon^{\prime^{-1}})})$. 
\end{remark}

\subsection{Hierarchical \texorpdfstring{$\eta$}{eta}-Admissible Decomposition}
\label{hierSAdecompsection}
Due to the locality of the finite element discretization,  each entry of the matrix $\bld{A}_h$ is determined by  the kernel's restriction to the local supports of the finite element basis functions as seen here
\[
    (\bld{A}_h)_{i,j} 
        = \theta_i^\star \mcal{A} \theta_j = \int_{\mcal{X}_i} \left[ \int_{\mcal{X}_j} \kappa(x,y) \theta_j(y) dy \right] \theta_i(x) dx.
\]
In particular, each row $\bld{A}_{h}[i,:]$ then depends on the kernel restricted to the slab $\mcal{X}_i \times \Omega$. 
In this section, we show how we partition the domain $\Omega$ to retain the efficiency from approximating the operator on local admissible domains to slabs of the form $\mcal{X} \times \Omega$ for a fixed subdomain $\mcal{X} \subset \Omega$ satisfying $\diam(\mcal{X}) \leq h$. 
We use a partition that is commonly referred to as a hierarchical decomposition, which shares significant properties with the Whitney cubes arising from a Whitney decomposition (Appendix J.1 in \cite{grafakos2008classical}). 
We discuss this connection further in the discussion surrounding Figure~\ref{fig:dyadiccubes} and revisit it at the end of the section in relation to Lemmas~\ref{lemmainclusioninlargerdomain}, \ref{CZadm} and \ref{CZnon}.        

Recalling Assumption~1\ref{assumerescaling}, we  partition  the domain $[0,1]^d$ and then restrict the partition to $\Omega$ instead of partitioning $\Omega$ directly. 
We recursively construct a dyadic hierarchical partition of the cube $[0,1]^d$ as follows.
At an arbitrary dyadic level $L \geq 0$, we take a cube $I$ and check whether the domain $\mcal{X} \times I$ is admissible.
If admissible, then we stop and mark $\mcal{X} \times I$ as an admissible domain.
If the diameter of the cube $I$ is less than or equal to the mesh size, then we stop and mark $\mcal{X} \times I$ as a \textit{terminal} non-admissible domain. 
Otherwise, we split the cube into $2^d$ children, obtained by bisecting each coordinate direction, and recurse.
Initially, we start with the cube $[0,1]^d$ at level $L=0$. 
A level-$L$ dyadic cube is a product of intervals of side length $2^{-L}$, and thus the cube has diameter $\sqrt{d}2^{-L}$.
This construction ensures that larger cubes far from $\mcal{X}$ stop early as admissible cubes, while cubes near $\mcal{X}$ are refined until they either become admissible or reach the mesh scale. Similar to the partition defined in \cite{boulle2023learning}, this partition forms a dyadic tree for each slab $\mcal{X} \times \Omega$:
\begin{itemize}
    \item $I_{1,  \ldots , 1}:= I_1 \times  \cdots \times I_1 = [0,1]^d$ is the root cube for level $L = 0$. 
    \item At a given level $L \geq 0$, if $I_{j_1 , \ldots , j_d}:= I_{j_1} \times \cdots \times I_{j_d}$ is a node of the tree, then it has $2^d$ children defined as 
    \begin{equation}
    \Lambda(I_{j_1, \ldots , j_d})
        :=\{I_{2j_1 + n_j(1) , \ldots ,2 j_d + n_j(d)}: n_j \in \{0,1\}^d\}. 
    \label{eqnchildren}
    \end{equation}

    Here, if $I_j= [a,b]$ for $0 \leq a < b \leq 1$, then $I_{2j} = [a , \frac{a + b}{2}]$ and $I_{2j+1} = [\frac{a+b}{2}, b]$. 
\end{itemize}
At a level $L$, the collection of non-admissible domains can be given by the expression
\[
    P_\text{non-adm}(L) 
        = \Big\{\mcal{X} \times I_{j_1, ..., j_d}: 2^L \leq j_1, ..., j_d \leq 2^{L+1}-1, \eta \dist(\mcal{X}, I_{j_1, \ldots, j_d}) 
        < \diam(I_{j_1, \ldots, j_d}) \Big\} 
\]
where the diameter $\diam(I_{j_1, ..., j_d}) = \frac{\sqrt{d}}{2^L}$. 
The recursion never proceeds beyond the first level $n_h$ for which the dyadic cubes have diameter at most $h$. 
\begin{lemma}
\label{lemma:terminallevel}
Assume Assumption~1\ref{assumerescaling}.  
Let $h \in \mcal{H}$.  
The first  level of a dyadic partition of $[0,1]^d$ at which every dyadic cube has diameter at most $h$ is $n_h = \lceil \log_2(\sqrt{d}h^{-1}) \rceil$. 
\end{lemma}
\begin{proof}
  The first dyadic level $n_h$ at which a cube's diameter is no larger than the mesh scale $h$ satisfies the following relation
\[
    \frac{\sqrt{d}}{2^{n_h}} 
         \leq h 
        < \frac{\sqrt{d}}{2^{n_h -1}}.
\] 
Therefore, $n_h = \lceil \log_2 \left( \sqrt{d}h^{-1} \right) \rceil$.  
\end{proof}
The admissible domains are expressed by the collection of children of non-admissible domains which are admissible in the subsequent level
\[
    P_\text{adm} 
        = \bigcup_{L=1}^{n_h} \left(\Lambda (P_\text{non-adm}(L-1))\  \backslash \  P_\text{non-adm}(L) \right)
\]
where $\Lambda(P_\text{non-adm}(L-1))$ denote the children of domains in  $P_\text{non-adm}(L-1)$: 
\[ 
    \Big\{ \mcal{X} \times I_{2j_1 + n_j(1), \ldots, 2j_d + n_j(d)}: n_j \in \{0,1\}^d, \mcal{X} \times I_{j_1, \ldots, j_d} \in P_\text{non-adm}(L-1) \Big\}.
\] 
When a domain is marked admissible, no further division is needed, and this domain remains a leaf of the dyadic tree. This recursive nature of the partition of $[0,1]^d$ with respect to $\mcal{X}$ ensures that the cubes of admissible domains, which we refer to as admissible cubes, possess a few key properties that closely resemble the structure of dyadic cubes stemming from a Whitney decomposition.
Namely, the parent of an admissible cube is not admissible, guaranteeing that an admissible cube is the unique maximal dyadic cube to satisfy the admissibility condition and the diameter of an admissible cube is comparable (up to constants) to its distance from the fixed subdomain. 
These properties are depicted in Figure \ref{fig:dyadiccubes}. A binary hierarchical decomposition of a 1-dimensional example is visualized in Figure \ref{fig:tree}. 
We refer to the remaining non-admissible domains at the final level in $P_\mathrm{non-adm}(n_h)$ as the \textit{terminal} non-admissible domains. At level $L = n_h$, the mesh-scale condition forces the recursion to end, and the terminal non-admissible domains also remain as leaves of the dyadic tree.

\begin{figure}[htbp]
    \centering
    \begin{subfigure}{0.48\textwidth}
        \centering
        \includegraphics[height=0.45\linewidth]{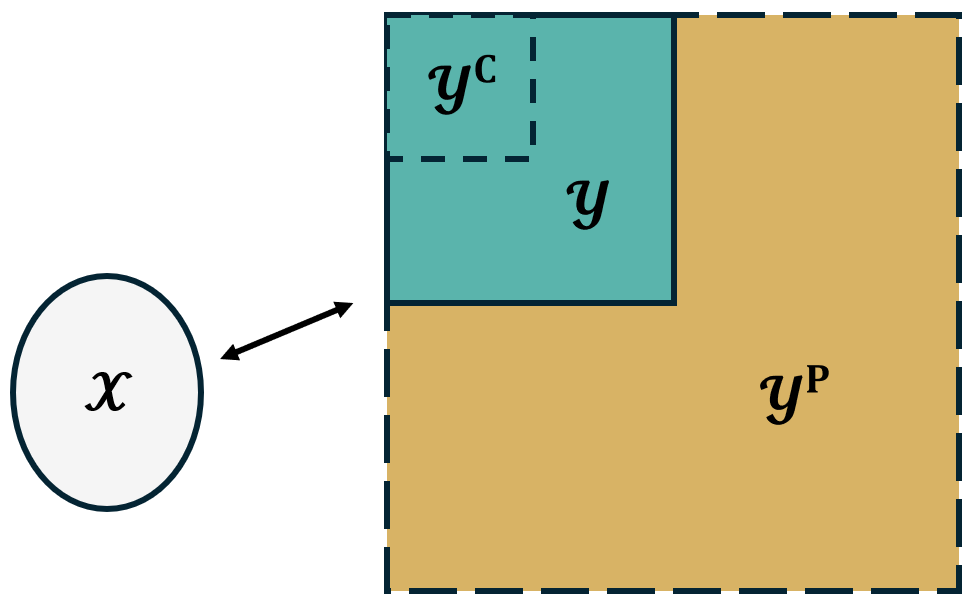}
    \caption{An admissible cube $\mcal{Y}$ from a hierarchical partition w.r.t $\mcal{X} \subset \Omega$ is the maximal dyadic cube of that dyadic decomposition to satisfy $\eta \dist(\mcal{X},\mcal{Y}) \geq \diam(\mcal{Y})$. Its children, e.g. $\mcal{Y}^C$, are also $\eta$-admissible; however, its parent $\mcal{Y}^P$ is not.}
     \label{fig:dyadiccubes}
    \end{subfigure}
    \hfill
    \begin{subfigure}{0.48\textwidth}
        \centering
\includegraphics[height=0.5\linewidth]{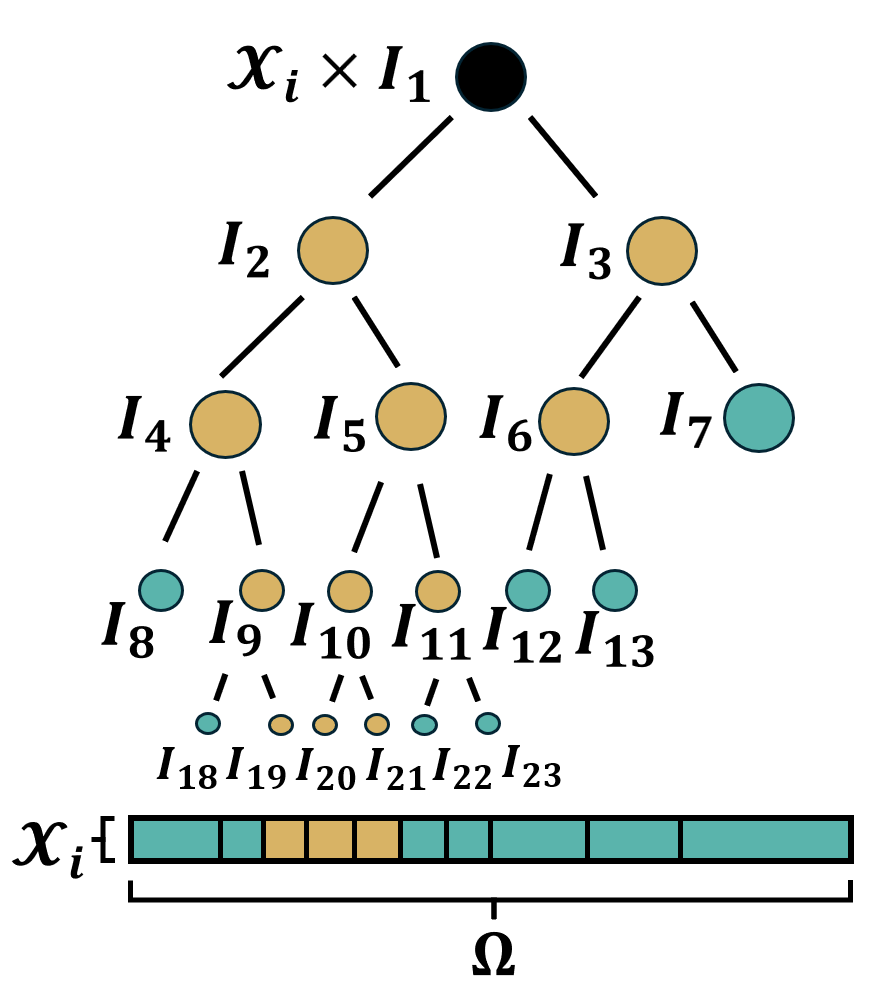}
    \caption{A binary hierarchical partition of $\mcal{X}_i \times [0,1]$ is depicted. 
    The collection of admissible domains is  $P_\mathrm{adm} = \{I_7, I_8, I_{12}, I_{13}, I_{18}, I_{22}, I_{23}\}$ and the collection of non-admissible domains at the final level $L = 4$ shown is $P_\mathrm{non-adm}( 4) = \{I_{19}, I_{20}, I_{21}\}$.}
    \label{fig:tree}
    \end{subfigure}
    \caption{}
\end{figure}

Algorithm~\ref{alghierarchicaldecomp} formalizes this recursive construction.
Given a fixed subdomain $\mcal{X}$, a dyadic cube $I_{j_1, \dots, j_d}$, an admissibility constant $\eta$ and a mesh size $h$, it
partitions the slab $\mcal{X} \times I_{j_1,\dots,j_d}$ into admissible
and terminal non-admissible domains down to the level of the finite element discretization. 
It is important to note here that the algorithm does not take $n_h$ as an input; rather, Lemma~\ref{lemma:terminallevel} shows that the mesh-scale condition forces the recursion to terminate by dyadic level $n_h$.
\begin{algorithm}

\caption{Hierarchical Decomposition of $d$-dimensional Cube $I_{j_1  , \ldots, j_d}$ w.r.t $\mcal{X}$} \label{alghierarchicaldecomp}

\begin{algorithmic}[1]
\Procedure{Hierarchical Partition}{ $\mcal{X}$, $I_{j_1 , \ldots, j_d}$, $\eta$, $h$}

\If {$\mcal{X} \times I_{j_1, \ldots, j_d}$ is $\eta$-admissible} 

\Return $P_\text{adm} = \{\mcal{X} \times I_{j_1, \ldots, j_d}\} \ \& \ P_\text{non} = \emptyset$

\ElsIf{$\diam(I_{j_1  \times \cdots \times j_d}) \leq h$}

\Return $P_\text{adm} = \emptyset \ \& \ P_\text{non} = \{\mcal{X} \times I_{j_1 , \ldots, j_d}\}$

\Else 
\State Initialize sets $P_\text{adm} = \emptyset$ \ \& \ $P_\text{non} = \emptyset$

\State Form $2^d$ children $\Lambda(I_{j_1  \times \cdots \times j_d})$
\color{gray} \algorithmiccomment{$\Lambda$ defined in Eqn. \ref{eqnchildren}}
\color{black}

\For {child $I_{j_1^\prime, \ldots,  j_d^\prime} \in \Lambda(I_{j_1, \ldots, j_d})$} \do \\

\State $(P_\text{adm}^\prime,P_\text{non}^\prime) =$ \textsc{Hierarchical Partition}($\mcal{X},I_{j_1^\prime , \ldots,  j_d^\prime}, \eta, h$)

\State Set $P_\text{adm} = P_\text{adm} \cup P_\text{adm}^\prime $ \ \& \ $P_\text{non} = P_\text{non} \cup P_\text{non}^\prime$
\EndFor

\Return $(P_\text{adm}, P_\text{non})$

\EndIf
\EndProcedure
\end{algorithmic}
\end{algorithm}
When initialized with inputs $(\mcal{X}, [0,1]^d, \eta, h)$, 
the algorithm returns the collections $P_\text{adm}$ and $P_\text{non}$ of admissible leaves and terminal non-admissible leaves, respectively. 
Writing these leaves as $\mcal{X} \times \mcal{Y}_j$ indexed by $j \in \{1: |P_\text{adm}|\}$ and
$\mcal{X} \times \mcal{N}_p$ indexed by $p \in \{1:|P_\text{non}|\}$, these two sets restricted to $\mcal{X} \times \Omega$ yield the following covering, disjoint up to boundaries: 
\begin{equation} 
\label{hierdecompeqn}
    \mcal{X} \times \Omega 
        = \bigcup_{j=1}^{|P_\text{adm}|} [\mcal{X} \times (\mcal{Y}_{j} \cap \Omega)] \cup \bigcup_{p=1}^{|P_\text{non}|} [\mcal{X} \times (\mcal{N}_{p} \cap \Omega)].  
\end{equation}
In the subsequent lemmas, we show that the resulting decomposition covers most of the slab with a number of admissible domains that grows only logarithmically in the inverse mesh size, while covering the region near $\mcal{X}$ with a uniformly bounded number of  terminal non-admissible blocks at the mesh scale.

\begin{lemma} [Number of Non-admissible domains]
\label{lemmaNAdomains}
Assume Assumptions 1\ref{assumefemmesh} and 1\ref{assumerescaling}. 
Let $h \in \mcal{H}$ and let $\mcal{X} \subset \Omega$ such that $\diam(\mcal{X}) \leq h$. 
Let $\eta >0 $.
Define $n_h = \lceil \log_2(\sqrt{d}h^{-1})\rceil$.
The number of non-admissible domains at any hierarchical level $L=1, \ldots, n_h$ is bounded uniformly by
\begin{equation}
    |P_\text{non-adm}(L)| 
        \leq \omega_d \sqrt{d}^d \left( 3 + \frac{1}{\eta}\right)^d =: K_{d,\eta}
\label{numberofnonadmissible}
\end{equation} 
with $\omega_d = \text{Vol}(B(0,1))$ being the volume of the unit ball in $d$ dimensions.
In particular, let $(P_\text{adm}, P_\text{non})$ be the output of Algorithm~\ref{alghierarchicaldecomp} applied to $(\mcal{X}, [0,1]^d, \eta, h)$. Then the number of terminal non-admissible domains satisfies
\[
    |P_\text{non}| \leq  K_{d,\eta}.
\] 
\end{lemma}

\begin{proof}
Let $1 \leq L \leq n_h$ be an arbitrary level. 
Let $\mcal{X} \times \mcal{N}$ be an arbitrary non-admissible domain at level $L$ which thus satisfies: 
\[
    \eta \dist(\mcal{X}, \mcal{N}) < \diam(\mcal{N}) = \frac{\sqrt{d}}{2^L}.
\]
Pad the domain $\mcal{X}$ with a Euclidean buffer $\delta = (1 + \frac{1}{\eta}) \diam(\mcal{N})$, and denote the padded domain $\mcal{X}_\delta = \{y \in [0,1]^d: \dist(y,\mcal{X}) \leq \delta\}$. 
For any $y \in \mcal{N}$, 
\[
    \dist(y,\mcal{X}) 
        \leq \dist(\mcal{N}, \mcal{X}) + \diam(\mcal{N}) 
        < \left( 1 + \frac{1}{\eta} \right) \diam(\mcal{N}) 
        = \delta
\]
and hence $\mcal{N} \subset \mcal{X}_\delta$. Since $\mcal{X}$ can be contained within a ball of radius $h$, the padded domain $\mcal{X}_\delta$ can be contained in a ball of radius $(1 + \frac{1}{\eta}) \diam(\mcal{N}) + h$. 
Then, the number of non-admissible domains at level $L$ is bounded by
\begin{align}
    |P_\text{non-adm}(L)| 
        & \leq \frac{\mathrm{Vol}(\mcal{X}_\delta)}{\mathrm{Vol}(\mcal{N})} 
            \leq \frac{\mathrm{Vol}(B_{r= (1 + \frac{1}{\eta}) \diam(\mcal{N}) + h}) }{\left(\frac{\diam(\mcal{N})}{\sqrt{d}}\right)^d}
            \nonumber \\
        & \leq \omega_d \frac{\left( (1 + \frac{1}{\eta}) \diam(\mcal{N}) + h\right)^d}{\left(\frac{\diam(\mcal{N})}{\sqrt{d}}\right)^d}
            = \omega_d \sqrt{d}^d \left( (1 + \frac{1}{\eta}) + \frac{h}{\diam(\mcal{N})} \right)^d. 
            \nonumber
\end{align}
Using $\frac{h}{\diam(\mcal{N})} < 2$, which holds for all $1 \leq L \leq n_h$, we obtain
\[
        |P_\text{non-adm}(L)|  
        \leq \omega_d \sqrt{d}^d \left( 3 + \frac{1}{\eta} \right)^d 
        = K_{d,\eta}.
\]
Every terminal non-admissible block returned by the algorithm is a level-$n_h$ non-admissible block. Therefore, $P_\text{non} \subset P_\text{non-adm}(n_h)$, and the desired bound follows from the level-wise estimate with $L = n_h$. 
\end{proof}

We can leverage Lemma~\ref{lemmaNAdomains} to bound the number of admissible domains by a logarithmic factor with respect to the inverse mesh size.

\begin{lemma}[Number of Admissible Domains]
\label{lemmaAdomains}
Assume Assumptions 1\ref{assumefemmesh} and 1\ref{assumerescaling}. 
Let $h \in \mcal{H}$ and let $\mcal{X} \subset \Omega$ such that $\diam(\mcal{X}) \leq h$. 
Let $\eta > 0$. 
Let $(P_\text{adm}, P_\text{non})$ be the output of Algorithm~\ref{alghierarchicaldecomp} applied to $(\mcal{X}, [0,1]^d, \eta, h)$. 
Define $n_h = \lceil \log_2(\sqrt{d} h^{-1})\rceil$.
The number of admissible domains is bounded by 
\begin{equation} 
    |P_\text{adm}| 
        \leq 2^d K_{d,\eta}n_h.
\label{numberofadmissible}
\end{equation}

\end{lemma}

\begin{proof}
Fix a subdomain $\mcal{X}$. For any level $L = 1, \ldots, n_h$,  Lemma~\ref{lemmaNAdomains} uniformly bounds the number of non-admissible domains by 
$|P_\text{non-adm}(L)| \leq K_{d,\eta}$. 
Therefore, the number of children of non-admissible blocks from level $L-1$ that are considered admissible at level $L$ can be bounded by
\[
    |\Lambda(P_\text{non-adm}(L-1))  \ \backslash \ P_\text{non-adm}(L)| 
        \leq |\Lambda(P_\text{non-adm}(L-1))| 
        \leq 2^d K_{d,\eta}.
\] 
Thus, the total number of admissible domains is bounded by
\[
    |P_\text{adm}| 
        \leq \sum_{L=1}^{n_h} 2^d K_{d,\eta} 
        = 2^d K_{d,\eta} n_h.
\]
\end{proof}
The preceding lemmas control the number of admissible and terminal non-admissible domains in the hierarchical decomposition of $[0,1]^d$ with respect to $\mcal{X}$. 
In Section~\ref{rowspaceconstructionsection}, we will construct low-dimensional subspaces that approximate the row space for each row $(\bld{A}_h)_{i,:}$ where $i = 1, \ldots, |\mcal{T}_h|$. 
For these row-wise estimates, we need bounds on two more quantities. 
The local approximation estimate from Corollary~\ref{sepexpansionwithL2orthogonal} is applied separately on each admissible cube after restricting it to the convex components of $\Omega$. 
Thus, for an admissible cube $\mcal{Y}_{i,j}$ and a convex component $\mcal{C}_n$, the local error on $\mcal{Y}_{i,j} \cap \mcal{C}_n$ is controlled by the kernel norm on an enlarged set $\widehat{\mcal{Y}_{i,j} \cap \mcal{C}_n}$. 
When these local estimates are summed over all admissible cubes $j = 1, \ldots, N_{\text{adm},i}$ and components $n = 1, \ldots, N_{cc}$, the same admissible cube of the hierarchical decomposition can be counted multiple times. 
More precisely, in the next section, we will need to bound the two multiplicities
\[
\mathrm{max}_{l \in \{1, ..., N_\text{adm}(i)\}} \bigg|\{j : \mcal{Y}_{i,l} \cap \widehat{\mcal{Y}_{i,j} \cap \mcal{C}_n} 
            \neq \emptyset\}\bigg|
\]
and
\[
\mathrm{max}_{p \in \{1, ..., N_\text{non}(i)\}} \bigg|\{j : \mcal{N}_{i,p} \cap \widehat{\mcal{Y}_{i,j} \cap \mcal{C}_n} 
            \neq \emptyset\}\bigg|.
\]
Directly counting the component-wise enlargements is inconvenient. 
We therefore introduce, for each admissible cube $\mcal{Y}_{i,j}$, a larger cube-based enlargement $\tilde{\mcal{Y}}_{i,j}$. 
Lemma~\ref{lemmainclusioninlargerdomain} below shows that every component-wise enlargement $\widehat{\mcal{Y}_{i,j} \cap \mcal{C}_n}$ is contained in $\tilde{\mcal{Y}}_{i,j}$. 

\begin{lemma}
\label{lemmainclusioninlargerdomain}
Let $\mcal{X} \subset \Omega$. Let $\mcal{Y} \subset [0,1]^d$ such that $\mcal{X} \times \mcal{Y}$ is $\eta$-admissible. 
Define $\tilde{\mcal{Y}}:= \{z^\prime \in [0,1]^d: \dist(z^\prime, \mcal{Y}) \leq \frac{1}{2}(1 + \eta) \dist(\mcal{X}, \mcal{Y})\}$. 
Let $\mcal{Z} \subset \mcal{Y}$. 
Then $\widehat{\mcal{Z}} \subset \tilde{\mcal{Y}} \cap \Omega$.
\end{lemma}

\begin{proof}
Let 
$z \in \widehat{\mcal{Z}}  
    := \{z^\prime \in \Omega: \dist(z^\prime, \mcal{Z}) 
    \leq \frac{1}{2} \dist(\mcal{X}, \mcal{Z})\}$.
Since $\mcal{Z} \subset \mcal{Y}$, 
\begin{align} 
    \dist(z, \mcal{Y}) 
        & \leq \dist(z, \mcal{Z}) \leq \frac{1}{2} \dist(\mcal{X}, \mcal{Z}) \nonumber \\
        & \leq \frac{1}{2} \left( \dist(\mcal{X}, \mcal{Y}) + \diam(\mcal{Y})\right) \nonumber \\
        & \leq \frac{1}{2} \left( \dist(\mcal{X}, \mcal{Y}) + \eta \dist(\mcal{X}, \mcal{Y})\right). \nonumber
\end{align} 
Therefore, $z \in \tilde{\mcal{Y}}$. Since $z \in \widehat{\mcal{Z}} \subset \Omega$, it follows that $z \in \tilde{\mcal{Y}} \cap \Omega. $
\end{proof}
Consequently, it is enough to bound the following two quantities
\begin{equation*}
    \mathrm{max}_{l \in \{1, ..., N_\text{adm}(i)\}} \bigg|\{j : \mcal{Y}_{i,l} \cap \tilde{\mcal{Y}}_{i,j} 
        \neq \emptyset\}\bigg|
\end{equation*} 
and 
\begin{equation*}
    \mathrm{max}_{p \in \{1, ..., N_\text{non}(i)\}} \bigg|\{j : \mcal{N}_{i,p} \cap \tilde{\mcal{Y}}_{i,j} 
        \neq \emptyset\}\bigg|.
\end{equation*}
Figure~\ref{fig:countingadmissibleandnon} visualizes the overlaps of an arbitrary admissible or terminal non-admissible cube with these larger cube-based enlargements. 
It is well-known there is a uniform bound for the number of neighboring cubes of any given dyadic cube from a Whitney decomposition (see J.1 in \cite{grafakos2008classical}). 
The same geometric principle applies to the hierarchical decomposition. 
Lemmas \ref{CZadm} and \ref{CZnon} leverage technical packing arguments of a similar nature with $\eta < 1$ to uniformly bound these two quantities by constants $K_{W,\text{adm}}$ and $K_{W,\text{non}}$ respectively. 
Both constants depend only on $d$ and $\eta$.

\begin{figure}[htbp]
    \centering
    \begin{subfigure}[t]{0.45\textwidth}
        \centering
        \includegraphics[width=0.68\linewidth]{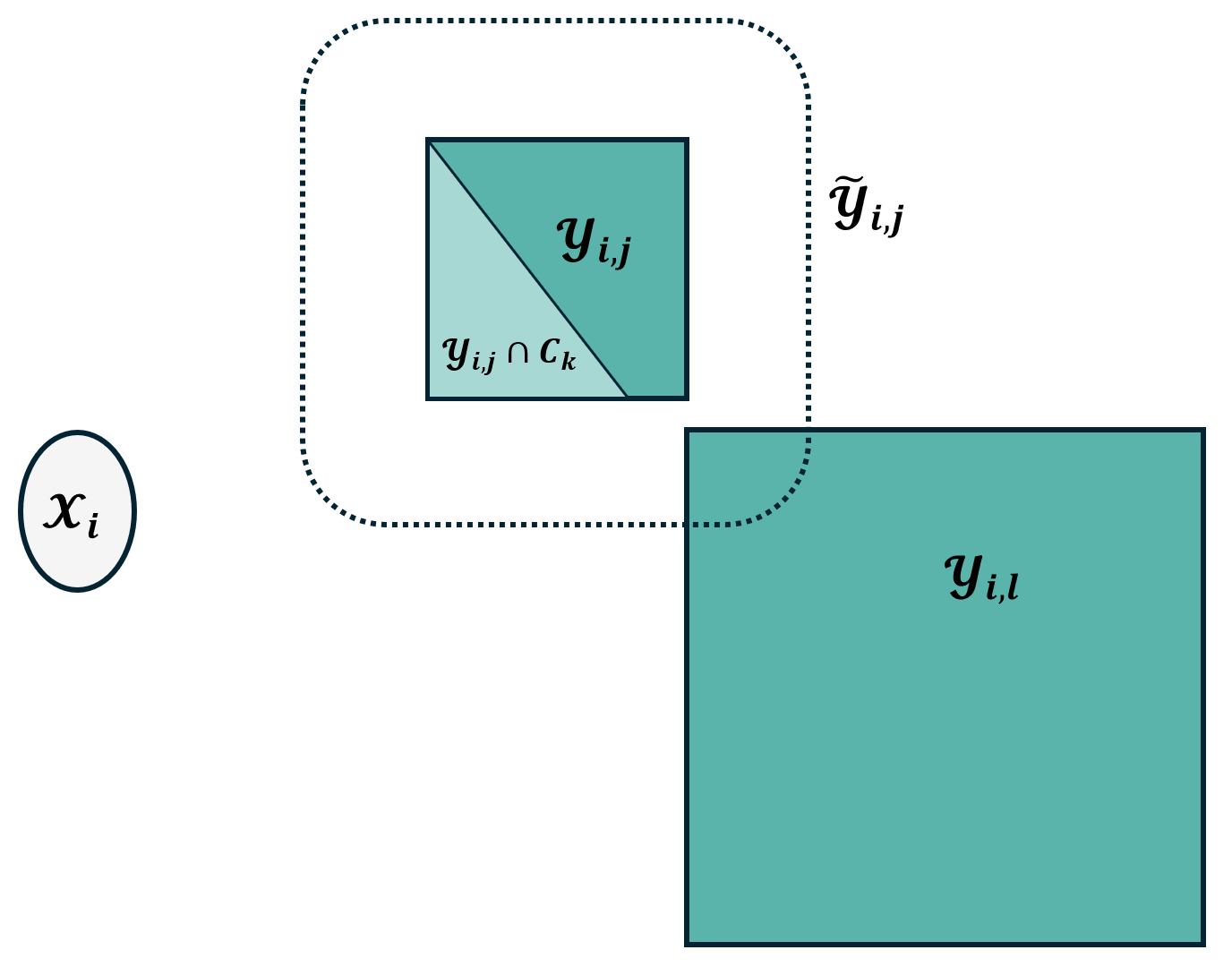}
        \caption{An admissible cube $\mcal{Y}_{i,j}$ whose enlarged
        domain intersects with a given admissible cube $\mcal{Y}_{i,l}$.}
    \end{subfigure}
    \hfill
    \begin{subfigure}[t]{0.45\textwidth}
        \centering
        \includegraphics[width=0.8\linewidth]{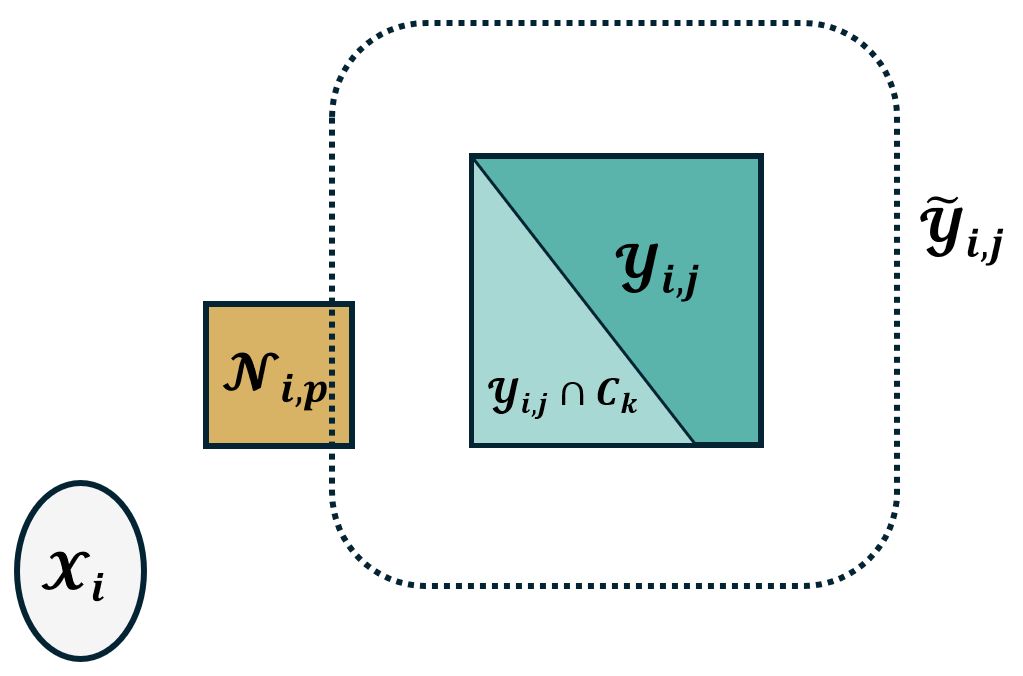}
        \caption{An admissible cube $\mcal{Y}_{i,j}$ whose enlarged
        domain intersects with a given non-admissible cube $\mcal{N}_{i,p}$.}
    \end{subfigure}
    \caption{}
    \label{fig:countingadmissibleandnon}
\end{figure}

\subsubsection{Overlap Bound for Admissible Cubes}
We first bound the number of enlarged admissible cubes that can intersect a fixed admissible leaf. 

\begin{lemma}
\label{CZadm} 
Let $0 < \eta < 1$, $h \in \mcal{H}$, let $\mcal{X} \subset \Omega$, and let $(P_\text{adm}, P_\text{non})$ be the output of Algorithm~\ref{alghierarchicaldecomp} applied to $(\mcal{X}, [0,1]^d, \eta, h)$.
Then there exists a constant $K_{W,\text{adm}}$ depending only on $d$ and $\eta$ such that for every admissible leaf $\mcal{X} \times \mcal{Y} \in P_\text{adm}$,
\[
    |\{ \mcal{X} \times \mcal{Q} \in P_\text{adm}: \mcal{Y} \cap \tilde{\mcal{Q}} \neq \emptyset \}| \leq K_{W,\text{adm}}.
\]

\end{lemma}
\begin{proof}
Fix an admissible leaf $\mcal{X} \times \mcal{Y} \in P_\text{adm}$, and consider any admissible leaf $ \mcal{X} \times \mcal{Q} \in P_\text{adm}$ satisfying $\mcal{Y} \cap \tilde{\mcal{Q}} \neq \emptyset$. 
We first derive the overlap bound assuming that 
\begin{equation} 
\label{diamQreltodiamYmod}
    K_\text{lower} \diam(\mcal{Y}) 
        \leq \diam(\mcal{Q}) 
        \leq K_\text{upper} \diam(\mcal{Y})
\end{equation}
where $K_\text{lower}$ and $K_\text{upper}$ are positive constants depending only on $\eta$. 
We establish these bounds at the end of the proof. 
Set a padding factor
\[
\delta:= \left( \left( \frac{1+\eta}{1- \eta}(\frac{2}{\eta}+2)\right) + K_\text{upper} + 1\right)\diam(\mcal{Y}),
\]
and define the padded domain
\[
    \mcal{Y}_\delta := \{z \in \R^d: \dist(z,\mcal{Y}) \leq \delta\}.
\]
We claim that $\mcal{Q} \subset \mcal{Y}_\delta$, as visualized in Figure~\ref{fig:packing}.
\begin{figure}[htbp]
    \centering
    \includegraphics[width=0.3\linewidth]{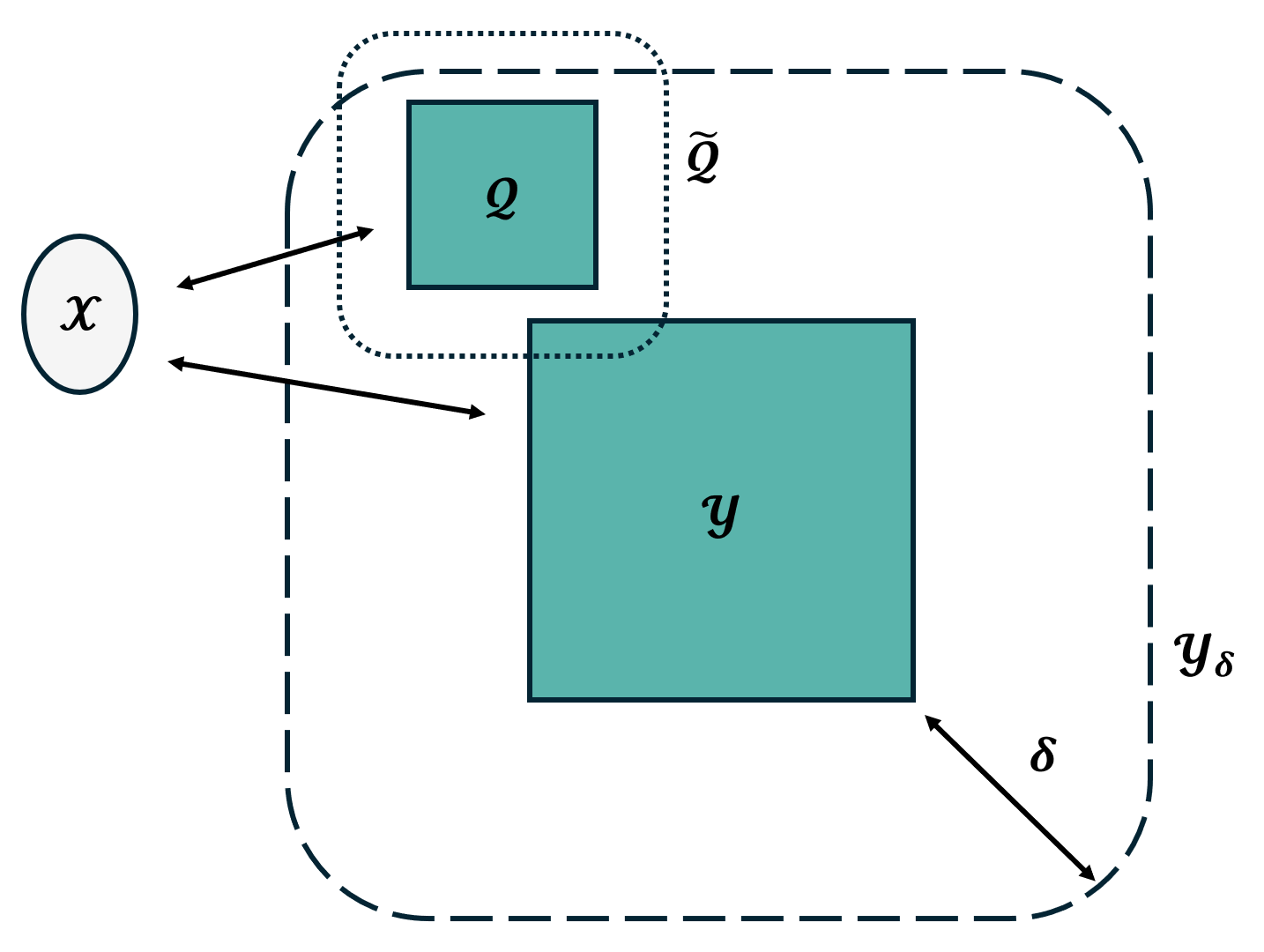}
    \caption{Both $\mcal{Q}$ and $\mcal{Y}$ are admissible cubes with respect to a fixed subdomain $\mcal{X}$ such that $\tilde{\mcal{Q}} \cap \mcal{Y} \neq \emptyset$. 
    A padding factor $\delta$ is found such that $\mcal{Q} \subseteq \mcal{Y}_\delta$ is guaranteed.}
\label{fig:packing}
\end{figure}
Suppose, for contradiction, that $\mcal{Q} \not\subset \mcal{Y}_\delta$. 
The intersection condition $\tilde{\mcal{Q}} \cap \mcal{Y} \neq \emptyset$ and the definition of $\tilde{\mcal{Q}}$ imply that 
$\dist(\mcal{Q}, \mcal{Y}) \leq \frac{1}{2}(1 + \eta) \dist(\mcal{X},\mcal{Q}).$ 
Combining this with $\dist(\mcal{X}, \mcal{Q}) \leq \dist(\mcal{X},\mcal{Y}) + \diam(\mcal{Y}) +  \dist(\mcal{Y},\mcal{Q})$, we obtain 
\[
    \left( 1 - \frac{1}{2}(1 + \eta) \right) \dist(\mcal{X}, \mcal{Q}) 
        \leq \dist(\mcal{X}, \mcal{Y}) + \diam(\mcal{Y}).
\]   
Since $\eta < 1$, the term $1 - \frac{1}{2}(1 + \eta)$ is positive. 
Moreover, because $\mcal{Y}$ is a dyadic child of its non-admissible parent $\mcal{Y}^P$, every point of $\mcal{Y}^P$ lies within $\text{diam}(\mcal{Y})$ of $\mcal{Y}$.
Hence,
$\dist(\mcal{X}, \mcal{Y}) \leq \dist(\mcal{X}, \mcal{Y}^P) + \text{diam}(\mcal{Y})$. 
Together, these estimates yield
\begin{align}
    \dist(\mcal{Y},\mcal{Q}) 
        & \leq \frac{1}{2}(1+\eta)\dist(\mcal{X}, \mcal{Q}) \nonumber \\
        & \leq \frac{\frac{1}{2}(1+\eta) }{1 - \frac{1}{2}(1 + \eta)} \left( \dist(\mcal{X}, \mcal{Y}) + \diam(\mcal{Y}) \right) 
        \label{eq:helpfulbddistYQ} \\
        & \leq  \frac{1+\eta}{1 - \eta} \left(\dist(\mcal{X},\mcal{Y}^P) + 2\diam(\mcal{Y})\right), \nonumber 
\intertext{
and by non-admissibility of the parent $\mcal{Y}^P$
}
    \dist(\mcal{Y},\mcal{Q}) 
        & <  \frac{1 + \eta}{1 - \eta}\left(\frac{1}{\eta} \diam(\mcal{Y}^P) + 2\diam(\mcal{Y})\right) \nonumber \\
        & = \frac{(1+\eta) (\frac{2}{\eta} + 2)}{1 - \eta}  \diam(\mcal{Y}). \nonumber
\end{align}
On the other hand, $\mcal{Q} \not \subset \mcal{Y}_\delta$ implies that some point of $\mcal{Q}$ has distance greater than $\delta$ from $\mcal{Y}$, and therefore $\delta \leq \dist(\mcal{Q}, \mcal{Y}) + \diam(\mcal{Q})$. 
However, this conclusion would imply by the definition of $\delta$ that
\begin{align} 
    \frac{(1+\eta)(\frac{2}{\eta} + 2)}{1 -\eta} \diam(\mcal{Y}) 
        & \geq \dist(\mcal{Q}, \mcal{Y}) \nonumber \\
        & \geq \delta - K_\text{upper} \diam(\mcal{Y}) \nonumber \\
        & = \left( \frac{(1+\eta)(\frac{2}{\eta} + 2)}{1 - \eta} + 1\right)\diam(\mcal{Y}) \nonumber 
\end{align} 
which is a contradiction. 
Therefore, any admissible cube $\mcal{Q}$ such that $\mcal{X} \times \tilde{\mcal{Q}} \cap \mcal{X} \times \mcal{Y} \neq \emptyset$ lies in  $\mcal{Y}_\delta$. 

The padded domain $\mcal{Y}_\delta$ is a rounded cube contained in the cube concentric with $\mcal{Y}$ having sidelength $\mathrm{len}(\mcal{Y}) + 2 \delta$.  
Denote this cube $\mcal{Y}_{\delta,\text{cube}}$, so $\mcal{Y}_\delta \subset \mcal{Y}_{\delta,\text{cube}}$. 
Moreover, the lower bound in Equation~\ref{diamQreltodiamYmod} yields
\[
\Vol{(\mcal{Q})} = \left(\frac{\diam(\mcal{Q})}{\sqrt{d}}\right)^d \geq \left(\frac{K_\text{lower}\diam(\mcal{Y})}{\sqrt{d}}\right)^d.
\]
Since distinct leaf cubes have disjoint interiors, a volume comparison gives
\begin{align} 
    |\{ \mcal{X} \times \mcal{Q} \in P_\text{adm}: \mcal{Y} \cap \tilde{\mcal{Q}} \neq \emptyset \}| 
        &  \leq \frac{\mathrm{Vol}(\mcal{Y}_{\delta,\text{cube}})}{\left( \frac{K_\text{lower}\diam(\mcal{Y})}{\sqrt{d}} \right)^d} 
        \leq \frac{\left(\frac{\diam(\mcal{Y})}{\sqrt{d}} + 2 \delta \right)^d}{\left( \frac{K_\text{lower}\diam(\mcal{Y})}{\sqrt{d}} \right)^d} \nonumber \\
        & \leq \frac{\left(\frac{\left[1 + 2 \sqrt{d}\left( \frac{(1+\eta)(\frac{2}{\eta} + 2)}{1 - \eta} + 1 + K_\text{upper} \right)\right]\diam(\mcal{Y})}{\sqrt{d}}\right)^d}{\left(\frac{K_\text{lower}\diam(\mcal{Y})}{\sqrt{d}}  \right)^d} \nonumber \\
        & \leq \left(\frac{\left[( 1 + 2 \sqrt{d}\left( \frac{(1+\eta)(\frac{2}{\eta} + 2)}{1 - \eta} + 1 + K_\text{upper} \right)\right]}{K_\text{lower}}\right)^d \nonumber \\
        & =: K_{W,\text{adm}}. \nonumber 
\end{align} 

It remains to establish the lower and upper bounds in Equation~\ref{diamQreltodiamYmod}. 
For the lower bound, every point of the non-admissible parent $\mcal{Q}^P$ lies within $\text{diam}(\mcal{Q})$ of $\mcal{Q}$, and therefore  
\begin{align}
     \diam(\mcal{Q}) 
        & \geq \dist(\mcal{X}, \mcal{Q}) - \dist(\mcal{X}, \mcal{Q}^P) \nonumber \\ 
        & \geq \dist(\mcal{X}, \mcal{Q}) - \frac{1}{\eta} \diam(\mcal{Q}^P) 
            = \dist(\mcal{X}, \mcal{Q}) - \frac{2}{\eta} \diam(\mcal{Q}). \nonumber
\end{align}
Together with the intersection condition $\tilde{\mcal{Q}} \cap \mcal{Y} \neq \emptyset$, this implies
\begin{align}
    \dist(\mcal{X}, \mcal{Y}) 
        &\leq \dist(\mcal{X}, \mcal{Q}) + \diam(\mcal{Q}) + \dist(\mcal{Q}, \mcal{Y}) \nonumber \\
        & \leq \left( 1 + \frac{1}{2}(1+\eta) \right) \dist(\mcal{X}, \mcal{Q}) + \diam(\mcal{Q}) \nonumber \\
        & \leq \left[ \left( 1 + \frac{1+\eta}{2}\right) \left( \frac{2}{\eta} + 1 \right) + 1\right]\diam(\mcal{Q}) \nonumber
\end{align}
By admissibility of $\mcal{Y}$, 
\[ \diam(\mcal{Y}) \leq \eta \dist(\mcal{X}, \mcal{Y}), \]
and so
\[
\diam(\mcal{Q}) 
    \geq \frac{2}{(\eta + 6)(\eta+1)}\diam(\mcal{Y}).
\] 
Thus, $K_\text{lower} = \frac{2}{(\eta + 6)(\eta+1)}$.

For the upper bound, admissibility of $\mcal{Q}$, the earlier estimate from Equation \ref{eq:helpfulbddistYQ} and the non-admissibility of the parent $\mcal{Y}^P$ give 
\begin{align} 
    \diam(\mcal{Q}) 
        & \leq \eta \dist(\mcal{X}, \mcal{Q}) \nonumber \\
        & \leq \eta \left[ \dist(\mcal{X}, \mcal{Y}) + \diam(\mcal{Y}) + \dist(\mcal{Y}, \mcal{Q})\right]. 
        \nonumber  \\
        & \leq  \eta \left( 1 + \frac{\frac{1}{2}(1+\eta)}{1 - \frac{1}{2}(1+\eta)} \right) \left[ \dist(\mcal{X}, \mcal{Y}) + \diam(\mcal{Y})\right] 
        \nonumber \\
    & \leq \eta \left( 1 + \frac{\frac{1}{2}(1+\eta)}{1 - \frac{1}{2}(1+\eta)} \right) \left[ \dist(\mcal{X}, \mcal{Y}^P) + \diam(\mcal{Y}) + \diam(\mcal{Y}) \right] \nonumber \\
     & < \eta \left(\frac{\frac{2}{\eta} + 2}{1 - \frac{1}{2}(1+\eta)} \right)  \diam(\mcal{Y}). \nonumber
\end{align}
Thus, $K_\text{upper} = \frac{4 + 4 \eta}{1-\eta}$. 
Both bounds therefore hold with constants depending only on $\eta$, completing the proof. 
\end{proof}

\subsubsection{Overlap Bound for Terminal Non-admissible Cubes}

We next establish the analogous bound when the fixed cube is terminal non-admissible. 
The geometry near the diagonal is different, so we treat this case separately.

\begin{lemma}
\label{CZnon} 
Let $0 < \eta < 1$, $h \in \mcal{H}$, let $\mcal{X} \subset \Omega$, and let $(P_\text{adm}, P_\text{non})$ be the output of Algorithm~\ref{alghierarchicaldecomp} applied to $(\mcal{X}, [0,1]^d, \eta, h)$.
Then there exists a constant $K_{W,\text{non}}$ depending only on $d$ and $\eta$ such that for every terminal non-admissible leaf $\mcal{X} \times \mcal{N} \in P_\text{non}$,
\[
    |\{ \mcal{X} \times \mcal{Q} \in P_\text{adm}: \mcal{N} \cap \tilde{\mcal{Q}} \neq \emptyset \}| \leq K_{W,\text{non}}.
\] 
\end{lemma}

\begin{proof}
Fix $\mcal{X} \times \mcal{N} \in P_\text{non}$, and consider any admissible leaf $\mcal{X} \times \mcal{Q} \in P_\text{adm}$ satisfying $\mcal{N} \cap \tilde{\mcal{Q}} \neq \emptyset$. 
We first derive the overlap bound assuming that 
\begin{equation} 
\label{diamQreltodiamNmod}
    K_\text{lower} \diam(\mcal{N}) 
        \leq \diam(\mcal{Q}) 
        \leq K_\text{upper} \diam(\mcal{N}),
\end{equation}
where $K_\text{lower}$ and $K_\text{upper}$ are positive constants depending only on $\eta$. 
We establish these bounds at the end of the proof. 
Set
\[
    \delta :=  \frac{(1+\eta)(\frac{1}{\eta} + 1)}{1-\eta} \diam(\mcal{N}) + K_\text{upper} \diam(\mcal{N}) +  \diam(\mcal{N}),
\] and define the padded domain 
\[
    \mcal{N}_\delta := \{z \in \R^d: \dist{(z,\mcal{N})} \leq \delta\}.
\]
We claim that $\mcal{Q} \subset \mcal{N}_\delta$, as visualized in Figure~\ref{fig:packingnon}.
\begin{figure}[htbp]
    \centering
    \includegraphics[width=0.3\linewidth]{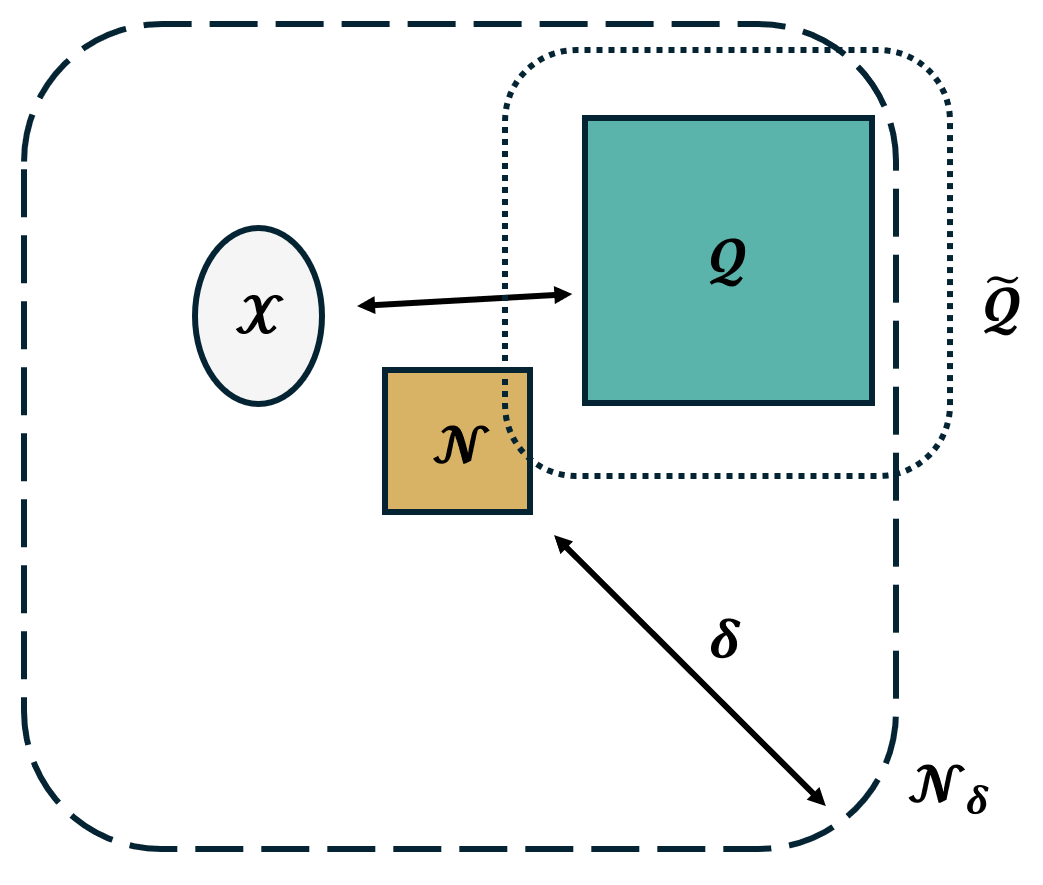}
    \caption{The cube $\mcal{Q}$ is admissible and the cube $\mcal{N}$ is not $\eta$-admissible with respect to a fixed subdomain $\mcal{X}$ such that $\tilde{\mcal{Q}} \cap \mcal{N} \neq \emptyset$.
    A padding factor $\delta$ is found such that $\mcal{Q} \subseteq \mcal{N}_\delta$ is guaranteed.}
    \label{fig:packingnon}
\end{figure}
Suppose, for contradiction, that $\mcal{Q} \not\subset \mcal{N}_\delta$. 
The intersection condition $\tilde{\mcal{Q}} \cap \mcal{N} \neq \emptyset$ and the definition of $\tilde{\mcal{Q}}$ imply that
\begin{align} 
    \dist(\mcal{Q}, \mcal{N}) 
        & \leq \frac{1}{2}(1 + \eta) \dist(\mcal{X},\mcal{Q}) 
        \nonumber \\
        & \leq \frac{1}{2}(1 + \eta) \left[ \dist(\mcal{X}, \mcal{N}) + \diam(\mcal{N}) + \dist(\mcal{N}, \mcal{Q}) \right] \nonumber \\
        & \leq \frac{1}{2}(1 + \eta) \left[ \left(\frac{1}{\eta}+1\right) \diam(\mcal{N}) + \dist(\mcal{N}, \mcal{Q}) \right], 
        \nonumber
\end{align}
where the last inequality follows from the non-admissibility of $\mcal{N}$.
Since $\eta < 1$, rearranging gives
\begin{align} 
    \dist(\mcal{N},\mcal{Q}) 
        & \leq \frac{\frac{1}{2}(1+\eta)\left( \frac{1}{\eta} +1 \right)}{1 - \frac{1}{2}(1 + \eta)} \diam(\mcal{N}) 
        \nonumber \\
        & = \frac{(1 + \eta) \left( \frac{1}{\eta} + 1\right)}{1 - \eta} \diam(\mcal{N}). 
        \label{eqn:distNQ}
\end{align}
On the other hand, $\mcal{Q} \not \subset \mcal{N}_\delta$ implies that some point of $\mcal{Q}$ has distance greater than $\delta$ from $\mcal{N}$, and therefore $\delta \leq \dist(\mcal{Q}, \mcal{N}) + \diam(\mcal{Q})$. 
Combining this and the definition of the padding factor $\delta$ with the estimate from Equation~\ref{eqn:distNQ} yields
\begin{align}  
    \frac{(1 + \eta) \left( \frac{1}{\eta} + 1\right)}{1 - \eta} \diam(\mcal{N})
        & \geq \dist(\mcal{Q}, \mcal{N}) 
        \nonumber \\
        & \geq \delta - K_\text{upper} \diam(\mcal{N}) 
        \nonumber \\
        & = \left(  \frac{(1 + \eta) \left( \frac{1}{\eta} + 1\right)}{1 - \eta} + 1 \right)\diam(\mcal{N}) \nonumber 
\end{align}
which is a contradiction. 
Therefore, every admissible cube $\mcal{Q}$ satisfying $\tilde{\mcal{Q}} \cap \mcal{N} \neq \emptyset$ lies in $\mcal{N}_\delta$. 

The padded domain $\mcal{N}_\delta$ is contained in the cube concentric with $\mcal{N}$ having side length $\mathrm{len}(\mcal{N}) + 2 \delta$. 
Define this cube $\mcal{N}_{\delta,\text{cube}}$ so $\mcal{N}_\delta \subset \mcal{N}_{\delta,\text{cube}}$.
Moreover, the lower bound in Equation~\ref{diamQreltodiamNmod} yields
\[
    \Vol{(\mcal{Q})} = \left(\frac{\diam(\mcal{Q})}{\sqrt{d}}\right)^d \geq \left(\frac{K_\text{lower}\diam(\mcal{N})}{\sqrt{d}} \right)^d.
\]
Since distinct leaf cubes have disjoint interiors, a volume comparison gives 
\begin{align} 
    |\{ \mcal{X} \times \mcal{Q} \in P_\text{adm}: \mcal{N} \cap \tilde{\mcal{Q}} \neq \emptyset \}| 
        & \leq \frac{\mathrm{Vol}(\mcal{N}_{\delta,\text{cube}})}{\left(\frac{K_\text{lower}\diam(\mcal{N})}{\sqrt{d}} \right)^d}  \leq \frac{\left( \frac{\diam(\mcal{N})}{\sqrt{d}} + 2 \delta  \right)^d}{\left(\frac{K_\text{lower}\diam(\mcal{N})}{\sqrt{d}} \right)^d} \nonumber \\
        & \leq \frac{\left(\frac{\left[1 + 2 \sqrt{d}\left(  \frac{(1 + \eta) \left( \frac{1}{\eta} + 1\right)}{1 - \eta} + 1 + K_\text{upper} \right)\right]\diam(\mcal{N})}{\sqrt{d}}\right)^d}{\left(\frac{K_\text{lower}\diam(\mcal{N})}{\sqrt{d}}  \right)^d} 
        \nonumber \\
        & \leq \left(\frac{[( 1 + 2 \sqrt{d}\left(  \frac{(1 + \eta) \left( \frac{1}{\eta} + 1\right)}{1 - \eta} + 1 + K_\text{upper} \right)]}{K_\text{lower}}\right)^d \nonumber \\
        & =: K_{W,\text{non}}. \nonumber
\end{align}

It remains to establish the lower and upper bounds in Equation~\ref{diamQreltodiamNmod}. 
Since $\mcal{N}$ is terminal non-admissible, it lies at the finest level $n_h$, whereas every admissible leaf lies at some level $L \leq n_h$. Hence, $\diam(\mcal{Q}) \geq \diam(\mcal{N})$, so we may take $K_\text{lower} =1$. 
For the upper bound, admissibility of $\mcal{Q}$, non-admissibility of $\mcal{N}$ and Equation~\ref{eqn:distNQ} give
\begin{align} 
    \diam(\mcal{Q}) 
        & \leq \eta \dist(\mcal{X}, \mcal{Q}) \nonumber \\
        & \leq \eta \left( \dist(\mcal{X}, \mcal{N}) + \diam(\mcal{N}) + \dist(\mcal{N}, \mcal{Q}) \right) \nonumber \\
        & \leq \eta \left( \left( \frac{1}{\eta} + 1 \right) \diam(\mcal{N}) + \dist(\mcal{N}, \mcal{Q}) \right) \nonumber \\
        & \leq \left(1 + \eta\right) \diam(\mcal{N}) + \eta \left( \frac{(1+\eta) \left( \frac{1}{\eta} + 1 \right)}{1-\eta} \right) \diam(\mcal{N}) \nonumber \\
        & = \frac{2(1+\eta)}{1 - \eta} \diam(\mcal{N}). \nonumber
\end{align} 
Thus, $K_\text{upper} = \frac{2(1+\eta)}{1-\eta}$.   
\end{proof}

To summarize, for each subdomain $\mcal{X} \subset \Omega$ with $\mathrm{diam}(\mcal{X}) \leq h$, Algorithm~\ref{alghierarchicaldecomp} produces a hierarchical decomposition of the slab $\mcal{X} \times \Omega$ into admissible domains and terminal non-admissible domains. 
The number of terminal non-admissible domains is uniformly bounded by a constant $K_{d,\eta}$, while the number of admissible domains grows logarithmically in the inverse mesh size.
Moreover, the enlarged admissible domains used in the local approximation estimates have uniformly bounded overlap, quantified by constants $K_{W,\text{adm}}$ and $K_{W,\text{non}}$. 
These geometric estimates are the ingredients needed in the next section to construct and analyze the low-dimensional row approximation spaces. 

\subsection{Row Space Construction for \texorpdfstring{$\bld{A}_h$}{A\_h}}
\label{rowspaceconstructionsection}
This section constructs, for each row of $\bld{A}_h$, a low-dimensional subspace $\mcal{V}_i \subset \R^{|\mcal{T}_h|}$ that approximates the transpose of the row $((\bld{A}_h)_{i,:} = \theta_i^\star \mcal{A}J_h)^T$. 
The construction uses the hierarchical partition of the row domain $\mcal{X}_i \times \Omega$ developed in Section \ref{hierSAdecompsection}. 
The row subspace includes two types of basis vectors. On terminal non-admissible domains, we include the canonical coordinate vectors corresponding to the finite elements whose supports non-trivially intersect the non-admissible region. The number of such vectors is uniformly bounded, as proven below in Lemma~\ref{lemmanumofFEtocoverNA}. 
On admissible domains restricted to convex components of $\Omega$, we include the discrete coefficient vectors generated by a basis for the space of polynomials of total degree at most a parameter $m$. 
The local approximation result from Corollary~\ref{sepexpansionwithL2orthogonal} implies that choosing $m$ above a logarithmic threshold in the inverse error tolerance controls the local approximation error to that desired error tolerance. 
The explicit construction of an orthonormal basis of such a subspace is given in Algorithm~\ref{algconstructrowsubspace} with details of the approximation in Lemma~\ref{lemmarowsubspaceconstruction}.  
Algorithm~\ref{algconstructrowsubspace} takes in as inputs: the admissible cubes $\{\mcal{Y}_{i,j}\}_{j=1}^{N_\text{adm}(i)}$ and the non-admissible cubes $\{\mcal{N}_{i,p}\}_{p=1}^{N_\text{non}(i)}$ from a hierarchical partition of $\mcal{X}_i \times \Omega$ and a polynomial degree parameter $m$. The algorithm forms a matrix with the discrete coefficient vectors of local polynomial bases and the canonical coordinate vectors, and performs a $QR$ factorization to return an orthonormal basis of this subspace. 
\begin{algorithm}
\caption{Row Subspace Construction for $(\bld{A}_h)_{i,:}^T$} 
\label{algconstructrowsubspace}

\begin{algorithmic}[1]

\Procedure{Row-Subspace Construction}{$\{\mcal{Y}_{i,j}\}_{j=1}^{N_\text{adm}(i)}, \{\mcal{N}_{i,p}\}_{p=1}^{N_\text{non}(i)}, m$}

\State Let $J_h$ be the finite element mapping operator 

\State $\text{Let} \ I := \Big\{k^\prime: \mathrm{supp}(\theta_{k^\prime}) \cap \left( \bigcup_{p=1}^{N_\text{non}(i)} \mcal{N}_{i,p} \cap \Omega \right) \neq \emptyset \Big\}$ \color{gray} \algorithmiccomment{non-admissible region}

\color{black}
\For {$j = 1,2, ..., N_\text{adm}(i)$} \do

\color{gray} \algorithmiccomment{admissible cube $\mcal{Y}_{i,j}$}
\color{black}
\For {$n = 1,2, ..., N_\text{cc}$} \do

\color{gray} \algorithmiccomment{convex component $\mcal{C}_n$}
\color{black}

\If {$\mcal{Y}_{i,j} \cap \mcal{C}_n$ has positive measure} 

\State Form a polynomial basis $\{\phi^\alpha_{\mcal{Y}_{i,j} \cap \mcal{C}_n}\}_{|\alpha| \leq m}$ of $\mcal{P}_{m}^d(\mcal{Y}_{i,j} \cap \mcal{C}_n)$

\EndIf
\EndFor
\EndFor

\color{black}
\State Form a matrix $\bld{\tilde{V}}$ with columns

$\left\{J_h^* (\phi^\alpha_{\mcal{Y}_{i,j} \cap \mcal{C}_n} \mathds{1}_{\mcal{Y}_{i,j} \cap \mcal{C}_n})\right\}_{\substack{1\leq j \leq N_{\text{adm}}(i) \\ 1\leq n \leq N_{\text{cc}} \\ |\alpha| \leq m }} \cup \{ e_{k^\prime}\}_{k^\prime \in I}$ \color{gray} \algorithmiccomment{spanning vectors of $\mcal{V}_i$}

\color{black}

\State $\text{Compute} \ \bld{\tilde{V}} \bld{\Pi} = \bld{Q}\bld{R}$ \color{gray} \algorithmiccomment{rank-revealing $\bld{QR}$ with column pivoting}

\color{black}
\Return O.N. basis $\bld{V}_i := \bld{Q}(:,1:\text{rank}(\widetilde{\bld{V}}))$

\EndProcedure
\end{algorithmic}
\end{algorithm}

\begin{remark}
Observe that Algorithm~\ref{algconstructrowsubspace} can construct an orthonormal basis of this subspace using any $\binom{m+d}{d}$ polynomials that span $\mcal{P}_{m}^d(\mcal{Y}_{i,j} \cap \mcal{C}_n)$ and does not require constructing the $L^2(\mcal{Y}_{i,j} \cap \mcal{C}_n)$-orthogonal polynomials from Corollary~\ref{sepexpansionwithL2orthogonal} which could be computationally difficult depending on the domain. 
A favorable choice may be to use a basis of $L^2(\mcal{Y}_{i,j})$-orthogonal polynomials of the cube and then restrict them onto the convex component $\mcal{C}_n$ for ease of construction and better conditioning when orthonormalizing.
\end{remark}

Lemma~\ref{lemmarowsubspaceconstruction} establishes the two properties of the row subspaces needed to apply Algorithm~\ref{alg:amselgen}.
It shows that the orthogonal projection onto the subspace constructed by Algorithm~\ref{algconstructrowsubspace} approximates each row of $\bld{A}_h$ to the prescribed accuracy, while the dimension of the subspace remains uniformly controlled over each row.
We state the result first. 

\begin{lemma}
\label{lemmarowsubspaceconstruction}
Assume Assumptions 1 \ref{assumedomaincondition}, \ref{assumefemmesh}, \ref{assumerescaling} and 2 \ref{assumeglobalregularity}, \ref{assumeasymptotic smoothness}. 
Let $0< \eta < \text{min}\{\frac{1}{\gamma d}, 1\}$ and $h \in \mcal{H}$. 
Let constants $K_{W,\text{adm}}$ and $K_{W,\text{non}}$ be the enlargement overlap constants from Lemmas~\ref{CZadm} and \ref{CZnon}, respectively.
For $\varepsilon^{\prime} > 0$, let $m \in \mathbb{N}$ satisfy $m \geq m_{\varepsilon^\prime}:= \max \Bigg\{ 1, \Bigg\lceil \frac{\log{\left( \frac{C}{\varepsilon^\prime (1-\gamma \eta d)}\right)}}{-\log{(\gamma \eta d)}} \Bigg\rceil \Bigg\}$. 
For each $i \in \{1:|\mcal{T}_h|\}$,
let 
\[
    \left(\{\mcal{Y}_{i,j}\}_{j=1}^{N_\text{adm}(i)}, \{\mcal{N}_{i,p}\}_{p=1}^{N_\text{non}(i)}\right)
\]
be the output of Algorithm~\ref{alghierarchicaldecomp} applied to $(\mcal{X}_i, [0,1]^d, \eta, h)$.
Let $\bld{V}_i$ be the orthonormal basis matrix returned by Algorithm~\ref{algconstructrowsubspace} applied to 
\[
    \left(\{\mcal{Y}_{i,j}\}_{j=1}^{N_\text{adm}(i)}, \{\mcal{N}_{i,p}\}_{p=1}^{N_\text{non}(i)},m\right),
\]
and define
\[
    \mcal{V}_i := \mathrm{range}(\bld{V}_i), \quad \bld{P}_i := \bld{V}_i \bld{V}_i^T.
\]
Writing the transpose of the $i$-th row as $a_i^T = (\bld{A}_h)_{i,:} = \theta_i^\star \mcal{A} J_h$, the projection error satisfies
\[
    \| a_i - \bld{P}_i a_i\|_2 
        \leq \varepsilon^\prime \sqrt{N_{cc} \text{max}\{K_{W,\text{adm}}, K_{W,\text{non}}\}} \| \kappa\|_{L^2(\mcal{X}_i \times \Omega)}
\]
and the dimension of the subspace satisfies
\[
    \mathrm{dim}\mcal{V}_i \leq k = \mcal{O}\left(\log(h^{-1})m^d\right) + \mcal{O}(1),
\]
where $k$ is independent of $i$ and depends only on $d, N_{cc}, \eta, h, m, c_\text{quasi}$ and $ \sigma_\#$.
\end{lemma}

Before proving Lemma~\ref{lemmarowsubspaceconstruction}, we first prove an auxiliary lemma used in its proof to bound the number of canonical coordinate vectors contributed by the terminal non-admissible region. 

\begin{lemma}
\label{lemmanumofFEtocoverNA} 
Assume Assumptions 1\ref{assumefemmesh} and 1\ref{assumerescaling}. 
Let $h \in \mcal{H}$ and let $\mcal{X} \subset \Omega$ such that $\diam(\mcal{X}) \leq h$. 
Let $\eta > 0$.
Let $(P_\text{adm}, P_\text{non})$ be the output of Algorithm~\ref{alghierarchicaldecomp} applied to $(\mcal{X}, [0,1]^d, \eta, h)$. 
Define the terminal non-admissible region as
\[
    \mcal{R}_{\text{non}, \mcal{X}} 
        := \bigcup_{\mcal{X} \times \mcal{N} \in P_\text{non}} \mcal{N} \cap \Omega.
\] 
Let $N_\text{non,mesh}(\mcal{X})$ denote the number of mesh cells whose supports intersect the terminal non-admissible region, and $N_\text{non, FE}(\mcal{X})$ denote the number of finite elements corresponding to those mesh cells. 
Then, both numbers can be bounded by
\[
    N_\text{non,FE}(\mcal{X}) 
        = N_\text{non,mesh}(\mcal{X}) 
        \leq  \omega_d \frac{(3 + \frac{1}{\eta})^d}{K_\text{mesh} } =: K_\text{non,FE} 
\]
where we define the constant $K_\text{non,FE}$ as this upper bound.
\end{lemma}

\begin{proof}
Let $\mcal{C} \in \mcal{T}_h$ such that $\mcal{C}\cap \mcal{R}_{\text{non},\mcal{X}} \neq \emptyset$. 
Then, there is a point $z \in \mcal{C} \cap (\mcal{N} \cap \Omega)$ for some terminal non-admissible cube $\mcal{N}$. 
For any $y \in \mcal{C}$, 
\[
    \dist(y,\mcal{X}) 
        \leq |y-z| + \dist(z, \mcal{X}) 
        \leq h +  \dist(\mcal{N}, \mcal{X}) + \diam(\mcal{N}) 
        < \left( 1 + \frac{1}{\eta} \right) \diam(\mcal{N}) + h. 
\]
Since terminal non-admissible cubes satisfy $\diam(\mcal{N}) \leq h$, this gives $\dist(y,\mcal{X}) < \left( 2 + \frac{1}{\eta} \right)h$. 
Then since $\mcal{X}$ is contained in a ball of radius $h$, every such mesh cell lies in a fixed ball of radius  $(3 + \frac{1}{\eta})h$. 
The number of such mesh cells can therefore be bounded by the volume of this ball divided by the minimum mesh-cell volume:
\begin{align} 
    N_\text{non,mesh}(\mcal{X}) 
        & \leq \frac{\text{Vol}\left(B_{(3 + \frac{1}{\eta})h} \right)}{ \min\{\text{Vol}(\mcal{C}^\prime): \mcal{C}^\prime \in \mcal{T}_h\}} 
            \leq \frac{\omega_d \left( (3 + \frac{1}{\eta}) h\right)^d}{K_\text{mesh} h^d} 
            = \omega_d \frac{(3 + \frac{1}{\eta})^d}{K_\text{mesh} }\nonumber
\end{align}
recalling from the proof of Lemma~\ref{numberofmeshcells} that the volume of any mesh cell from $\mcal{T}_h$ is lower bounded by $K_\text{mesh} h^d$. 
The constant $K_\text{mesh}$ is explicitly provided in the proof of Lemma~\ref{numberofmeshcells}. 
Under Assumption~1\ref{assumefemmesh}, each finite element mesh cell has a single basis element so $N_\text{non, FE} (\mcal{X}) = N_\text{non, mesh} (\mcal{X})$. 
Thus the number of associated piecewise-constant basis functions is bounded by the same constant.
\end{proof}

The proof of Lemma~\ref{lemmarowsubspaceconstruction} now follows.

\begin{proof}[Lemma~\ref{lemmarowsubspaceconstruction}]
Let $\varepsilon^{\prime} > 0$ and choose $m \geq m_{\varepsilon^\prime}$. Let $i \in \{1:|\mcal{T}_h|\}$ and let $\mcal{X}_i$ denote the mesh cell supporting the finite element $\theta_i$. 
Algorithm~\ref{alghierarchicaldecomp}'s output $\left(\{\mcal{Y}_{i,j}\}_{j=1}^{N_\text{adm}(i)}, \{\mcal{N}_{i,p}\}_{p=1}^{N_\text{non}(i)}\right)$ yields a disjoint covering, up to boundaries,
\[
    \mcal{X}_i \times \Omega 
        = \bigcup_{j=1}^{N_\text{adm}(i)} \left[ \mcal{X}_i \times (\mcal{Y}_{i,j} \cap \Omega) \right] \cup  \bigcup_{p=1}^{N_\text{non}(i)} \left[\mcal{X}_i \times (\mcal{N}_{i,p} \cap \Omega)\right],
\]
as seen in Equation~\ref{hierdecompeqn}.
We first consider the admissible part of this decomposition.
For each admissible cube $\mcal{Y}_{i,j}$, the decomposition 
$
    \overline{\Omega} = \bigcup_{n=1}^{N_{cc}} \mcal{C}_n
$ from Assumption~1\ref{assumedomaincondition} 
gives
\[
    \mcal{Y}_{i,j} \cap \overline{\Omega} 
        = \bigcup_{n=1}^{N_\text{cc}} (\mcal{Y}_{i,j} \cap \mcal{C}_n)
\]
where the intersections are pairwise disjoint up to boundaries. In Algorithm~\ref{algconstructrowsubspace}, we omit intersections of measure zero. For every remaining intersection, $\operatorname{int}(\mcal{Y}_{i,j} \cap \mcal{C}_n)$ is a bounded convex domain with Lipschitz boundary by Remark \ref{rmkconvexlocallipschitz}. 
Moreover, restricting an admissible cube to a convex component preserves admissibility. 
Indeed, 
\[
    \eta \dist(\mcal{X}_i,\mcal{Y}_{i,j} \cap \mcal{C}_n) \geq \eta \dist(\mcal{X}_i,\mcal{Y}_{i,j}) \geq \diam(\mcal{Y}_{i,j}) \geq  \diam(\mcal{Y}_{i,j} \cap \mcal{C}_n).
\]
Hence, each pair $\mcal{X}_i \times (\mcal{Y}_{i,j} \cap \mcal{C}_n)$ is $\eta$-admissible.
Since $m \geq m_{\varepsilon^\prime}$, Corollary~\ref{sepexpansionwithL2orthogonal} applies on the interior of each such intersection.
Since the boundary of a bounded convex set with nonempty interior has Lebesgue measure zero, we identify the corresponding $L^2$ spaces and norms with those on $\mcal{Y}_{i,j} \cap \mcal{C}_n$.
Let 
\[
    \{\phi^\alpha_{\mcal{Y}_{i,j} \cap \mcal{C}_n}\}_{|\alpha| \leq m}
\] 
be an $L^2(\mcal{Y}_{i,j} \cap \mcal{C}_n)$-orthonormal basis of $\mcal{P}^d_m(\mcal{Y}_{i,j} \cap \mcal{C}_n)$. 
Then, for almost every $x \in \mcal{X}_i$ and $y \in \mcal{Y}_{i,j} \cap \mcal{C}_n$,
\[
    \kappa(x,y) 
        = \sum_{|\alpha| \leq m} h^{\alpha,n}_{i,j}(x) \phi^\alpha_{\mcal{Y}_{i,j} \cap \mcal{C}_n} (y)  + R^m_{\mcal{Y}_{i,j} \cap \mcal{C}_n}(x,y)
\] 
where $ h^{\alpha,n}_{i,j}(x) := \langle \kappa(x, \cdot), \phi^\alpha_{\mcal{Y}_{i,j} \cap \mcal{C}_n} \rangle_{L^2(\mcal{Y}_{i,j} \cap \mcal{C}_n)}$ and 
\[
    \| R^{m}_{\mcal{Y}_{i,j} \cap \mcal{C}_n}(x,\cdot)\|_{L^2(\mcal{Y}_{i,j} \cap \mcal{C}_n)} 
        \leq \varepsilon^\prime \| \kappa(x,\cdot)\|_{L^2(\widehat{\mcal{Y}_{i,j} \cap \mcal{C}_n})}.
\] 
Extend each local polynomial by zero outside $\mcal{Y}_{i,j} \cap \mcal{C}_n$, and define
\[v_{i,j}^{\alpha,n}(y) := \phi^\alpha_{\mcal{Y}_{i,j} \cap \mcal{C}_n} (y)\mathds{1}_{\mcal{Y}_{i,j} \cap \mcal{C}_n}(y).\]
Combining the local approximations over all admissible domains gives, for almost every $(x,y) \in \mcal{X}_i \times \Omega$,
\begin{align} 
    \kappa(x,y) 
        = & \sum_{j=1}^{N_\mathrm{adm}(i)} \sum_{n=1}^{N_\text{cc}} \left[ \sum_{|\alpha| \leq m} h^{\alpha,n}_{i,j}(x) v_{i,j}^{\alpha,n}(y) +  R^{m}_{\mcal{Y}_{i,j} \cap \mcal{C}_n}(x,y) \mathds{1}_{\mcal{Y}_{i,j} \cap \mcal{C}_n}(y) \right] \nonumber \\
        & \qquad + \sum_{p =1}^{N_\mathrm{non}(i)}  \kappa(x,y) \mathds{1}_{\mcal{N}_{i,p}\cap \Omega}(y).
\label{kernelexpression}
\end{align}
We now use this decomposition of the kernel on $\mcal{X}_i \times \Omega$ to obtain a corresponding decomposition of the $i$-th row of $\bld{A}_h$. 
For $k^\prime \in \{1, \ldots, |\mcal{T}_h|\}$, 
\[
    (\bld{A}_h)_{i,k^\prime} 
        = \int_{\mcal{X}_i}  \int_\Omega \kappa(x,y) (J_h \bld{e}_{k^\prime})(y) dy \theta_{i}(x) dx.
\]
Substituting in the kernel representation from Equation~\ref{kernelexpression} and using
\[
    (J_h^\star v)_{k^\prime} = \int_\Omega v(y) (J_h e_{k^\prime})(y)dy
\]
gives that the transpose of the $i$-th row can be represented as
\begin{align}
    a_i = (\bld{A}_h)_{i,:}^T
        &= \sum_{j=1}^{N_\mathrm{adm}(i)} \sum_{n=1}^{N_{cc}} \sum_{|\alpha| \leq m} \langle h^{\alpha,n}_{i,j}, \theta_{i}\rangle_{L^2(\mcal{X}_i)} J_h^* v^{\alpha,n}_{i,j} \nonumber \\
        & \qquad + J_h^* \left( \sum_{j=1}^{N_\mathrm{adm}(i)} \sum_{n=1}^{N_{cc}}  \int_{\mcal{X}_i} R^{m}_{\mcal{Y}_{i,j} \cap \mcal{C}_n} (x,y) \mathds{1}_{\mcal{Y}_{i,j} \cap \mcal{C}_n}(y)\theta_{i}(x) dx \right) \nonumber \\
        & \qquad + J_h^* \left( \int_{\mcal{X}_i} \kappa(x,y) \mathds{1}_{ \bigcup_{p=1}^{N_\mathrm{non}(i)}  \mcal{N}_{i,p}\cap \Omega}(y)  \theta_{i}(x) dx  \right). \nonumber 
\end{align}
Motivated by this decomposition, define the set of discrete coefficient vectors generated by the polynomial basis functions on each admissible domain, together with the canonical coordinate vectors for the finite elements that intersect the terminal non-admissible region, by
\begin{equation} 
    \mcal{S}_i := \big\{J_h^*(v^{\alpha,n}_{i,j}) \big\}_{j,\alpha,n}  \bigcup  \Bigg\{e_{k^\prime}: \text{supp}(\theta_{k^\prime}) \cap \left(\bigcup_{p=1}^{N_\text{non}(i)}   \mcal{N}_{i,p}\cap \Omega \right) 
        \neq \emptyset \Bigg\},
\label{eqnrowsubspaceset}
\end{equation}
and let 
\[
    \mcal{V}_i := \operatorname{span}(\mcal{S}_i).
\]
By construction, the first term in the decomposition of $a_i$ lies in $\mcal{V}_i$ by definition. 
Moreover, the $k^\prime$-th entry of the third term vanishes whenever 
\[
    \mathrm{supp}(\theta_{k^\prime}) \cap \left( \bigcup_{p=1}^{N_\text{non}(i)} (\mcal{N}_{i,p} \cap \Omega \right) = \emptyset,
\]
so the third term also lies in $\mcal{V}_i$. 

Algorithm~\ref{algconstructrowsubspace} constructs an orthonormal basis matrix $\bld{V}_i$ for $\mcal{V}_i$, so that $\text{range}(\bld{V}_i) = \mcal{V}_i$.
Let $N_\text{non,FE}(i)$ denote the number of canonical coordinate vectors included in $\mcal{S}_i$ from Equation~\ref{eqnrowsubspaceset}. 
Lemma~\ref{lemmanumofFEtocoverNA} gives
\[
    N_\text{non,FE}(i) 
        = \bigg|  \bigg\{e_{k^\prime}: \mathrm{supp}(\theta_{k^\prime}) \cap \left(\bigcup_{p=1}^{N_\mathrm{non}(i)}   \mcal{N}_{i,p}\cap \Omega \right) 
        \neq \emptyset \bigg\} \bigg| 
        \leq K_\text{non,FE}.
\]
Since each admissible cube contributes at most $N_{cc}$ polynomial spaces, each of dimension $\binom{m+d}{d}$,
\[
    \dim \mcal{V}_i  \leq |\mcal{S}_i| \leq  \left( \sum_{j=1}^{N_\text{adm}(i)} \sum_{n=1}^{N_{cc}} \binom{m+d}{d} \right) + N_\text{non,FE}(i).
\]
Using the bound on $N_\text{adm}(i)$, we obtain 
\[
    \dim \mcal{V}_i 
        \leq N_{cc} \cdot 2^d K_{d,\eta} \big\lceil \log_2(\sqrt{d}h^{-1}) \big \rceil \cdot \binom{m+d}{d} + K_\text{non,FE} =:k. 
\]

We now bound the approximation error of the $i$-th row. 
Let $\bld{P}_i = \bld{V}_i \bld{V}_i^T$ denote the orthogonal projector onto the subspace $\mcal{V}_i$. 
By construction, only the polynomial remainders on the admissible domains contribute to the projection error. 
Define the total admissible remainder by 
\[
    R_i(x,y)
        := \sum_{j=1}^{N_\mathrm{adm}(i)} \sum_{n=1}^{N_{cc}} R^{m}_{\mcal{Y}_{i,j} \cap \mcal{C}_n} (x,y) \mathds{1}_{\mcal{Y}_{i,j} \cap \mcal{C}_n}(y).
\] 
Since the sets $\mcal{Y}_{i,j} \cap \mcal{C}_n$ are pairwise disjoint up to measure-zero boundaries,
\[
    \| R_i\|^2_{L^2(\mcal{X}_i \times \Omega)} 
        = \sum_{j=1}^{N_\mathrm{adm}(i)} \sum_{n=1}^{N_{cc}} \| R^{m}_{\mcal{Y}_{i,j} \cap \mcal{C}_n}\|^2_{L^2(\mcal{X}_i \times \mcal{Y}_{i,j} \cap \mcal{C}_n)}.
\]
For almost every $y \in \Omega$, define 
\[
    g_i(y) 
        :=  \int_{\mcal{X}_i} R_i(x,y) \theta_{i}(x) dx.
\] 
By Cauchy-Schwarz and the normalization $\|\theta_i\|_{L^2(\mcal{X}_i)} = 1$, 
\[
    \| g_i\|_{L^2(\Omega)} \leq \| R_i\|_{L^2(\mcal{X}_i \times \Omega)}.
\]
Since the remaining first and third terms in $a_i$ lie in $\mcal{V}_i$, the best-approximation property of the orthogonal projector gives
\begin{align} 
    \| a_i - \bld{P}_i a_i\|_2^2 
        & \leq \|   J_h^* g_i\|_2^2  
        \leq \| g_i\|_{L^2(\Omega)}^2 \nonumber \\
        & \leq   \sum_{j=1}^{N_\mathrm{adm}(i)} \sum_{n=1}^{N_{cc}} \| R^{m}_{\mcal{Y}_{i,j} \cap \mcal{C}_n}\|^2_{L^2(\mcal{X}_i \times (\mcal{Y}_{i,j} \cap \mcal{C}_n))}  \nonumber \\ 
        & \leq \varepsilon^{\prime^2} \sum_{j=1}^{N_\text{adm}(i)} \sum_{n=1}^{N_\text{cc}} \| \kappa\|^2_{L^2(\mcal{X}_i \times (\widehat{\mcal{Y}_{i,j} \cap \mcal{C}_n}))}  \nonumber 
\end{align}
where the final inequality follows from Corollary~\ref{sepexpansionwithL2orthogonal}. 
To bound the sum of the enlarged-domain norms, it is convenient to replace the component-dependent enlargements by enlargements of the original admissible cubes. 
For an admissible cube $\mcal{Y}_{i,j}$, define
\[
    \tilde{\mcal{Y}}_{i,j}
        := \{z^\prime \in [0,1]^d: \dist(z^\prime, \mcal{Y}_{i,j}) 
        \leq \frac{1}{2}(1 + \eta) \dist(\mcal{X}_i, \mcal{Y}_{i,j})\}.
\]
Lemma~\ref{lemmainclusioninlargerdomain} shows that 
\[
    \widehat{\mcal{Y}_{i,j} \cap \mcal{C}_n} \subset \tilde{\mcal{Y}}_{i,j} \cap \Omega
\]
for every $n$. 
Therefore, 
\[
    \|\kappa \|^2_{L^2(\mcal{X}_i \times (\widehat{\mcal{Y}_{i,j} \cap \mcal{C}_n}))} \leq \|\kappa \|^2_{L^2(\mcal{X}_i \times (\tilde{\mcal{Y}}_{i,j}\cap \Omega))},
\]
and hence
\[
    \|a_i - \bld{P}_i a_i \|^2_2 \leq \varepsilon^{\prime^2} N_{cc} \sum_{j=1}^{N_\text{adm}(i)} \|\kappa\|^2_{L^2(\mcal{X}_i \times (\tilde{\mcal{Y}}_{i,j}\cap \Omega))}.
\]
Since the admissible and terminal non-admissible cubes partition $\Omega$ up to measure-zero boundaries,
\begin{align}
    \sum_{j=1}^{N_\text{adm}(i)} \|\kappa\|^2_{L^2(\mcal{X}_i \times (\tilde{\mcal{Y}}_{i,j}\cap \Omega))} & \leq  
    \sum_{l=1}^{N_\text{adm}(i)} \| \kappa\|^2_{L^2(\mcal{X}_i \times (\mcal{Y}_{i,l} \cap \Omega))} \sum_{j=1}^{N_\text{adm}(i)} \mathds{1}_{\tilde{\mcal{Y}}_{i,j} \cap \mcal{Y}_{i,l} 
            \neq \emptyset}  \nonumber \\
        & \quad +  \sum_{p=1}^{N_\text{non}(i)} \| \kappa\|^2_{L^2(\mcal{X}_i \times (\mcal{N}_{i,p} \cap \Omega))} \sum_{j=1}^{N_\text{adm}(i)} \mathds{1}_{\tilde{\mcal{Y}}_{i,j} \cap \mcal{N}_{i,p} \neq \emptyset}.  \nonumber \\
        & \leq  \mathrm{max}_{l \in \{1, ..., N_\text{adm}(i)\}} \bigg|\{j : \mcal{Y}_{i,l} \cap \tilde{\mcal{Y}}_{i,j} 
            \neq \emptyset\}\bigg| \sum_{l=1}^{N_\text{adm}(i)} \| \kappa\|^2_{L^2(\mcal{X}_i \times (\mcal{Y}_{i,l} \cap \Omega))} \nonumber \\
        & \quad +    \mathrm{max}_{p \in \{1, ..., N_\text{non}(i)\}} \bigg|\{j : \mcal{N}_{i,p} \cap \tilde{\mcal{Y}}_{i,j} 
            \neq \emptyset\}\bigg| \sum_{p=1}^{N_\text{non}(i)} \| \kappa\|^2_{L^2(\mcal{X}_i \times (\mcal{N}_{i,p} \cap \Omega))}. \nonumber
\end{align}
The overlap lemmas from Section~\ref{hierSAdecompsection} provide uniform bounds for the multiplicities above. 
Lemmas \ref{CZadm} and \ref{CZnon} give 

\begin{equation*}
    \mathrm{max}_{l \in \{1, ..., N_\text{adm}(i)\}} \bigg|\{j : \mcal{Y}_{i,l} \cap \tilde{\mcal{Y}}_{i,j} 
        \neq \emptyset\}\bigg| 
        \leq K_{W,\text{adm}}
\end{equation*} 
and 
\begin{equation*}
    \mathrm{max}_{p \in \{1, ..., N_\text{non}(i)\}} \bigg|\{j : \mcal{N}_{i,p} \cap \tilde{\mcal{Y}}_{i,j} 
        \neq \emptyset\}\bigg| 
        \leq K_{W,\text{non}}
\end{equation*}
where both constants depend only on $d$ and $\eta$. 
It follows that
\[
    \|a_i - \bld{P}_i a_i\|_2^2 \leq \varepsilon^{\prime^2} N_{cc} \text{max}\{K_{W,\text{adm}}, K_{W,\text{non}}\} \left(\sum_{l=1}^{N_\text{adm}(i)} \| \kappa\|^2_{L^2(\mcal{X}_i \times (\mcal{Y}_{i,l} \cap \Omega))} + \sum_{p=1}^{N_\text{non}(i)} \| \kappa\|^2_{L^2(\mcal{X}_i \times (\mcal{N}_{i,p} \cap \Omega))} \right).
\]
Since these domains partition $\mcal{X}_i \times \Omega$ up to measure-zero boundaries, the row-wise approximation estimate yields 
\[
    \| a_i - \bld{P}_i a_i\|_2 
        \leq \varepsilon^\prime  \sqrt{N_{cc} \text{max}\{K_{W,\text{adm}}, K_{W,\text{non}}\} } \ \| \kappa\|_{L^2(\mcal{X}_i \times \Omega)}.
\]  
\end{proof} 

\subsection{Proof of Theorem~\ref{bigtheorem}}
\label{sectheoremproof}
We now apply Algorithm~\ref{alg:amselgen} to the task of approximating $\bld{A}_h$ by constructing a matrix approximation where each row lives in a pre-determined subspace $\mcal{V}_i$ from Section \ref{rowspaceconstructionsection} and leverage its approximation error to prove Theorem~\ref{bigtheorem}.

\begin{proof}[Theorem~\ref{bigtheorem}]
Let $c$ be the constant from Lemma~\ref{lemmaglobalapproxfe}, and choose 
\[
    \varepsilon_0 := \min\bigg\{\frac{1}{2c}, \frac{1}{e},\sqrt{h_1}\bigg\}.
\]
Fix $0 < \varepsilon < \varepsilon_0$, and choose the mesh size 
\[
    h = \max\{h^\prime \in \mcal{H}: h^\prime \leq \varepsilon^\frac{2}{\rho} \}.
\]
Since $h_n \rightarrow 0$ as $n \rightarrow \infty$, such a mesh size $h$ exists.
In particular,
\[
    ch^\rho \leq c \left( \varepsilon^\frac{2}{\rho} \right)^\rho < \frac{\varepsilon}{2}.
\]
Define $\varepsilon^\prime := \frac{\varepsilon \sqrt{\delta}}{2\sqrt{2\text{max}\{K_{W,\text{adm}}, K_{W,\text{non}}\} N_{cc}}}$. 
We first show that the assumed lower bound on $m$ is sufficient to apply Lemma~\ref{lemmarowsubspaceconstruction}. 
In particular, we need $m \geq m_{\varepsilon^\prime}:= \max \Bigg\{ 1, \Bigg\lceil \frac{\log{\left( \frac{C}{\varepsilon^\prime (1-\gamma \eta d)}\right)}}{-\log{(\gamma \eta d)}} \Bigg\rceil \Bigg\}$. 
If
\[
\log\!\left(\frac{C}{\varepsilon'(1-\gamma\eta d)}\right)\leq0,
\]
then \(m_{\varepsilon'}=1\), and the desired estimate is trivial as follows
\begin{align}
    m_{\varepsilon^\prime} 
        & = 1 \leq \underbrace{1}_{=:K_1} \cdot ( \log{\varepsilon^{-1}} + \log{\delta^{-\frac{1}{2}}} )\nonumber
\end{align}
since $\varepsilon < \frac{1}{e}$ and $\delta \in (0,1)$. 
Otherwise, 
\begin{align}
    m_{\varepsilon^\prime} 
        & = \Bigg\lceil \frac{\log{\left( \frac{C}{\varepsilon^{\prime}(1 - \gamma \eta d)}  \right)}}{-\log{(\gamma \eta d)}} \Bigg\rceil \nonumber \\
        & \leq \frac{1}{-\log{\gamma \eta d}} \log{\left( \frac{2C \sqrt{2 N_{cc}\text{max}\{K_{W,\text{adm}}, K_{W,\text{non}}\}}}{1 - \gamma \eta d} (\sqrt{\delta}\varepsilon)^{-1}\right)} + 1 
    \nonumber \\
        & \leq  \frac{1}{-\log{\gamma \eta d}} \left[ \log{\left( \frac{2C \sqrt{2 N_{cc}\text{max}\{K_{W,\text{adm}}, K_{W,\text{non}}\}}}{1 - \gamma \eta d}\right)} + \log{(\sqrt{\delta}\varepsilon)^{-1}} \right] + \log{(\sqrt{\delta}\varepsilon)^{-1}} \nonumber \\
        & \leq \underbrace{ \left( \frac{1}{-\log{\gamma \eta d}} \left[ \max \Bigg\{ \log{\left( \frac{2C\sqrt{2 N_{cc}\text{max}\{K_{W,\text{adm}}, K_{W,\text{non}}\}}}{1 - \gamma \eta d}\right)}, 0 \Bigg\} + 1 \right]+ 1 \right)}_{=:K_2} \log{(\sqrt{\delta}\varepsilon)^{-1}}. \nonumber
\end{align}
Define $K_0 := \max\{K_1,K_2\}$, and let $K \geq K_0$. 
Choose $m \geq K \log{(\sqrt{\delta}\varepsilon)^{-1}} \geq m_{\varepsilon^\prime}$. 
Lemma~\ref{lemmarowsubspaceconstruction} therefore, for every row $a_i^T = (\bld{A}_h)_{i,:}$, constructs a subspace $\mcal{V}_i$ with an orthogonal projector $\bld{P}_i$ satisfying
\[
    \| a_i - \bld{P}_i a_i\|_2 
        \leq \varepsilon^\prime \sqrt{N_{cc} \text{max}\{K_{W,\text{adm}}, K_{W,\text{non}}\}} \| \kappa\|_{L^2(\mcal{X}_i \times \Omega)}.
\] 
Moreover, 
\[
    \dim \mcal{V}_i 
        \leq k,
\]
where $k = 2^{d} K_{d,\eta} N_{cc} \big\lceil \log_2(\sqrt{d}h^{-1}) \big\rceil \cdot \binom{m+d}{d} + K_\text{non,FE}$.

We next show that the stated query bound implies $q \geq 2k + 1$. 
Since each row subspace contains at least the canonical basis vector associated with its corresponding mesh cell $\mcal{X}_i$, we have $k \geq 1$. 
Thus, the condition $q \geq 2 k + 1$ implies that $q \geq k + 2$, as required by Theorem~\ref{thm:amselgen}, and also gives $\frac{k}{q-k-1} \leq 1$.
Since $\rho < 1$ and $0 < \varepsilon < \frac{1}{e} < 1$,
\[
    \varepsilon^{2/\rho} \leq \varepsilon^2 < h_1,
\]
so $h = h_n$ for some $n \geq 2$ satisfying $h_n \leq \varepsilon^{2/\rho} < h_{n-1}$. 
The refinement condition in Assumption~1\ref{assumefemmesh} therefore gives  
\[
    c_\mcal{H} \varepsilon^\frac{2}{\rho} 
        < c_\mcal{H} h_{n-1} \leq h_n = h \leq \varepsilon^\frac{2}{\rho} < h_{n-1}.
\]
This relation provides a lower bound for $h$ in terms of $\varepsilon$ allowing us to bound
\begin{align}
    2k + 1 
        & =  2^{d+1} K_{d,\eta} N_{cc} \big\lceil\log_2{\left( \sqrt{d} h^{-1} \right)}\big\rceil  \binom{m+d}{d} + 2 K_\text{non,FE} + 1 \nonumber \\
        & \leq 2^{d+1} K_{d,\eta} N_{cc} \Bigg\lceil \log_2\left( \frac{\sqrt{d}}{c_\mcal{H}} \varepsilon^\frac{-2}{\rho} \right) \Bigg\rceil  \binom{m+d}{d}  + 2 K_\text{non,FE} + 1 \nonumber \\
        & \leq \frac{2^{d+1} K_{d,\eta} N_{cc} }{\log(2)} \left( \log{\frac{\sqrt{d}}{c_\mcal{H}}} + \frac{2}{\rho} \log{\varepsilon^{-1}} + \log{2}  \right)  \binom{m+d}{d}  + 2 K_\text{non,FE} + 1 \nonumber 
\intertext{
using that $\rho < 1$ and  $\log{(\varepsilon^{-1})} > 1$ since $\varepsilon < \frac{1}{e}$
}
        & \leq \frac{1}{\rho} \left[ \frac{2^{d+1}N_{cc} K_{d,\eta}}{\log(2)} \left( \log{\frac{\sqrt{d}}{c_\mcal{H}}} + 2 + \log{2} \right) \log{(\varepsilon^{-1})} \binom{m+d}{d}   + 2 K_\text{non,FE} + 1 \right] \nonumber \\
\intertext{
using that $\binom{m+d}{d} \leq \frac{(m+d)^d}{d!} \leq \frac{(2dm)^d}{d!}$ since $d\geq 1$ and $m \geq 1$
}
        & \leq \frac{1}{\rho} \left[ \frac{2^{2d+1}N_{cc} K_{d,\eta}d^d}{\log(2)d!} \left( \log{\frac{\sqrt{d}}{c_\mcal{H}}} + 2 + \log{2} \right) \log{(\varepsilon^{-1})} m^d   + 2 K_\text{non,FE} + 1 \right]. \nonumber 
\end{align}
Here we observe that $m\geq K_0 \log{(\delta^{-\frac{1}{2}}\varepsilon^{-1})}$ implies that $1 < \frac{m^d}{K_0^d} \log{\varepsilon^{-1}}$, yielding
\begin{align}
    2k + 1
        & \leq \frac{1}{\rho} \underbrace{\left[ \frac{2^{2d+1}N_{cc} K_{d,\eta}d^d}{\log(2)d!} \left( \log{\frac{\sqrt{d}}{c_\mcal{H}}} + 2 + \log{2} \right)   + \frac{2 K_\text{non,FE} + 1}{K_0^d}  \right]}_{=:K^\prime_0} m^d\log{(\varepsilon^{-1})} \nonumber.
\end{align}
Let $K^\prime \geq K^\prime_0$. 
Thus, taking a number of queries $q \geq \frac{K^\prime}{\rho} m^d \log{\varepsilon^{-1}}$ implies the condition $q \geq 2k + 1$.
We now apply Theorem~\ref{thm:amselgen}.
Using the same notation from Section \ref{amselgensection}, let 
$ \bld{P} [\bld{A}_h]$ denote the matrix obtained by projecting each row of $\bld{A}_h$ onto its corresponding row subspace $\mcal{V}_i$.
Algorithm~\ref{alg:amselgen} constructs a matrix approximation $\bld{\tilde{A}}_h$ such that
\[
    \mathbb{E}[\| \bld{P}[\bld{A}_h] - \tilde{\bld{A}}_h\|_F^2] 
        \leq   \frac{k}{q-k-1} \| \bld{A}_h - \bld{P}[\bld{A}_h]\|_F^2. 
\]
By the row-space estimate, 
\[
    \| \bld{A}_h - \bld{P}[\bld{A}_h]\|_F^2 = \sum_{i=1}^{|\mcal{T}_h|} \|  a_i - \bld{P}_i a_i\|_2^2 \leq \varepsilon^{\prime^2}  \text{max}\{K_{W,\text{adm}},K_{W,\text{non}}\} N_{cc} \| \kappa\|_{L^2(\Omega \times \Omega)}^2.
\]
Furthermore, row-wise orthogonality gives
\begin{align}
   \mathbb{E}[\| \bld{A}_h - \bld{\tilde{A}}_h\|_F^2] 
        & = \| \bld{A}_h - \bld{P} [\bld{A}_h]\|_F^2 + \mathbb{E}[\| \bld{P} [\bld{A}_h] - \bld{\tilde{A}}_h\|_F^2] \nonumber \\
        & \leq \left( 1 + \frac{k}{q-k-1}\right) \varepsilon^{\prime^2}  \text{max}\{K_{W,\text{adm}},K_{W,\text{non}}\} N_{cc} \| \mcal{A}\|_{HS}^2 . \nonumber
\end{align}
Using $q \geq 2k + 1$ queries, we bound $\| \bld{A}_h - \bld{\tilde{A}}_h\|_F $ with a probabilistic bound leveraging Markov's inequality to yield
\begin{align} 
    \mathbb{P}( \| \bld{A}_h - \bld{\tilde{A}}_h\|_F 
        & \geq \frac{\varepsilon}{2}\| \mcal{A}\|_{HS:L^2 \rightarrow L^2}) \nonumber \\
        & = \mathbb{P}(\| \bld{A}_h - \bld{\tilde{A}}_h\|_F^2 
            \geq \left( \frac{\varepsilon}{2}\right)^2 \| \mcal{A}\|_{HS:L^2 \rightarrow L^2}^2) \nonumber \\
        & \leq \frac{\mathbb{E}[\| \bld{A}_h - \bld{\tilde{A}}_h\|_F^2] }{\left(\frac{\varepsilon}{2} \right)^2 \|\mcal{A}\|^2_{HS}} \nonumber \\
        & \leq \frac{2 \varepsilon^{\prime^2}  \text{max}\{K_{W,\text{adm}},K_{W,\text{non}}\} N_{cc}\|\mcal{A}\|^2_{HS}}{\left(\frac{\varepsilon}{2} \right)^2 \|\mcal{A}\|^2_{HS}} \nonumber \\
        & \leq \delta \nonumber
\end{align}
where the final inequality follows from the choice of $\varepsilon^\prime = \frac{\sqrt{\delta}\varepsilon}{2\sqrt{2\text{max}\{K_{W,\text{adm}},K_{W,\text{non}}\} N_{cc}}}$.
Therefore, with probability at least $1 - \delta$,
\[
    \| \bld{A}_h - \bld{\tilde{A}}_h\|_F 
        \leq \frac{\varepsilon}{2}\| \mcal{A}\|_{HS:L^2 \rightarrow L^2}.
\]
Since $J_h$ is an isometry, 
\[
    \|P_h \mcal{A} P_h - \tilde{\mcal{A}}\|_{HS:L^2 \rightarrow L^2} = \|\bld{A}_h - \tilde{\bld{A}}_h\|_F
\]
and hence
\[
   \|P_h \mcal{A} P_h - \tilde{\mcal{A}}\|_{HS:L^2 \rightarrow L^2} \leq  \frac{\varepsilon}{2}\| \mcal{A}\|_{HS:L^2 \rightarrow L^2}.
\]
On the other hand, Lemma ~\ref{FEdiscretizationlemma} and $ch^\rho < \varepsilon/2$ give
\[
    \|\mcal{A} - P_h \mcal{A}P_h\|_{HS:L^2 \rightarrow L^2} \leq \frac{\varepsilon}{2}\left( \| \mcal{A}\|_{HS:L^2 \rightarrow H^\rho} +  \| \mcal{A}^*\|_{HS:L^2 \rightarrow H^\rho} \right).
\]
Combining these two estimates by triangle inequality gives, with probability at least $1 - \delta$, 
\[
    \| \mcal{A} - \tilde{\mcal{A}}\|_{HS:L^2 \rightarrow L^2}         \leq \varepsilon \left( \| \mcal{A}\|_{HS:L^2 \rightarrow H^\rho} +  \| \mcal{A}^*\|_{HS:L^2 \rightarrow H^\rho} \right)
\]
where we use that $\| \mcal{A} \|_{HS:L^2 \rightarrow L^2} \leq \| \mcal{A} \|_{HS:L^2 \rightarrow H^\rho}$.
\end{proof}
\section{Concluding Remarks and Outlook}
\begin{wrapfigure}{r}{0.32\textwidth}
    \centering
    \includegraphics[width=0.95\linewidth]{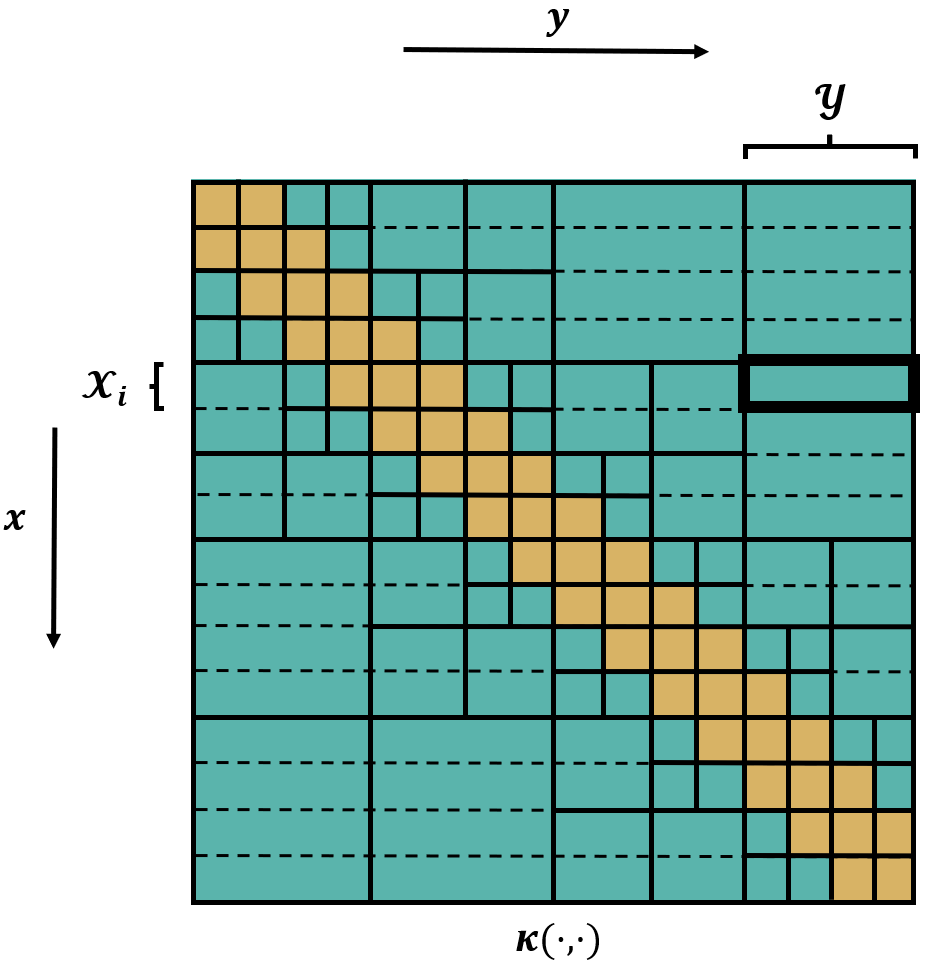}
    \caption{A polynomial basis for an outer admissible block $\mcal{X}_i \times \mcal{Y}$ of the $i$-th row is shared across multiple rows. }
\label{fig:shared bases}
\end{wrapfigure}

In this work, we presented a randomized matrix approximation algorithm for approximating matrices with pre-defined row-subspace structure and used it to construct an adjoint-free approximation algorithm for asymptotically smooth integral operators. 
We established that the resulting algorithm achieves Hilbert-Schmidt error guarantees with a polylogarithmic query complexity in the inverse tolerance, or equivalently a stretched exponential decay in the query budget. 
We also demonstrated this query-efficient behavior numerically for one-dimensional non-self-adjoint elliptic solution operators. 
Together, these results show that adjoint-free operator approximation can be query-efficient when sufficient a priori structure is available to construct accurate low-dimensional approximations of the local row spaces.

This work opens up several practical and theoretical avenues to pursue. 
We briefly comment on a few. 
As currently stated, the row subspaces construction and row-wise least-squares problems of Algorithm~\ref{adjointfreealg} can be implemented in parallel in practice, but it could be further optimized for computational performance. 
As shown in Figure \ref{fig:shared bases}, polynomial bases for the outer admissible domains are, in fact, shared across many row subspaces, leading to redundant floating point operations and storage. Instead of proceeding row-wise in parallel, the algorithm could be restructured to proceed in parallel across the levels of a standard hierarchical partition inwards toward the diagonal. 

Moreover, the practitioner will most likely not be privy to the exact constants governing the regime where an exponential convergence rate is possible, so they are likely to proceed by increasing the number of queries and decreasing the mesh size according to the scalings suggested by Corollary~\ref{corquerybudget} until the desired tolerance is reached. 
However, the algorithm is rather limited by the fact that Gaussian queries of a given resolution size to the underlying operator cannot be reused in the setting of a finer or coarser mesh since Gaussian random vectors of different sizes are not governed by the same law. 
It would be of great value if one could reuse queries when trying to efficiently ascertain a reasonable mesh size for the discretization.
Furthermore, for applications where arbitrary Gaussian inputs may be impractical or physically inadmissible, it would be useful to investigate alternative admissible query models and characterize when they still permit query-efficient approximation.

The numerical experiments further indicate that the difficulty of the underlying PDE is reflected in the convergence of the approximation. 
In particular, stronger advection and poorer ellipticity lead to slower observed convergence. 
This suggests studying whether challenging operators can be preconditioned into a form with a more favorable local approximation structure. 
Such preconditioning may need to exploit known information about the coefficients and could lead to more effective adjoint free approximation for difficult problems. 

On the theoretical side, asymptotic smoothness is simply one sufficient assumption for producing the local row-space information required by the approximation algorithm, rather than an essential feature of the framework.
For elliptic PDEs, the $L^2$-asymptotic smoothness assumptions are closely related to interior regularity and Caccioppoli-type estimates for PDE-harmonic functions. 
One natural direction is therefore to establish asymptotic smoothness for broader classes of non-self-adjoint and variable-coefficient elliptic solution operators.
More generally, it would be also be valuable to identify kernel structures arising naturally for other classes of operators that likewise yield accurate low-dimensional local row approximation spaces. 

\section*{Data Availability}
Code and data supporting this manuscript are publicly available on GitHub at \url{https://github.com/egallmeier/1DAdvectionDiffusionEqn.git}. 

\section*{Acknowledgements}
We thank Nicolas Boullé, Diana Halikias and Alex Townsend for inspiration and insightful discussions. 

\addcontentsline{toc}{section}{References}

\printbibliography

\appendix
\section{Appendix}

\subsection{Proof of Corollary~\ref{corquerybudget}}
\label{adxcorquerybudget}
\begin{C1}
Assume the setting of Theorem~\ref{bigtheorem} with $\eta < \min\big\{\frac{1}{\gamma d}, 1\big\}$. There exist constants $\tilde{\varepsilon}_0, \tilde{K}_0,\tilde{K}^\prime_0$ such  that for all $\tilde{K} \geq \tilde{K}_0, \tilde{K}^\prime \geq \tilde{K}_0^\prime$, $\delta \in (0,1)$ and $\rho \in (0,1)$ with $\rho < s - \frac{d}{2}$, using a number of Gaussian queries 
\[
    q > q_0
        := \frac{\tilde{K}^\prime \tilde{K}^d}{\rho} \left[ \max \Big\{\log{(\tilde{\varepsilon}_0^{-1})}, \frac{1}{2}\log{(\delta^{-1})}\Big\} \right]^{d+1}
\] 
with a finite element mesh size
\[
    h = \text{max}\Bigg\{h' \in \mcal{H}: h' \leq e^{-\frac{2}{\rho}\left(\frac{ q \rho}{\tilde{K}^\prime\tilde{K}^d} \right)^\frac{1}{d+1}}\Bigg\}
\]
and a polynomial degree of $m_0$ for each admissible cube restricted to a convex component of $\Omega$ 
\[
    m_0 = \min\Bigg\{m \in \mathbb{N}: m \geq (\tilde{K}-1)\left( \frac{q \rho}{\tilde{K}^\prime\tilde{K}^d} \right)^\frac{1}{d+1} \Bigg\}, 
\]
Algorithm~\ref{adjointfreealg} constructs a finite rank approximation $\tilde{\mcal{A}} = Alg(f \mapsto \mcal{A}f, q, m_0, h, \eta)$ such that 
\[
    \mathbb{P}\left( \| \mcal{A} - \tilde{\mcal{A}}\|_{HS:L^2 \rightarrow L^2}
        \leq e^{- \left( \frac{q \rho}{\tilde{K}^\prime\tilde{K}^d} \right)^\frac{1}{d+1}} \left(\| \mcal{A}\|_{HS:L^2 \rightarrow H^\rho} + \| \mcal{A}^*\|_{HS:L^2 \rightarrow H^\rho} \right) \right) 
        \geq 1 - \delta. 
\]
\end{C1}

\begin{proof}
Let $\varepsilon_0,K_0,K_0^\prime$ be the constants from Theorem~\ref{bigtheorem} where $\varepsilon_0 = \text{min}\big\{\frac{1}{2c}, \frac{1}{e},\sqrt{h_1}\big\}$.
Set $\tilde{\varepsilon}_0 := \varepsilon_0$, $\tilde{K}_0 := 2K_0 + 1$ and $\tilde{K}^\prime_0 := K^\prime_0$. 

Let $q > q_0$.
Define $\varepsilon_q := e^{-\left( \frac{q \rho}{\tilde{K}^\prime\tilde{K}^d} \right)^\frac{1}{d+1}}$. 
Define the finite element mesh size $h$ and the polynomial degree parameter $m_0$ as stated above. 
We show that these parameters satisfy the conditions required by Theorem~\ref{bigtheorem}. 
First, the inequality $q > q_0$ and the definition of $\tilde{\varepsilon}_0$ ensure that $\varepsilon_q < \varepsilon_0$. 
Now we show that the polynomial degree parameter satisfies $m_0 \geq K_0 \log{((\varepsilon_q \sqrt{\delta})^{-1})}$. 
Using $\varepsilon_q < \varepsilon_0$ and the inequality $q > q_0$, observe that
\begin{align}
    K_0 \log{((\varepsilon_q \sqrt{\delta})^{-1})} & \leq K_0 \Big[\log{\varepsilon_q^{-1}} + \frac{1}{2}\log{\delta^{-1}}\Big] \nonumber \\
    & \leq 2 K_0 \left( \frac{q \rho}{\tilde{K}^\prime\tilde{K}^d} \right)^\frac{1}{d+1} = (\tilde{K}_0 -1) \left( \frac{q \rho}{\tilde{K}^\prime\tilde{K}^d} \right)^\frac{1}{d+1} \nonumber  \\
    & \leq (\tilde{K}-1)\left( \frac{q \rho}{\tilde{K}^\prime\tilde{K}^d} \right)^\frac{1}{d+1} \leq m_0. \nonumber
\end{align}
We next show that the number of queries satisfies $q \geq \frac{K_0^\prime}{\rho}\log{\varepsilon_q^{-1}}m_0^d$. 
The polynomial degree parameter $m_0$ can be upper bounded as follows
\[
    m_0 \leq (\tilde{K}-1)\left( \frac{q \rho}{\tilde{K}' \tilde{K}^d} \right)^\frac{1}{d+1} + 1 \leq \tilde{K} \left( \frac{q \rho}{\tilde{K}' \tilde{K}^d} \right)^\frac{1}{d+1}
\]
using that $\left( \frac{q \rho}{\tilde{K}' \tilde{K}^d} \right)^\frac{1}{d+1} \geq \log{\varepsilon_q^{-1}} \geq 1$.
Then,
\[ 
    \frac{K^\prime_0}{\rho} \log{(\varepsilon_q^{-1})} m_0^d \leq \frac{\tilde{K}^\prime}{\rho} \left( \frac{q \rho}{\tilde{K}' \tilde{K}^d} \right)^\frac{1}{d+1} m_0^d  \leq \frac{\tilde{K}'\tilde{K}^d}{\rho}\frac{q \rho}{\tilde{K}' \tilde{K}^d} = q.
\]
Furthermore, since $h_n \rightarrow 0$ and is strictly decreasing, the set $\Big\{h' \in \mcal{H}: h' \leq \varepsilon_q^\frac{2}{\rho}\Big\}$ has a maximum.  
Therefore, all hypotheses of Theorem~\ref{bigtheorem} are satisfied, and applying it implies that Algorithm~\ref{adjointfreealg}  constructs an $\varepsilon_q$-approximation $\tilde{\mcal{A}}$ in probability at least $1 - \delta$.
\end{proof}

\subsection{Finite Element Discretization}
For a mesh $\mcal{T}_h$, let $h_\mcal{X}$ denote the diameter of a cell $\mcal{X} \in \mcal{T}_h$, and set $h:= \max_{\mcal{X} \in \mcal{T}_h} h_\mcal{X}$. 

\begin{lemma}
Let $(\mcal{T}_h)_{h \in \mcal{H}}$ be a quasi-uniform mesh sequence. 
Then, the number of mesh cells $\mcal{X} \in \mcal{T}_h$ satisfies 
\[
    |\mcal{T}_h| = \mcal{O}(h^{-d}). 
\]
\label{numberofmeshcells}
\end{lemma}

\begin{proof}
Denote $\rho_\mcal{X}$ as the diameter of the largest ball that can be inscribed in $\mcal{X} \in \mcal{T}_h$. 
Therefore,
\begin{align}
    \mathrm{vol}(\mcal{X}) 
        & \geq \omega_d \left(\frac{\rho_\mcal{X}}{2}\right)^d  
            \geq \omega_d \left(\frac{h_\mcal{X}}{2\sigma_{\#}}\right)^d 
            \geq \omega_d \left(\frac{c_\text{quasi}h}{2\sigma_{\#}}\right)^d \nonumber
\end{align}
using shape regularity in the second inequality and quasi-uniformity in the last inequality. 
Let $K_\text{mesh} := \omega_d \left( \frac{c_\text{quasi}}{2 \sigma_\#} \right)^d$. Every mesh cell has volume at least $K_\text{mesh} h^d$. 
Since $\overline{\Omega} = \bigcup_{\mcal{X} \in \mcal{T}_h} \mcal{X}$, 
\[
    \mathrm{vol}(\Omega) 
        \geq |\mcal{T}_h| K_\text{mesh} h^d
\]
then
\[
    |\mcal{T}_h| 
        \leq \frac{\mathrm{vol}(\Omega)}{K_\text{mesh}} h^{-d}. 
\] 
\end{proof}
\subsection{Reference Discretization \& Error Measurement for Section~\ref{sec:numericalexamples}}
\label{appendix:numerics}

This appendix details the reference discretizations used in Figures~\ref{fig:constant-numerics} and \ref{fig:numerical_experiments} and explains how the reported errors relate to the continuous Hilbert-Schmidt error.
Let $\mcal{V}_{h_\text{ref}}^0$ denote the piecewise constant space on the reference mesh $\mcal{T}_{h_\text{ref}}$, equipped with its basis.
Let $J_{\text{ref},0}$ denote the corresponding prolongation operator. 
Since the basis is $L^2$-orthonormal, $J_{\text{ref},0}$ is an isometry, and $P_{\text{ref},0} = J_{\text{ref},0} J_{\text{ref},0}^\star$ is the $L^2$-orthogonal projection onto $\mcal{V}_{h_\text{ref}}^0$.
Define the exact projected operator by $\mcal{A}_\text{ref}^\text{proj} = P_{\text{ref},0} \mcal{A} P_{\text{ref},0}.$ For the constant-coefficient problem, we construct the reference matrix directly from the known Green's function
\[
    \bld{A}_{\text{ref}_{i,j}} = \frac{1}{\sqrt{|\mcal{X}_{\text{ref},i}||\mcal{X}_{\text{ref},j}|}} \int_{\mcal{X}_{\text{ref}_i}} \int_{\mcal{X}_{\text{ref}_j}} G(x,y) dy dx. 
\]
With exact evaluation of these integrals, this matrix represents $\mcal{A}_\text{ref}^\text{proj}$ in the normalized reference basis. 

For the variable-coefficient problems, we approximate the solution operator using continuous piecewise linear finite elements on the reference mesh. 
Let $\mcal{V}_{h_\text{ref}}^1 \subset H^1_0(0,1)$ denote this space, its prolongation operator 
$J_\text{ref,1}$ in the continuous piecewise linear basis, and $\bld{K}_\text{ref}$ the corresponding Galerkin matrix associated with the PDE bilinear form. 
The reference matrix is
\[
    \bld{A}_\text{ref} = J_{\text{ref},0}^\star J_{\text{ref},1} \bld{K}_\text{ref}^{-1} J_{\text{ref},1}^\star J_{\text{ref},0}.
\]
For both the constant and variable-coefficient problems, the corresponding reference operator on $L^2(0,1)$ is $\mcal{A}_\text{ref} = J_{\text{ref},0} \bld{A}_\text{ref} J_{\text{ref},0}^\star.$ 
In particular, $\mcal{A}_\text{ref} = \mcal{A}_\text{ref}^\text{proj}$ for the constant-coefficient problem. 

Algorithm~\ref{adjointfreealg} operates on a coarser mesh $\mcal{T}_h$, with normalized piecewise constant space $\mcal{V}_h^0$ and prolongation operator $J_h$.  
For simplicity, the reference mesh is chosen so that $\mcal{V}_h^0 \subset \mcal{V}_{h_\text{ref}}^0$
for every coarse mesh size $h$ used in the experiments. 
Consequently, the coarse forcing can be represented exactly on the fine reference mesh without introducing an additional interpolation error. 
For $\bld{\omega}\sim \mcal{N}(\bld{0},\bld{I}_{|\mcal{T}_h|})$, a query uses the forcing $J_h \bld{\omega}$ and returns 
\[
    \bld{y} = J_h^\star \mcal{A}_\text{ref} J_h \bld{\omega} = J_h^\star J_{\text{ref},0} \bld{A}_\text{ref} J_{\text{ref},0}^* J_h \bld{\omega}.
\]
The algorithm accesses the reference solution operator only by applying these Gaussian piecewise constant forcings and projects the resulting solution back onto $\mcal{V}_h^0$. 
Ultimately, the algorithm returns a matrix approximation $\tilde{\bld{A}}_h$, defining the approximated operator $\tilde{\mcal{A}} = J_h \tilde{\bld{A}}_h J_h^\star$.

By triangle inequality, the absolute continuous error can be bounded as follows
\[
    \|\mcal{A} - \tilde{\mcal{A}}\|_{HS}  \leq   \underbrace{\|\mcal{A} - \mcal{A}_\text{ref}^\text{proj}\|_{HS}}_{\text{reference projection error}} + \underbrace{\|\mcal{A}_\text{ref}^\text{proj} - \mcal{A}_\text{ref}\|_{HS}}_{\text{reference  discretization error}} + \underbrace{\|\mcal{A}_\text{ref} - \tilde{\mcal{A}}\|_{HS}}_{\text{measured approximation error}}.
\]
Nesting of the coarse and reference meshes and orthonormality of the reference basis imply that the reported relative error satisfies
\[
    \frac{\|\bld{A}_\text{ref} - \tilde{\bld{A}}_\text{ref}\|_F}{\|\bld{A}_\text{ref}\|_F} = \frac{\|\mcal{A}_\text{ref} - \tilde{\mcal{A}}\|_{HS}}{\|\mcal{A}_\text{ref}\|_{HS}}.
\]
To compare the relative reported error with the continuous relative error, define
\[
    \varepsilon_\text{ref}: = \frac{\|\mcal{A} - \mcal{A}_\text{ref}^\text{proj}\|_{HS} + \|\mcal{A}_\text{ref}^\text{proj} - \mcal{A}_\text{ref}\|_{HS}}{\|\mcal{A}_\text{ref}\|_{HS}}.
\]
Reverse triangle inequality gives $\|\mcal{A}\|_{HS} \geq (1 - \varepsilon_\text{ref}) \|\mcal{A}_\text{ref}\|_{HS}.$ Consequently, provided that $\varepsilon_\text{ref} < 1$, 
\[
    \frac{\|\mcal{A} - \tilde{\mcal{A}}\|_{HS}}{\|\mcal{A}\|_{HS}} \leq \frac{ \frac{\|\mcal{A}_\text{ref} - \tilde{\mcal{A}}\|_{HS}}{\|\mcal{A}_\text{ref}\|_{HS}} + \varepsilon_\text{ref}}{1 - \varepsilon_\text{ref}}
\]
For the constant-coefficient experiment, the reference discretization error vanishes when the Green's function integrals are evaluated exactly. 
The bound shows that the continuous relative error is controlled by the reported relative error together with the reference projection and discretization errors. 
We use $|\mathcal{T}_{h_\text{ref}}| = 16384$ elements for the constant-coefficient experiment and $|\mathcal{T}_{h_\text{ref}}| = 8192$ elements for the variable-coefficient experiments. 
In each experiment, the reference mesh has at least four times as many elements as the finest mesh used by the algorithm. 
Further refinement of the reference mesh does not considerably change the reported errors. 
We therefore report $\frac{\|\bld{A}_\text{ref} - \tilde{\bld{A}}_\text{ref}\|_F}{\|\bld{A}_\text{ref}\|_F}$
as a numerical proxy for the continuous relative Hilbert-Schmidt error appearing in the analysis.

\end{document}